\documentclass{mcom-l}

\usepackage{lipsum}
\usepackage{amsfonts}
\usepackage{algorithmic}
\usepackage{amsmath}
\usepackage{bm}
\usepackage{amssymb}
\usepackage{mathrsfs}
\usepackage{xcolor}
\usepackage{color}
\usepackage{url}
\usepackage{cases}
\usepackage{float}
\usepackage{booktabs}
\usepackage{anyfontsize}
\usepackage{newtxtext,newtxmath}
\usepackage{subcaption}

\usepackage{graphicx} % 插入图片的核心包

\usepackage[colorlinks,
            linkcolor=red,
            anchorcolor=green,
            citecolor=blue
            ]{hyperref} 
\AtBeginDocument{
  \label{CorrectFirstPageLabel}
  
}

\newtheorem{theorem}{Theorem}[section]
\newtheorem{lemma}{Lemma}[section]
\newtheorem{hypothesis}{Hypothesis}[section]

\newtheorem{remark}{Remark}[section]

\usepackage{amsopn}

\newcommand{\R}{\mathbb{R}}
\newcommand{\N}{\mathbb{N}}

\usepackage{hyperref}

\numberwithin{equation}{section}

\begin{document}

\title[Optimal error analysis for high-order Navier-Stokes equations]%{Error analysis of high-order IMEX fully discrete finite element schemes for the incompressible Navier-Stokes  system}
%{Analysis of fully discrete IMEX finite element schemes of orders 1-6 for the incompressible Navier-Stokes  system}
{Stability and  error analysis of  IMEX-BDF$k$ finite element schemes for  the incompressible Navier-Stokes  system}

%    author one information
\author{Qianqian Ding}
\address{School of Mathematics and Statistics, Beijing Jiaotong University, Beijing 100044, China}
\curraddr{}
\email{dingqianqian@lsec.cc.ac.cn}
\thanks{}

%    author two information
\author{Yifan Luo}
\address{State Key Laboratory of Mathematical Sciences, Academy of Mathematics and Systems Science, Chinese Academy of Sciences, Beijing 100190, China\\
 School of Mathematical Sciences, University of Chinese Academy of Sciences, Beijing 100049, China}
\curraddr{}
\email{luoyifan24@mails.ucas.ac.cn}
\thanks{}

% author three information
\author{Shipeng Mao}
\address{%Corresponding author.
	 State Key Laboratory of Mathematical Sciences, Academy of Mathematics and Systems Science, Chinese Academy of Sciences, Beijing 100190, China\\
 School of Mathematical Sciences, University of Chinese Academy of Sciences, Beijing 100049, China}
\curraddr{}
\email{maosp@lsec.cc.ac.cn}
\thanks{S. Mao was supported by the National Key Research and Development Program of China (2024YFA1012502),
and the National Natural Science Foundation of China (No. 12271514).}

%    \subjclass is required.

% \keywords is required.

\date{}

\dedicatory{}

%    Abstract is required.
%\begin{abstract}
%In this paper, we propose and analyze a class of high-order numerical schemes within a fully discrete finite element framework for the incompressible Navier-Stokes equations with no-slip boundary conditions.  The temporal discretization employs a $k$th-order ($k=1,\cdots,6$) implicit-explicit backward difference  formula (IMEX‑BDF$k$), in which the nonlinear convection term is treated explicitly and the linear Stokes part implicitly, whereas the spatial discretization utilizes  Taylor-Hood finite elements.  We establish the stability and uniform boundedness of the numerical solution.  We further establish optimal order error estimates in both space and time without any CFL-type condition, in the sense that the time step is independent of the spatial mesh size. In three dimensions, these include  $L^2$- and $H^1$-norm error estimates for the velocity and  $L^2$-norm error estimates for the pressure, with temporal convergence rates up to sixth order for all variables. Numerical experiments are presented to demonstrate the effectiveness of the scheme and to confirm the theoretical convergence rates.
%\end{abstract}

\begin{abstract}
In this paper, we propose and analyze a class  of high-order numerical schemes within a fully discrete finite element framework for the incompressible Navier-Stokes equations with no-slip boundary conditions.  The temporal discretization employs a $k$th-order ($k=1,\cdots,6$) implicit-explicit backward difference  formula (IMEX‑BDF$k$), in which the nonlinear convection term is treated explicitly and the linear Stokes part implicitly, whereas the spatial discretization utilizes  Taylor-Hood finite elements.  Under suitable regularity assumptions, we establish stability and optimal  error estimates  in both space and time without any CFL-type condition, in the sense that the time step is independent of the spatial mesh size. In three dimensions, these include  $L^2$- and $H^1$-norm error estimates for the velocity and  $L^2$-norm error estimates for the pressure, with temporal convergence rates up to sixth order for all variables. The analysis of the sixth-order scheme combines a tailored BDF6 multiplier argument with a coercivity estimate for the discrete time derivative obtained after summation over the time levels.  Numerical experiments are presented to demonstrate the effectiveness of the schemes and to confirm the theoretical convergence rates.
\end{abstract}

\maketitle

\thispagestyle{empty}   % 只清空当前页的页眉页脚

\section{Introduction}

This work develops and analyzes a class of high-order fully discrete finite element methods for the incompressible Navier-Stokes equations based on $k$th-order backward differentiation formula (BDF$k$), with temporal orders $k=1,\cdots,6$. The governing equations are posed on a three-dimensional domain $\Omega\subset\mathbb{R}^3$ and  read
\begin{align}
&\frac{\partial\bm{u}}{\partial t}-\nu\Delta\bm{u}+(\bm{u}\cdot\nabla)\bm{u}+\nabla p=\bm{f}       \,\,\, \,\,\,\mbox{in  } \Omega\times (0,T],\label{eq:mhd1}\\[2mm]
&\mathrm{div}\, \bm{ u }=0  \qquad\qquad\qquad   \qquad \qquad\,\,\,   \mbox{in  } \Omega\times (0,T],\label{eq:mhd2}
\end{align}
where $T>0$ is a given finite final time, $\bm{u}$ denotes the velocity field, $p$ the pressure, $\nu$ the viscosity of the fluid, and $\bm{f}$ the external force. The system is supplemented with the initial condition and no‑slip boundary condition
 \begin{equation}
 \bm{u}(\bm{x},0)=\bm{u}^0 \quad \forall\, \bm{x}\in\Omega, \qquad\quad 
 \bm{u}=\bm{0} \quad \mbox{on  } \partial\Omega\times(0,T]. \label{eq:bdy2}
\end{equation}

Considerable effort has been devoted to the development of efficient and reliable numerical methods for the Navier-Stokes equations, with particular emphasis on temporal discretization strategies. These methods can be naturally classified according to their order of accuracy. First‑order methods, such as the semi-implicit Euler scheme, are widely used due to their simplicity, robustness, and unconditional stability \cite{Allendes2021sisc, HanYongbin2022ima, LiBuyang2022sinum, AnRong2022cnsns}. Implicit-explicit (IMEX)  first-order schemes have also attracted attention \cite{HanYongbin2023mc, HeYinnian2008mc, LiXiaoli2023anm}, as they treat the nonlinear terms explicitly while retaining an implicit treatment of the stiff linear components, thereby avoiding the solution of nonlinear systems at each time step and reducing computational costs.  Second‑order methods, including the Crank-Nicolson scheme and the BDF2 method, provide an improved balance between accuracy and stability and are among the most commonly employed approaches in finite element frameworks \cite{Bermejo2012sinum, Diegel2017, Archilla2023ima, Hesinum2003,  Ingramint2013, JiBingquan2025ima, Obbadi2025cmame, WangXiaomingnm2012}. However, even these first‑ and second‑order schemes may be insufficient for capturing complex transient phenomena, especially in problems involving multi‑scale dynamics such as turbulence and interfacial instabilities.

Motivated by these limitations, the development of high‑order temporal discretization techniques, such as third‑, fourth‑, and higher‑order BDF and linear multistep methods, has become an important research direction. Compared with low‑order schemes, high‑order methods can significantly improve temporal accuracy under the same spatial discretization, leading to faster convergence toward the exact solution. This advantage is particularly pronounced in problems that involve fine‑scale physical features, including small‑scale turbulence, interfacial instabilities, and vortex dynamics. When a prescribed accuracy is targeted, high‑order schemes permit larger time steps or coarser spatial meshes, thereby reducing the overall computational cost and substantially enhancing efficiency, a crucial benefit in large‑scale three‑dimensional simulations \cite{Hairer1993ode1, Hairer1996ode2}.

In recent years, several key advances have been  achieved in the analysis of high-order temporal discretization methods  for the Navier-Stokes equations. In \cite{ChengKelong2016}, uniform a priori bounds were established for semi-discrete high-order IMEX schemes of up to fourth-order in two-dimensional (2D) periodic domains. Later, \cite{HuangFukeng2021SIAM} proposed a unified analytical framework for IMEX‑BDF$k$ schemes coupled with Fourier-Galerkin spectral methods under periodic boundary conditions, proving $k$th-order convergence for $1 \leq k \leq 5$. This framework was further extended in \cite{JiBingquan2024jsc} through the introduction of discrete orthogonal convolution kernels, yielding optimal $L^2$-error estimates for the velocity under  condition $\tau \leq c N^{-2/k}$.  It should be noted that all these works rely on periodic boundary conditions or spectral methods, which limits their applicability to realistic problems involving complex geometries and essential physical constraints.

  However, many practical flow phenomena arise in domains with curved boundaries, corners, or complex geometries, where spectral or Fourier methods are not readily applicable.  Finite element methods, by contrast, offer great flexibility in handling such irregular domains and naturally accommodate a variety of boundary conditions, including the no‑slip condition.   Although the finite element method (FEM) is widely acknowledged as a highly effective tool for  simulating incompressible flows in complex geometries, rigorous numerical analysis of high‑order temporal discretization schemes within the FEM framework remains a challenging issue.   We begin with a brief review of representative results for the Stokes problem and related linear models. For transient Stokes equations,  \cite{Ahmed2017cmame} analyzed a discontinuous Galerkin time discretization combined with continuous finite elements in space, deriving error estimates by exploiting the commutativity between temporal and spatial interpolation operators. For Stokes equations posed on time-dependent domains,  \cite{LiuJie2013sinum} developed an arbitrary Lagrangian-Eulerian formulation employing BDF$k$ time discretizations ($1\le k\le5$) together with Taylor-Hood finite elements.  A major recent development concerns the sixth-order BDF method, which lies outside the classical multiplier framework of Nevanlinna and Odeh \cite{Nevanlinna1981}. Since BDF methods of order greater than six are no longer zero-stable (see, e.g., Chapter III of \cite{Hairer1993ode1}), BDF6 is the highest-order member of the classical BDF family suitable for practical computations. In particular, \cite{Akrivis2021sinum} established a stability theory for BDF6 applied to linear parabolic problems. This breakthrough was subsequently extended to the Stokes equations in \cite{Alessandro2025sinum}, where optimal convergence was proved for fully discrete finite element schemes of order up to six.

    For the nonlinear incompressible Navier-Stokes equations, error  analysis of high-order BDF finite element methods remains  relatively limited. To the best of our knowledge, the first rigorous error analysis of a third-order BDF finite element scheme was carried out by Baker et al. \cite{Bakermathcomp1982}, who established optimal $L^2$-error estimates for the velocity under the restriction $\tau\le ch^{2/3}$. More recently, \cite{Bosco2025arxiv} presented the first rigorous finite element error analysis for fourth- and fifth-order BDF discretizations of the incompressible Navier-Stokes equations. Their analysis established optimal error estimates for fully discrete BDF$k$ schemes ($k\le5$) under a CFL-type condition of the form $\tau\le Ch^{\alpha}$, with the nonlinear convection term treated fully implicitly.  In a related study, \cite{Qiuhe2026ima} proposed semi-implicit BDF finite element schemes of orders up to five for the unsteady Navier-Stokes-Darcy system and derived optimal $L^2$-error estimates within the fully discrete framework.  The convergence proof was established under a CFL-type restriction. Thus, despite the recent advances  for both linear problems and high-order Navier-Stokes discretizations, existing rigorous analyses of high-order BDF finite element methods for the incompressible Navier-Stokes equations are restricted to orders at most five and require CFL-type conditions. A rigorous stability and optimal error analysis for sixth-order finite element discretizations of the incompressible Navier-Stokes equations  remains unavailable.

  Motivated by the above discussion, we develop and analyze a class of fully discrete finite element approximations of the incompressible Navier-Stokes equations based on IMEX-BDF$k$ time discretizations of orders  $k=1,\cdots,6$. The nonlinear convection term is extrapolated explicitly, whereas the viscous and pressure terms are treated implicitly. This IMEX treatment of the nonlinearity yields, at each time step, a linear system with a constant coefficient matrix. For $k=1,\cdots,5$, our analysis combines the Nevanlinna-Odeh multiplier
framework \cite{Nevanlinna1981} with mixed finite element approximation and nonlinear stability estimates, yielding a unified proof of stability and optimal-order convergence. The sixth-order case requires a separate and substantially more delicate argument. Compared with the linear theories in \cite{Akrivis2021sinum,Alessandro2025sinum}, the principal difficulty is the interaction between the BDF6 multiplier and the explicitly extrapolated convection term.  In the optimal $H^1$-error analysis for BDF6, the discrete BDF6 derivative is retained as the Galerkin test function, and a suitable linear combination of four consecutive error equations is formed. The viscous contribution then produces a telescoping $G$-energy. By contrast, the corresponding discrete derivative contribution is not coercive at an individual time level. Its coercivity is recovered only after summation in time through a finite-interval Toeplitz estimate derived from the positivity of the associated multiplier symbol. 
The temporal consistency errors are decomposed into the BDF truncation error of the exact velocity and a discrete difference of the mixed Stokes projection error. %The first part is bounded by the temporal regularity of the velocity, while the second is controlled by standard approximation properties of the mixed Stokes projection applied to time derivatives of velocity and pressure. 
This decomposition avoids imposing unnecessarily strong temporal regularity assumptions on the pressure.  To handle the explicitly extrapolated convection term without loss of temporal accuracy, we further rewrite the  multistep velocity combination  as  a weighted sum of sixth-order temporal derivatives and derive  an identity  linking the current time derivative to the solution values at the current and six previous time levels. Together, these ingredients yield the full sixth-order temporal convergence rate. 
% In deriving the optimal $H^1$-error estimate, we retain the discrete BDF6 derivative as the Galerkin test function and combine four consecutive error equations. The viscous contribution then yields a telescoping $G$-energy, while coercivity of the discrete derivative contribution is recovered after summation in time by a finite-interval Toeplitz estimate. We further rewrite the  multistep velocity combination appearing in the extrapolated convection term as  a weighted sum of sixth-order temporal derivatives and derive  an identity  linking the current time derivative to the solution values at the current and six previous time levels.  Together, these ingredients preserve the full sixth-order temporal accuracy.  
 Combining the unified analysis for $k=1,\cdots,5$ with the dedicated sixth-order analysis, we obtain energy stability and optimal error estimates for all IMEX-BDF$k$ schemes with $k=1,\cdots,6$. More precisely, optimal convergence rates are obtained for the velocity in the $L^{\infty}(0,T;\bm{L}^2)$ and $L^{\infty}(0,T;\bm{H}^1)$ norms and for the pressure in the $L^{2}(0,T;L^2)$ norm, without any CFL-type condition in the sense
that the time step is independent of the spatial mesh size. The error bounds are optimal in both space and time, yielding full $k$th-order temporal convergence for every order considered.   Numerical experiments are presented to confirm the theoretical findings and to demonstrate the efficiency and accuracy of the proposed methods.

  %The remainder of this paper is organized as follows. Section \ref{section:setting} introduces  the necessary preliminaries, including notation, Sobolev spaces, the variational formulation and the discrete method for Navier-Stokes equations, and then states the main  results on unconditional stability and optimal error estimates.   Section \ref{section analysis} provides a detailed proof of the stability of the numerical solutions, followed by a rigorous derivation of the optimal error estimates for the velocity in the $L^2$- and $H^1$-norms and for the pressure in the $L^2$-norm.  Numerical experiments are given in Section \ref{section:numerical} to support the theoretical  findings. Finally,   concluding remarks are summarized  in Section \ref{section:summary}.

%The rest of this paper is organized as follows. Section \ref{section:setting} introduces some preliminary knowledge,  including symbols, Sobolev spaces,  and the weak formulation of Navier-Stokes equations. Section \ref{section:fem} develops a fully discrete numerical scheme that combines BDF-$k$ temporal discretization with finite element spatial approximation, and derives the corresponding error equations for the numerical solution for orders $k=1,\cdots,6$. Section \ref{section:error}  rigorously derives optimal error estimates of the  BDF-$k$ numerical scheme under  reasonable regularity assumptions on exact solutions. Numerical experiments are reported  in Section \ref{section:numerical} to support the theoretical  results, and concluding remarks are summarized  in Section \ref{section:summary}.

The remainder of this paper is organized as follows. Section \ref{section:setting} introduces  the necessary preliminaries, including notation, Sobolev spaces, the variational formulation and the discrete method for Navier-Stokes equations, and then states the main  results on stability and optimal error estimates.   Section \ref{section analysis} provides a detailed proof of the stability of the numerical solutions, followed by a rigorous derivation of the optimal error estimates for the velocity in the $L^2$- and $H^1$-norms and for the pressure in the $L^2$-norm.  Numerical experiments are given in Section \ref{section:numerical} to support the theoretical  findings. Finally,   concluding remarks are summarized  in Section \ref{section:summary}.

\section{Preliminaries}
\label{section:setting}
%\subsection{Preliminaries and notation}
\subsection{Notation and weak formulation}
We begin by introducing basic notation that will be used throughout the paper.  For any integer $m\in \N^+$, $1\leq p\leq\infty$, let $W^{m,p}(\Omega)$ denote the standard Sobolev space, which is denoted by $H^{m}(\Omega)$ when $p=2$. The associated norm in $W^{m,p}(\Omega)$ is denoted by $\Vert\cdot\Vert_{m,p}$. % such that 
%\begin{align*}
%\Vert v\Vert_{m,p}=&\,\Big(\sum_{\vert\alpha\vert\leq m}\Vert D^{\alpha}v\Vert_{0,p}^p\Big)^{1/p}\qquad\mathrm{for}\,\, 1\leq p<\infty,\\
%\Vert v\Vert_{m,\infty}=&\,\max_{\vert\alpha\vert\leq m}\Vert D^{\alpha}v\Vert_{0,\infty},
%\end{align*}
%where
%\begin{align*}
%\Vert v\Vert_{0,p}=&\,\Big(\int_{\Omega}\vert v\vert^p\,dx\Big)^{1/p}\quad \mathrm{for}\,\,1\leq p<+\infty,\quad \Vert v\Vert_{0,\infty}=&\,\operatorname*{ess\,sup}_{x\in \Omega}\vert v(x)\vert,
%\end{align*}
%and 
%\begin{align*}
%D^{\alpha}=\frac{\partial^{\vert\alpha\vert}}{\partial x_1^{\alpha_1}\partial x_2^{\alpha_2}\partial x_3^{\alpha_3}},
%\end{align*}
%for the multi-index $\alpha=(\alpha_1, \alpha_2, \alpha_3)$ and $\vert\alpha\vert=\alpha_1+\alpha_2+\alpha_3$, with $\alpha_1, \alpha_2, \alpha_3\ge0$. For the function spaces $L^p(0,T;X)$, $1\leq p\leq\infty$, the norms are denoted as
%\begin{align*}
%\Vert v\Vert_{L^p(0,T;X)}=&\,\Big(\int_0^T \Vert v(t)\Vert_X^p\,dt\Big)^{1/p}\qquad \mathrm{for}\,\,1\leq p<+\infty,\\
%\Vert v\Vert_{L^{\infty}(0,T;X)}=&\,\operatorname*{ess\,sup}_{0\leq t\leq T}\Vert v(t)\Vert_X,
%\end{align*}
%where $X$ is a real Banach space with the norm $\Vert\cdot\Vert_X$. 
The $L^2$ inner product is defined as  $(\phi, \psi)=\int_{\Omega}\phi\psi \,dx$, and the corresponding norm is denoted by $\Vert\cdot\Vert_{0}$. Vector-valued functions appear in boldface notation, such as $\bm{u}=(u_1, u_2, u_3)^{\top}$ and $\bm{L}^2(\Omega)=(L^2(\Omega))^3$.  The superscript  $\top$ denotes transposition. We use $C$ and $c$, with or without subscripts, to denote generic positive constants independent of the discretization parameters, which may take different values at different places. %We use $C$ and $c$, with or without subscripts, to denote generic positive constants independent of the discretization parameters, which may take different values at different places.

 For the problem  described by (\ref{eq:mhd1})-(\ref{eq:bdy2}), we introduce the following Sobolev spaces
\begin{align*}
 &\bm{X} = \bm{H}_0^1(\Omega),  \,\,\, \bm{X}_0 =\left \{\bm{v}\in \bm{X}, \mathrm{div}\,\bm{v}=0\right\},\,\,\, Q=\left\{q\in L^2(\Omega), \int_{\Omega}q(x)\,dx=0\right\}. 
%& \bm{Y}= \bm{L}^2(\Omega),\,\,\, \bm{H}=\{\bm{v}\in \bm{Y},\,\,\, \mathrm{div}\,\bm{v}=0, \quad \bm{v}\cdot\bm{n}\vert_{\partial\Omega}=0\}.
\end{align*}
%We denote the Stokes operator by $\mathcal{A}=-P_{0}\Delta$, see  \cite{Heywood1982}, where $P_{0}$ is the $L^2$-orthogonal projection of $\bm{Y}$ onto $\bm{H}$.
%Here and in what follows, we define the norm 
%\begin{equation*}
%\Vert\bm{C}\Vert_{\bm{W}}=\Vert\bm{C}\Vert_{\bm{H}(\mathbf{curl};\Omega)}=\Big(\Vert\bm{C}\Vert_0^2+\Vert\mathbf{curl}\,\bm{C}\Vert_0^2\Big)^{1/2}\qquad \forall \,\bm{C}\in\bm{W}.
%\end{equation*}

For the sake of convenience, we define the following notation
\begin{equation*}
\begin{split}
a(\bm{w}, \bm{u}, \bm{v})=((\bm{w}\cdot\nabla)\bm{u}, \bm{v}), \quad d(\bm{v}, p)=(\mathrm{div}\,\bm{v}, p).
\end{split}
\end{equation*}

The following scheme presents the weak formulation of (\ref{eq:mhd1})-(\ref{eq:bdy2}): find the solution $(\bm{u}, p)\in L^2(0,T;\bm{X})\cap L^{\infty}(0,T;\bm{L}^2(\Omega))\times L^{\infty}(0,T;Q)$ such that, for all $(\bm{v}, q)\in (\bm{X}\times Q)$
\begin{align}
&\left(\frac{\partial\bm{u}}{\partial t}, \bm{v}\right)+\nu(\nabla\bm{u}, \nabla\bm{v})+a(\bm{u}, \bm{u}, \bm{v})-d(\bm{v}, p)=(\bm{f}, \bm{v}) ,\label{weak:eq1}\\[2mm]
&d(\bm{u}, q)=0. \label{weak:eq2}
\end{align}

%Throughout this paper,  we impose the following assumptions on the problem  (\ref{weak:eq1})-(\ref{weak:eq2}), which  serve as the foundation for the subsequent analysis.
%\begin{hypothesis}\label{assum:right1}
%The initial data $\bm{u}^0$, with $l\ge1$, and the force $\bm{f}$ satisfy the following estimate
%\begin{equation*}
%\begin{split}
%&\sup\limits_{0\leq t\leq T}\{\Vert\bm{f}(t)\Vert_0\}\leq C_f, \qquad \Vert\bm{u}^0\Vert_{1+l,2}\leq M_0.
%\end{split}
%\end{equation*}
%\end{hypothesis}

We also frequently employ the following inequalities \cite{Temam1983}
 \begin{equation}\label{bestimate}
|a(\bm{w}, \bm{u}, \bm{v})|\leq
\begin{cases}
\begin{aligned}
&C\|\bm{w}\|_{1,2}\|\bm{u}\|_{1,2}\|\bm{v}\|_{1,2}.\\
& C\|\bm{w}\|_{2,2}\|\bm{u}\|_{0}\|\bm{v}\|_{1,2}.\\
&C\|\bm{w}\|_{2,2}\|\bm{u}\|_{1,2}\|\bm{v}\|_{0}.\\
&C\|\bm{w}\|_{1,2}\|\bm{u}\|_{2,2}\|\bm{v}\|_{0}.\\
&C\|\bm{w}\|_{0}\|\bm{u}\|_{2,2}\|\bm{v}\|_{1,2}.
\end{aligned}
\end{cases}
\end{equation}

\subsection{Fully discrete schemes}
The domain $\Omega$ is assumed to be convex polyhedral and partitioned into a mesh $\mathcal{T}_h$, which consists of tetrahedral elements $K$. The family of meshes $\{\mathcal{T}_h\}$ is assumed to be shape-regular and quasi-uniform. %Each tetrahedron K is supposed to be the image of a reference $\hat K$, the affine map is $F_K$.  Let $P_k(K)$ be the space of polynomials of total degree at most $k\ge0$ on $K$. 
To approximate the velocity-pressure pair $(\bm{u}, p)$, we employ generalized Taylor-Hood elements $(\bm{X}_h^{l}, Q_h^{l-1})$ with $l \ge 2$. Here,  $\bm{X}_h^{l}$ represents the $l$th-order vectorial Lagrange finite element subspace of $\bm{X}$, while $Q_h^{l-1}$ is the $(l-1)$th-order scalar Lagrange finite element subspace of $Q$, see \cite{Girault1986} for more details.  %In the lowest-order stable approximation, the MINI element pair is adopted for the velocity-pressure approximation. 
For the lowest-order approximation, we use the MINI element, which pairs continuous piecewise linear pressures with continuous piecewise linear velocities enriched by element bubble functions.

%The lowest-order case ($l=1$) is handled by the MINI element. %For the case $l=1$,   velocity-pressure is s approximated by the MINI elements. %When the case $k=0$, $(\bm{u}, p)$ is approximated by the mini-element, which still suitable for our theoretical analysis and numerical results in this paper, cf., for instance, \cite{MR1115205, MR2059447}.

Furthermore, the discrete kernel space of the divergence operator can be defined by
\begin{equation*}
\bm{X}_{0h}^{l}=\left\{\bm{v}_h\in \bm{X}_h^{l},\, d(\bm{v}_h, q_h)=0 \quad \forall\,q_h\in Q_h^{l-1}\right\}.
\end{equation*}
%it can be observed that $\bm{W}_{0h}^k\not\subset\bm{W}_h^k$.
By the Fortin criterion, the discrete inf-sup condition (see, e.g. Chapter 2 of \cite{Brezzi1991} or \cite{Hiptmair2002}) holds
\begin{equation}\label{infsup}
 \inf\limits_{0\ne q_h\in Q_h^{l-1}}\sup\limits_{\bm{0}\ne\bm{v}_h\in \bm{X}_h^{l}}\frac{(q_h, \mathrm{div}\,\bm{v}_h)}{\Vert\bm{v}_h\Vert_{1,2}\Vert q_h\Vert_0}\ge{\chi}^*,
\end{equation}
where $\chi^*$ denotes a generic positive constant dependent on the domain  $\Omega$.

%Recall the inverse estimate from Theorem 3.2.6 in \cite{Ciarlet1978}: on a quasi-uniform mesh $\mathcal{T}_h$,  the following inequality holds:
%\begin{eqnarray}
%\Vert\bm{v}_h\Vert_{m,q}\leq C_{inv}h^{\ell-m+3(1/q-1/p)}\Vert\bm{v}_h\Vert_{\ell,p}\qquad \forall \,\bm{v}_h\in\bm{X}_h^{l},\label{inverse} 
%\end{eqnarray}
%where $C_{inv}>0$ denotes a generic constant independent of the mesh-size $h$, $\ell$ and $m$ are two real numbers with $0\leq\ell\leq m\leq1$, $p$ and $q$ are two integers with $1 \leq p\leq q\leq\infty$. 

To describe the BDF$k$ time discretization, let $N$ be a positive integer and $0=t^0<t^1<\cdots<t^N=T$ a uniform partition of $[0, T]$ with time step size  $\tau =t^i-t^{i-1}, i=1,2,\cdots,N$, and denote by $\bm{u}^n=\bm{u}(t^n)$  the exact solution at time $t^n$. We consider the following fully discrete finite element approximation of (\ref{weak:eq1})-(\ref{weak:eq2}): given $\bm{u}_h^0, \bm{u}_h^1,\cdots, \bm{u}_h^{k-1}$, find  $(\bm{u}_h^{n+1}, p_h^{n+1})\in(\bm{X}_h^{l}\times Q_h^{l-1})$ such that for all $(\bm{v}_h, q_h)\in(\bm{X}_h^{l}\times Q_h^{l-1})$
\begin{align}
&\left(\frac{\alpha_k{\bm{u}}_h^{n+1}-\beta_k({\bm{u}}_h^{n})}{\tau }, \bm{v}_h\right)+\nu(\nabla\bm{u}_h^{n+1},\nabla \bm{v}_h)+a(\gamma_k(\bm{u}_h^n), \gamma_k(\bm{u}_h^n), \bm{v}_h)\nonumber\\[2mm]
&-d(\bm{v}_h, p_h^{n+1})=(\bm{f}^{n+1}, \bm{v}_h),  \label{weak space discrete:eq1}\\[2mm]
&d(\bm{u}_h^{n+1}, q_h)=0.  \label{weak space discrete:eq2}
\end{align}
%The initial values are defined by $\bm{u}_h^{\ell}=L_{h}\bm{u}^{\ell}$, with $\ell=0,1,\cdots, k-1$, where $L_{h}:\bm{L}^2(\Omega)\to\bm{X}_{0h}^{l}$ denotes  the $\bm{L}^2$-orthogonal projection. The following standard estimate holds  \cite{AitOuAmmi1994, Heywood1982},
%\begin{equation}\label{initial value estimate}
%\begin{split}
%&\Vert\bm{u}^{\ell}-\bm{u}_h^{\ell}\Vert_0\leq Ch^{l+1}\Vert\bm{u}^{\ell}\Vert_{l+1,2}.
%\end{split}
%\end{equation}
%where $s>1/2$ depends on the regularity of the domain $\Omega$.
Since the extrapolated velocity $\gamma_k(\bm{u}_h^n)$ is known, each time step requires only the solution of a linear Stokes-type problem. In the above formulation,   the coefficients $\alpha_k$ and the operators $\beta_k$  and $\gamma_k$ for $k=1,\cdots,6$, see  \cite{Hairer1993ode1, Hairer1996ode2}, are given by: 
first-order,
\begin{align}
&\alpha_1=1, \qquad \beta_1({\bm{u}}_h^n)={\bm{u}}_h^n, \qquad \gamma_1(\bm{u}_h^n)=\bm{u}_h^n,\label{k1}
\end{align}
second-order,
\begin{align}
&\alpha_2=\frac{3}{2}, \quad \beta_2({\bm{u}}_h^n)=2{\bm{u}}_h^n-\frac{1}{2}{\bm{u}}_h^{n-1}, \quad \gamma_2(\bm{u}_h^n)=2\bm{u}_h^n-\bm{u}_h^{n-1},\label{k2}
\end{align}
third-order,
\begin{align}
\alpha_3=\frac{11}{6}, \quad& \beta_3({\bm{u}}_h^n)=3{\bm{u}}_h^n-\frac{3}{2}{\bm{u}}_h^{n-1}+\frac{1}{3}{\bm{u}}_h^{n-2}, \nonumber\\[2mm]
&\gamma_3(\bm{u}_h^n)=3\bm{u}_h^n-3\bm{u}_h^{n-1}+\bm{u}_h^{n-2},\label{k3}
\end{align}
fourth-order,
\begin{align}
\alpha_4=\frac{25}{12}, \quad  &\beta_4({\bm{u}}_h^n)=4{\bm{u}}_h^n-3{\bm{u}}_h^{n-1}+\frac{4}{3}{\bm{u}}_h^{n-2}-\frac{1}{4}{\bm{u}}_h^{n-3}, \nonumber\\[2mm]
& \gamma_4(\bm{u}_h^n)=4\bm{u}_h^n-6\bm{u}_h^{n-1}+4\bm{u}_h^{n-2}-\bm{u}_h^{n-3},\label{k4}
\end{align}
fifth-order,
\begin{align}
\alpha_5=\frac{137}{60}, \quad&  \beta_5({\bm{u}}_h^n)=5{\bm{u}}_h^n-5{\bm{u}}_h^{n-1}+\frac{10}{3}{\bm{u}}_h^{n-2}-\frac{5}{4}{\bm{u}}_h^{n-3}+\frac{1}{5}{\bm{u}}_h^{n-4},\nonumber \\[2mm]
&\gamma_5(\bm{u}_h^n)=5\bm{u}_h^n-10\bm{u}_h^{n-1}+10\bm{u}_h^{n-2}-5\bm{u}_h^{n-3}+\bm{u}_h^{n-4}, \label{k5}
\end{align}
sixth-order,
\begin{align}
%\alpha_6=\frac{882}{120}=\frac{441}{60},  \quad& \beta_6({\bm{u}}_h^n)=6{\bm{u}}_h^n-\frac{15}{2}{\bm{u}}_h^{n-1}+\frac{19}{3}{\bm{u}}_h^{n-2}-\frac{13}{4}{\bm{u}}_h^{n-3}+\frac{6}{5}{\bm{u}}_h^{n-4}-\frac{1}{6}{\bm{u}}_h^{n-5},\nonumber  \\
%& \gamma_6(\bm{u}_h^n)=6\bm{u}_h^n-15\bm{u}_h^{n-1}+19\bm{u}_h^{n-2}-13\bm{u}_h^{n-3}+6\bm{u}_h^{n-4}-\bm{u}_h^{n-5}.\label{k6}
\alpha_6=\frac{147}{60},  \quad& \beta_6({\bm{u}}_h^n)=6{\bm{u}}_h^n-\frac{15}{2}{\bm{u}}_h^{n-1}+\frac{20}{3}{\bm{u}}_h^{n-2}-\frac{15}{4}{\bm{u}}_h^{n-3}+\frac{6}{5}{\bm{u}}_h^{n-4}-\frac{1}{6}{\bm{u}}_h^{n-5},\nonumber  \\
& \gamma_6(\bm{u}_h^n)=6\bm{u}_h^n-15\bm{u}_h^{n-1}+20\bm{u}_h^{n-2}-15\bm{u}_h^{n-3}+6\bm{u}_h^{n-4}-\bm{u}_h^{n-5}.\label{k6}
\end{align}

\begin{remark}
We emphasize that the present analysis differs substantially from existing results on BDF-based finite element discretizations for the Navier-Stokes equations, which are typically limited to orders $k=3,4, 5$ and require CFL-type step-size restrictions of the form $\tau\leq Ch^{\alpha}$. The unified framework developed in this work establishes stability and convergence for all orders $k=1,\cdots,6$, thereby removing any coupling condition between the time-step size and the mesh size. In particular, the stability and optimal-order error estimates for the sixth-order IMEX-BDF scheme appear to be new.  A further feature of the analysis is that the nonlinear convection term is treated in an IMEX manner, so that only linear problems have to be solved at each time step. The scheme therefore retains the computational efficiency of linearized time stepping while still admitting the energy estimate without any CFL-type restriction. %The theory is developed in a general finite element setting and applies to a broad class of inf-sup stable velocity-pressure pairs. Besides standard conforming pairs such as Taylor-Hood and MINI elements, the framework also covers $H(\mathrm{div})$-conforming discretizations, including Raviart-Thomas and Brezzi-Douglas-Marini elements. Thus, the proposed analysis provides a flexible and broadly applicable framework for high-order finite element approximations of incompressible flow problems.
\end{remark}

\subsection{The auxiliary estimates and main results}
%To establish the boundedness of the numerical solution in the energy norm  utilizing  BDF$k$ scheme within a unified framework, we define herein  the discrete Stokes operator $\mathcal{A}_{h}$, cf \cite{HeYinnian2007sinum}. %and $-\Delta_{h}$ is defined  in such a way that
%\begin{equation*}
% (-\Delta_{h}\bm{u}_h, \bm{v}_h)=(\nabla\bm{u}_h, \nabla\bm{v}_h), \,\,\forall \, \bm{u}_h,\, \bm{v}_h\in \bm{X}_h^{l}.
%\end{equation*}
%The corresponding discrete norm is defined as $\|\bm{v}_h\|_{2,h}=\|\mathcal{A}_h\bm{v}_h\|_0$. 
%The aforementioned result plays a key role  in establishing the stability of high-order BDF schemes for nonlinear parabolic equations \cite{Akrivis2015stability}. %, and it assumes a pivotal role in our error analysis.
%The above result played a key role in proving the stability of high-order BDF schemes for nonlinear parabolic equations \cite{Akrivis2015stability}, and it plays an important role in our error analysis.
The discrete Stokes operator will be used to control the nonlinear term and to formulate the stability estimate. Let $\bm{V}_h:=\bm{X}_{0h}^l$ be the discretely divergence-free velocity space. The discrete Stokes operator
$\mathcal{A}_h: \bm{V}_h\rightarrow \bm{V}_h$ is defined by
\begin{equation}
  (\mathcal{A}_h \bm{v}_h, \bm{w}_h)=(\nabla \bm{v}_h,\nabla \bm{w}_h)
  \qquad\forall\,\bm{w}_h\in \bm{V}_h.
\end{equation}
The operator $\mathcal{A}_h$ is self-adjoint and positive definite on $\bm{V}_h$. We accordingly introduce the discrete Stokes norm
\begin{align*}
  \|\bm{v}_h\|_{2,h}:=\|\mathcal{A}_h\bm{v}_h\|_0,
  \qquad \bm{v}_h\in \bm{V}_h.
\end{align*}
The following regularity property of the stationary Stokes problem is used in the discrete Sobolev estimates.
\begin{hypothesis}[Stokes regularity]
For every $\bm{g}\in\bm{L}^2(\Omega)$, the stationary Stokes problem
\[
\begin{aligned}
-\Delta\bm{z}+\nabla r&=\bm{g}
&&\text{in }\Omega,\\
\nabla\cdot\bm{z}&=0
&&\text{in }\Omega,\\
\bm{z}&=\bm{0}
&&\text{on }\partial\Omega
\end{aligned}
\]
admits a unique solution
\[
(\bm{z},r)
\in
\bigl(\bm{H}_0^1(\Omega)\cap
      \bm{H}^2(\Omega)\bigr)
\times
\bigl(H^1(\Omega)\cap Q\bigr)
\]
satisfying
\begin{equation}
 \|\bm{z}\|_{2,2}
 +\|r\|_{1,2}
 \le C\|\bm{g}\|_0.
\end{equation}
\end{hypothesis}

We next  state two discrete Sobolev estimates used in the nonlinear stability analyses. The first  inequality is standard,  whereas the second does not seem to appear in the  literature; we provide its proof in Appendix \ref{appen} for completeness.
\begin{lemma}[Discrete Sobolev estimates]\label{lemmaSobolev}
Assume that $\{\mathcal T_h\}_h$ is a shape-regular and quasi-uniform family of tetrahedral meshes of $\Omega$ and that the Stokes regularity assumption holds. Then, for every $\bm{v}_h\in \bm{V}_h$,  the following discrete Sobolev inequalities hold
\begin{equation}
 \|\bm{v}_h\|_{0,6} \le C\|\nabla \bm{v}_h\|_0,\qquad \|\nabla \bm{v}_h\|_{0,3} \le C\|\nabla \bm{v}_h\|_0^{1/2}    \|\mathcal{A}_h\bm{v}_h\|_0^{1/2}.
\end{equation}
Here,  $C$ denotes a generic positive constant independent of $h$. %Consequently, the nonlinear estimates used in Sections~3--5 hold with constants independent of the spatial mesh size.
\end{lemma}

%\subsection{Assumptions and main results}

The $k$-step scheme requires the starting velocities $\bm{u}_h^j\in \bm{X}_{0h}^l$, $j=0,\cdots,k-1$. The starting values are assumed to satisfy the following conditions  \cite{AitOuAmmi1994, Heywood1982}.
\begin{hypothesis}\label{initial value estimate}
Let $1\le k\le6$. The starting values $\bm{u}_h^j$,  $j=0,\cdots,k-1$, satisfy
\begin{align*}
& \max_{0\le j\le k-1}\|\bm{u}(t^j)-\bm{u}_h^j\|_0\le C\bigl(h^{l+1}+\tau^k\bigr), \qquad  \max_{0\le j\le k-1} \|\nabla(\bm{u}(t^j)-\bm{u}_h^j)\|_0\le C\bigl(h^l+\tau^k\bigr). 
\end{align*}
and
\begin{equation}
 \max_{0\le j\le k-1}\|\nabla\bm{u}_h^j\|_0^2
 +\nu\tau\sum_{j=0}^{k-1}
       \|\mathcal{A}_h\bm{u}_h^j\|_0^2\le C.
 \label{start-energy}
\end{equation}
Here $C$ is independent of $h$ and $\tau$.
\end{hypothesis}

%This section is dedicated to the error analysis of the numerical solution to the system (\ref{weak space discrete:eq1})-(\ref{weak space discrete:eq2}). 
To achieve an optimal-order $L^2$-norm error estimate for the velocity field, we impose the following regularity conditions on the exact solution.
%To derive the error estimates of the fully discrete solution, we will use a discrete version of the Gronwall inequality in a slightly more general form than usually found in the literature, and this detailed proof can be found in \cite{Heywood1990}.
%\begin{lemma}\label{dis:gronwall lemma}
%Let $C_*, \tau , a_n, b_n, c_n$ and $d_n$ be non-negative numbers with $n\ge 0$ such that 
%\begin{equation*}
%a_m+\tau \sum_{n=0}^mb_n\leq\tau \sum_{n=0}^{m}d_na_n+\tau \sum_{n=0}^{m}c_n+C_* \qquad \forall\,m\ge 0.
%\end{equation*}
%Suppose that $\tau  d_n<1$, for all $n$, and set $\lambda_n=(1-\tau  d_n)^{-1}$. Then 
%\begin{equation*}
%a_m+\tau \sum_{n=0}^mb_n\leq \exp\Big(\tau \sum_{n=0}^{m}\lambda_nd_n\Big)\Big\{\tau \sum_{n=0}^{m}c_n+C_*\Big\} \qquad \forall\,m\ge 0.
%\end{equation*}
%\end{lemma}
\begin{hypothesis}\label{assum:solution2}
The solution of (\ref{weak:eq1})-(\ref{weak:eq2}) exists and satisfies
\begin{equation*}
\begin{split}
&\bm{u}\in{L^\infty(0,T; \bm{H}^{1+l}(\Omega))}, \quad \frac{\partial\bm{u}}{\partial t}\in{L^2(0,T; \bm{H}^{1+l}(\Omega))}, \quad p\in {L^\infty(0,T; H^l(\Omega))},\\[2mm]
& \frac{\partial^{k+1}\bm{u}}{\partial t^{k+1}}\in{L^{2}(0,T; \bm{L}^{2}(\Omega))},\quad  \frac{\partial^{k}\bm{u}}{\partial t^{k}}\in{L^{2}(0,T; \bm{H}^{1}(\Omega))}, \quad \frac{\partial p}{\partial t}\in {L^2(0,T; H^l(\Omega))}.
%&\bm{u}_{ttt}\in{L^2(0,T; \bm{L}^{2})},\qquad \bm{A}_{ttt}\in{L^2(0,T; \bm{L}^{2})}, \qquad \theta_{ttt}\in{L^2(0,T; \bm{L}^{2})}.
\end{split}
\end{equation*}
%where the exponent $s>1/2$ depends on the regularity of the domain $\Omega$.
\end{hypothesis}
%\begin{remark}
%The domain $\Omega$ is a general polyhedral in $\mathbb{R}^3$, it might be non-smooth, non-convex, so the magnetic induction $\bm{B}$ may not be in $\bm{H}^1(\Omega)$ in general Lipschitz domains which are not convex, thus it is necessary to assume such a hypothesis, and this is the minimum regularity hypothesis to analyse error estimates. The regularity Hypothesis \ref{assum:solution2} ensures the well-posedness of (\ref{weak:eq1})-(\ref{weak:eq3}), see \cite{Tabata2005,Lorca1999,Sermange1983}.
%\end{remark}

In order to derive the optimal $H^1$-norm of velocity and the  optimal $L^2$-norm error estimate of pressure, we need to further  impose the following regularity assumption on the exact solution.
\begin{hypothesis}\label{assum:solution3}
The solution of (\ref{weak:eq1})-(\ref{weak:eq2}) exists and satisfies
\begin{equation*}
\begin{split}
 \frac{\partial^{k+1}\bm{u}}{\partial t^{k+1}}\in{L^{\infty}(0,T; \bm{H}^{1}(\Omega))},\qquad  \frac{\partial^{k}\bm{u}}{\partial t^{k}}\in{L^{\infty}(0,T; \bm{H}^{2}(\Omega))}. 
%&\bm{u}_{ttt}\in{L^2(0,T; \bm{L}^{2})},\qquad \bm{A}_{ttt}\in{L^2(0,T; \bm{L}^{2})}, \qquad \theta_{ttt}\in{L^2(0,T; \bm{L}^{2})}.
\end{split}
\end{equation*}
%where the exponent $s>1/2$ depends on the regularity of the domain $\Omega$.
\end{hypothesis}
\begin{theorem}\label{theoremmain}
Let $1\le k\le6$, and let $(\bm{u}_h^n, p_h^n)$ be generated by (\ref{weak space discrete:eq1})-(\ref{weak space discrete:eq2}). Assume Hypotheses \ref{initial value estimate}-\ref{assum:solution3} and $\max_{0\le n\le N}\|\bm{f}(t^n)\|_0\le C_f$.  For  $m+1\le N$, 
 there exist $h_0,\tau_0>0$, independent of $\tau$ and $h$, respectively, such that, for $h\le h_0$, $\tau\le\tau_0$, the fully discrete problem admits a unique solution 
\begin{align}
&\|\nabla\bm{u}_h^{m+1}\|_0^2+\sum_{n=k-1}^{m}\nu \tau \|\mathcal{A}_{h}\bm{u}_h^{n+1}\|_0^2\leq C_*, \label{lemmadeltah}
\end{align}
where $C_*>0$  is a constant independent of $\tau$ and $h$. Moreover,
\begin{equation}
\begin{split}
\Vert\bm{u}^{m+1}-\bm{u}_h^{m+1}\Vert_0%+\tau \frac{1-\mu_k^2}{2}\sum_{n=1}^m\nu\Vert\nabla(\bm{u}^n-\bm{u}_h^n)\Vert_0^2
\leq  C_u\left(h^{l+1}+\tau^{k}\right),\label{space discrete : theorem 1}
\end{split}
\end{equation}
and
\begin{equation}
\begin{split}
&\Vert\nabla\bm{u}^{m+1}-\nabla\bm{u}_h^{m+1}\Vert_0
\leq  C_u\left(h^{l}+\tau^{k}\right).\label{space discrete : theorem 2}
\end{split}
\end{equation}
Here $C_u$  is independent of $h$ and $\tau$.
\end{theorem}
Based on the results established in Theorem  \ref{theoremmain}, we derive the optimal error estimate for the pressure field.
\begin{theorem}\label{space discrete : theorem 3}
Under the same conditions as Theorem \ref{theoremmain}, with $k=1,\cdots,6$, the following estimate holds for all $m+1\leq N$
%Under the same conditions as Theorem \ref{space discrete : theorem 2}. %Then the continuous system (\ref{weak:eq1})-(\ref{weak:eq2}) and the finite element problem (\ref{weak space discrete:eq1})-(\ref{weak space discrete:eq2}) admit the unique solution $(\bm{u}^n, p^n)$ and $(\bm{u}_h^n,p_h^n)$, respectively. For all $1\leq m\leq N$, we  have the following estimate,
\begin{equation*}
\begin{split}
\tau  \sum_{n=k-1}^{m}\Vert p^{n+1}-p_h^{n+1}\Vert_0^2\leq  C_{u,p}\left(h^{2l}+\tau^{2k}\right),
\end{split}
\end{equation*}
where the constant $ C_{u,p}$ depends on $T,\nu,\Omega$  and the exact solution $\bm{u}, p$, but is independent
of $\tau $ and $h$.
\end{theorem}

\begin{remark}
The proposed fully discrete IMEX-BDF$k$ schemes are energy stable for sufficiently small time steps $\tau$ and mesh sizes $h$. In space, the schemes achieve optimal-order convergence rates for the velocity in both the $L^2$- and $H^1$-norms and for the pressure in the $L^2$-norm. In time, they attain the optimal $k$th-order convergence rate for all $k=1,\cdots,6$ in all of these error estimates. Moreover, the stability and error estimates are established without any CFL-type condition, in the sense that the admissible time-step restriction is independent of the spatial mesh size.
\end{remark}

%\begin{remark}
%The constants $\tau_0$ and $h_0$ are independent of $h$ and $\tau$, respectively. Hence, no coupling condition of the form $\tau\le Ch^\alpha$ is required. 
%\end{remark}
 
%We will derive the error estimates for the numerical solution $(\bm{u}_h^n, \bar{\bm{u}}_h^n, p_h^n)$.  
\section{Proofs of the main results}
 \label{section analysis}
 
 \subsection{Preparations for  error analysis}
 Prior to analyzing the stability and optimal error estimates of the numerical solution, we first  collect several necessary lemmas. In particular, the discrete Gr$\ddot{\text{o}}$nwall lemma stated below is frequently invoked, and its  proof  can be found in \cite{Heywood1990}. %These lemmas  are essential for establishing the error estimates of the numerical schemes.
%\begin{lemma}\label{def:gronwall}(Gr$\ddot{\text{o}}$nwall Lemma 1)
%Let $C_{\star\star}, \tau , a_n, b_n, c_n$ and $d_n$ be non-negative numbers with $n\ge 0$ such that 
%\begin{equation*}
%a_m+\tau \sum_{n=0}^mb_n\leq\tau \sum_{n=0}^{m}d_na_n+\tau \sum_{n=0}^{m}c_n+C_{\star\star} \qquad \forall\,m\ge 0.
%\end{equation*}
%For all $n$, 
%Suppose that $\tau  d_n<1$, for all n, and set $\sigma_n=(1-\tau  d_n)^{-1}$.  Then, 
%\begin{equation*}
%a_m+\tau \sum_{n=0}^mb_n\leq \exp\left(\tau \sum_{n=0}^{m}\sigma_nd_n\right)\left\{\tau \sum_{n=0}^{m}c_n+C_{\star\star}\right\} \qquad \forall\,m\ge 0.
%\end{equation*}
%\end{lemma}
\begin{lemma}\label{def:gronwall2}(Gr$\ddot{\text{o}}$nwall Lemma)
Let $C_{\star\star}, \tau , a_n, b_n, c_n$ and $d_n$ be non-negative numbers with $n\ge 0$ such that 
\begin{equation*}
a_m+\tau \sum_{n=0}^mb_n\leq\tau \sum_{n=0}^{m-1}d_na_n+\tau \sum_{n=0}^{m-1}c_n+C_{\star\star} \qquad \forall\,m\ge 0.
\end{equation*}
%For all $n$, 
Then, 
\begin{equation*}
a_m+\tau \sum_{n=0}^mb_n\leq \exp\left(\tau \sum_{n=0}^{m-1}d_n\right)\left\{\tau \sum_{n=0}^{m-1}c_n+C_{\star\star}\right\} \qquad \forall\,m\ge 0.
\end{equation*}
\end{lemma}

%Additionally, we recall  the following lemma from \cite{LiuJianGuo2010}, which plays an essential role in establishing local error estimates in the three-dimensional setting.
%\begin{lemma}\label{lemma:con-in}
%Let $\Phi:(0,\infty)\to(0,\infty)$ be continuous and increasing, and let $M>0$. Given $T_*$ such that $0<T_*<\int_M^{\infty}dz/\Phi(z)$, there exists a constant $C_*>0$ independent of $\tau >0$ 
%with the following property. Suppose that quantities $z_i, w_i\ge0$ satisfy
%\begin{equation*}
%z_m+\sum_{i=0}^{m-1}\tau  w_i\leq y_m:=M+\sum_{i=0}^{m-1}\tau  \Phi(z_i), \,\,\forall n\leq n_*,
%\end{equation*}
%with $n_*\tau \leq T_*$. Then $y_{n_*}\leq C_*$.
%\end{lemma}

For BDF methods of orders at most five, we use the Nevanlinna-Odeh multiplier technique, which is based on Dahlquist's $G$-stability theory \cite{Nevanlinna1981}.
%We  conclude by recalling the following important result established by Nevanlinna and Odeh \cite{Nevanlinna1981} based on Dahlquist’s G-stability theory, which establishes the stability properties of BDF$k$ schemes for $k=1,\cdots,5$.
\begin{lemma}\label{lemma_symmetricmatrix}
For  $k=1,\cdots,5$, there exist $0\leq \mu_k<1$, a positive definite symmetric matrix $G=(g_{ij})\in \R^{k,k}$ and real numbers $\delta_0,\cdots,\delta_k$ such that
\begin{align*}
(\alpha_k\bm{u}^{n+1}-\beta_k(\bm{u}^n), \bm{u}^{n+1}-\mu_k\bm{u}^n)=&\sum_{i,j=1}^kg_{ij}(\bm{u}^{n+1+i-k}, \bm{u}^{n+1+j-k})\\[2mm]
&-\sum_{i,j=1}^kg_{ij}(\bm{u}^{n+i-k}, \bm{u}^{n+j-k})+\left\|\sum_{i=0}^k\delta_{i}\bm{u}^{n+1+i-k}\right\|_0^2,
\end{align*}
where the smallest possible values of $\mu_k$ are
$\mu_1=\mu_2=0$, $\mu_3=0.0836, \mu_4=0.2878, \mu_5=0.8160$,  and $\alpha_k, \beta_k$ are defined in (\ref{k1})-(\ref{k5}).
\end{lemma}

%Owing to the absence of a Nevanlinna and Odeh multiplier for the six-step BDF method, we resort to the energy estimates developed for BDF-$6$ in the context of an abstract parabolic PDE within a Hilbert space in \cite{Akrivis2021sinum} (sections 3.1 and 3.2) to establish both energy stability and optimal error estimates for the case $k = 6$. Further details can be found in Lemma 3.6 and Theorem 5.1 of \cite{Alessandro2025sinum}.

The absence of a Nevanlinna and Odeh multiplier for the six-step BDF method necessitates an alternative approach to establishing energy stability and optimal error estimates. We therefore employ the energy estimates developed for BDF$6$ in the setting of an abstract parabolic PDE in a Hilbert space, as presented in sections 
 3.1-3.2 of \cite{Akrivis2021sinum}. 
 Further details can be found in Lemma 3.6 and Theorem 5.1 of \cite{Alessandro2025sinum}.
\begin{lemma}\label{lemma_symmetricmatrix1}
%Let $(\cdot ,\cdot )$ be a real inner product on a Hilbert space,
There exists a positive definite symmetric matrix $G=(g_{ij})\in \R^{6,6}$ such that, for every sequence $\{\bm{u}^j\}$ in the underlying real Hilbert space,
\begin{align*}
&(\alpha_6\bm{u}^{n+1}-\beta_6(\bm{u}^n), \hat{\bm{u}}^{n+1})\ge\sum_{i,j=1}^6g_{ij}(\bm{u}^{n+i-5}, \bm{u}^{n+j-5})
-\sum_{i,j=1}^6g_{ij}(\bm{u}^{n+i-6}, \bm{u}^{n+j-6}),
\end{align*}
where $\hat{\bm{u}}^{n+1}=\bm{u}^{n+1}-\frac{13}{9}\bm{u}^{n}+\frac{25}{36}\bm{u}^{n-1}-\frac{1}{9}\bm{u}^{n-2}$, and $\alpha_6, \beta_6$ are defined in (\ref{k6}).
\end{lemma}

To facilitate the error analysis, we introduce the Galerkin projection operators. Given $(\bm{u}, p)\in (\bm{X}_0\times Q)$, then our aim is to find $\mathcal{P}_{h}\bm{u}\in \bm{X}_h^{l}$, $\mathcal{Q}_{h} p\in Q_h^{l-1}$, for all $(\bm{v}_h, q_h)\in (\bm{X}_h^{l}\times Q_h^{l-1})$, the following holds
\begin{align}
&\nu(\nabla\mathcal{P}_{h}\bm{u},\nabla \bm{v}_h)-d(\bm{v}_h, \mathcal{Q}_{h}p)+d(\mathcal{P}_{h}\bm{u}, q_h)% \nonumber\\
=\nu(\nabla\bm{u},\nabla \bm{v}_h)-d(\bm{v}_h, p)+d(\bm{u}, q_h).\label{Galerkin projection:define}
\end{align}
As established in \cite{Girault1986},   the following key properties are recalled
\begin{align}
&\Vert \bm{u}^{n+1}-\mathcal{P}_h\bm{u}^{n+1}\Vert_0+h\Vert\nabla (\bm{u}^{n+1}-\mathcal{P}_h\bm{u}^{n+1})\Vert_0+h\Vert p^{n+1}-\mathcal{Q}_hp^{n+1}\Vert_0\nonumber\\[2mm]
&\leq Ch^{1+l}\left(\Vert\bm{u}^{n+1}\Vert_{1+l,2}+\Vert p^{n+1}\Vert_{l,2}\right),\nonumber\\[2mm]
&\left\|\frac{\partial(\bm{u}^{n+1}-\mathcal{P}_h\bm{u}^{n+1})}{\partial t}\right\|_{0}\leq Ch^{1+l}\left(\left\|\frac{\partial\bm{u}^{n+1}}{\partial t}\right\|_{l+1,2}+\left\|\frac{\partial p^{n+1}}{\partial t}\right\|_{l,2}\right).\label{galerkin projection propertites:1}
\end{align}

For notational convenience,  we define here and hereafter,
%For the sake of convenience, we introduce the following notations,
\begin{eqnarray*}
&e_{\bm{u}}^{n+1}=\bm{u}^{n+1}-\mathcal{P}_h\bm{u}^{n+1},\, e_p^{n+1}=p^{n+1}-\mathcal{Q}_hp^{n+1},
\, \eta_{\bm{u}}^{n+1}=\mathcal{P}_h\bm{u}^{n+1}-\bm{u}_h^{n+1},\, \eta_p^{n+1}=\mathcal{Q}_hp^{n+1}-p_h^{n+1}.
\end{eqnarray*}

Setting $\bm{v}=\bm{v}_h$ in (\ref{weak:eq1}) and $q=q_h$ in (\ref{weak:eq2}),  and combining with (\ref{weak space discrete:eq1})-(\ref{weak space discrete:eq2}) as well as (\ref{Galerkin projection:define}), we derive the error equations for ${\bm{u}}_h^{n+1}$ and  $p_h^{n+1}$ as
\begin{align}
&(\alpha_k{\eta}_{\bm{u}}^{n+1}-\beta_k({\eta}_{\bm{u}}^{n}),\bm{v}_h)+\nu\tau (\nabla{\eta}_{\bm{u}}^{n+1},\nabla\bm{v}_h)-\tau (\eta_p^{n+1},\nabla\cdot\bm{v}_h)\nonumber\\[2mm]
=&-\tau a(\bm{u}^{n+1},\bm{u}^{n+1}, \bm{v}_h)+\tau a(\gamma_k(\bm{u}_h^{n}),\gamma_k(\bm{u}_h^{n}), \bm{v}_h)\nonumber\\[2mm]
&+(\alpha_k\mathcal{P}_h\bm{u}^{n+1}-\beta_k(\mathcal{P}_h\bm{u}^{n})-\tau \partial_t\bm{u}^{n+1},\bm{v}_h)\nonumber\\[2mm]
=& (\tau  E_{k,1}+E_{k,2}+\tau E_{k,3},\bm{v}_h),\label{eq:erroreq1}\\[2mm]
%\end{align}
%\begin{align}
&\tau (q_h,\nabla\cdot{\eta}_{\bm{u}}^{n+1})=0,\label{eq:erroreq2}
\end{align}
where
\begin{align*}
E_{k,1}=&-\bm{u}^{n+1}\cdot\nabla\bm{u}^{n+1}+\gamma_k(\bm{u}_h^{n})\cdot\nabla\gamma_k(\bm{u}_h^{n})\\[2mm]
=&[-\bm{u}^{n+1}+\gamma_k(\bm{u}^n)-\gamma_k(\bm{u}^n-\bm{u}_h^n)]\cdot\nabla\bm{u}^{n+1}\\[2mm]
&-\gamma_k(\bm{u}_h^{n})\cdot\nabla[\bm{u}^{n+1}-\gamma_k(\bm{u}^n)]-\gamma_k(\bm{u}_h^{n})\cdot\nabla\gamma_k(\bm{u}^n-\bm{u}_h^{n}),
\end{align*}
and
\begin{align*}
E_{k,2}=&\alpha_k\bm{u}^{n+1}-\beta_k(\bm{u}^{n})-\tau \frac{\partial\bm{u}^{n+1}}{\partial t}%\\[2mm]
=\sum_{i=1}^ko_i\int_{t^{n+1-i}}^{t^{n+1}}(s-t^{n+1-i})^k\frac{\partial^{k+1} \bm{u}}{\partial t^{k+1}}(s)\,ds,
\end{align*}
where $o_i$  are fixed, bounded constants determined by the truncation errors, for example, in the case $k=4$,  the following expression holds
\begin{align*}
E_{4,2}=&\alpha_4\bm{u}^{n+1}-\beta_4(\bm{u}^{n})-\tau \frac{\partial\bm{u}^{n+1}}{\partial{t}}\\[2mm]
%=&4(\mathcal{P}_h\bm{u}^{n+1}-\mathcal{P}_h\bm{u}^{n})-3(\mathcal{P}_h\bm{u}^{n+1}-\mathcal{P}_h\bm{u}^{n-1})\\
%&+\frac{4}{3}(\mathcal{P}_h\bm{u}^{n+1}-\mathcal{P}_h\bm{u}^{n-2})-\frac{1}{4}(\mathcal{P}_h\bm{u}^{n+1}-\mathcal{P}_h\bm{u}^{n-3})-\tau \partial_{t}\mathcal{P}_h\bm{u}^{n+1}\\
=&\frac{1}{6}\int_{t^{n}}^{t^{n+1}}(s-t^{n})^4\frac{\partial^{5}\bm{u}}{\partial t^{5}}(s)\,ds-\frac{1}{8}\int_{t^{n-1}}^{t^{n+1}}(s-t^{n-1})^4\frac{\partial^{5}\bm{u}}{\partial t^{5}}(s)\,ds\\[2mm]
&+\frac{1}{18}\int_{t^{n-2}}^{t^{n+1}}(s-t^{n-2})^4\frac{\partial^{5}\bm{u}}{\partial t^{5}}(s)\,ds-\frac{1}{96}\int_{t^{n-3}}^{t^{n+1}}(s-t^{n-3})^4\frac{\partial^{5} \bm{u}}{\partial t^{5}}(s)\,ds,
\end{align*}
and in the case $k=6$,  the following expression holds
\begin{align*}
E_{6,2}=&\alpha_6\bm{u}^{n+1}-\beta_6(\bm{u}^{n})-\tau\frac{\partial\bm{u}^{n+1}}{\partial{t}}\\[2mm]
%=&6(\mathcal{P}_h\bm{u}^{n+1}-\mathcal{P}_h\bm{u}^{n})-\frac{15}{2}(\mathcal{P}_h\bm{u}^{n+1}-\mathcal{P}_h\bm{u}^{n-1})+\frac{20}{3}(\mathcal{P}_h\bm{u}^{n+1}-\mathcal{P}_h\bm{u}^{n-2})\\
%&-\frac{15}{4}(\mathcal{P}_h\bm{u}^{n+1}-\mathcal{P}_h\bm{u}^{n-3})+\frac{6}{5}(\mathcal{P}_h\bm{u}^{n+1}-\mathcal{P}_h\bm{u}^{n-4})-\frac{1}{6}(\mathcal{P}_h\bm{u}^{n+1}-\mathcal{P}_h\bm{u}^{n-5})-\tau\partial_{t}\mathcal{P}_h\bm{u}^{n+1}\\
=&\frac{1}{120}\int_{t^{n}}^{t^{n+1}}(s-t^{n})^6\frac{\partial^{7} \bm{u}}{\partial t^{7}}(s)\,ds-\frac{1}{96}\int_{t^{n-1}}^{t^{n+1}}(s-t^{n-1})^6\frac{\partial^{7} \bm{u}}{\partial t^{7}}(s)\,ds\\[2mm]
&+\frac{1}{108}\int_{t^{n-2}}^{t^{n+1}}(s-t^{n-2})^6\frac{\partial^{7} \bm{u}}{\partial t^{7}}(s)\,ds-\frac{1}{192}\int_{t^{n-3}}^{t^{n+1}}(s-t^{n-3})^6\frac{\partial^{7}\bm{u}}{\partial t^{7}}(s)\,ds\\[2mm]
&+\frac{1}{600}\int_{t^{n-4}}^{t^{n+1}}(s-t^{n-4})^6\frac{\partial^{7} \bm{u}}{\partial t^{7}}(s)\,ds-\frac{1}{4320}\int_{t^{n-5}}^{t^{n+1}}(s-t^{n-5})^6\frac{\partial^{7} \bm{u}}{\partial t^{7}}(s)\,ds.
\end{align*}
Moreover, $E_{k,3}$ is represented as follows
\begin{align*}
E_{k,3}= \alpha_k\mathcal{P}_h\bm{u}^{n+1}-\beta_k(\mathcal{P}_h\bm{u}^{n})-[\alpha_k\bm{u}^{n+1}-\beta_k(\bm{u}^{n})].
\end{align*}
To facilitate the subsequent proof, we first give an estimate for $E_{k,3}$.
\begin{lemma}
\label{lemmaEk3}
Let $1\leq k\leq 6$ and $n\geq k-1$. Suppose that the exact solution satisfies the regularity assumptions in Hypothesis \ref{assum:solution2}. Then
\begin{align*}
  \| E_{k,3}\|_0^2\leq  Ch^{2(l+1)}\tau \int_{t^{n+1-k}}^{t^{n+1}} \left( \left\|\frac{\partial\bm u}{\partial t}(s)\right\|_{l+1,2}^2+ \left\|\frac{\partial p}{\partial t}(s)\right\|_{l,2}^2\right)\,ds.
\end{align*}
where $C>0$ is independent of $h$, $\tau$.
\end{lemma}

\begin{proof}
Write the BDF$k$ difference operator in the form
\begin{align*}
  \alpha_ke_{\bm{u}}^{n+1}-\beta_k(e_{\bm{u}}^{n})=\sum_{j=0}^{k}a_{k,j}e_{\bm{u}}^{n+1-j},
\end{align*}
where the coefficients $a_{k,j}$ are fixed and satisfy $\sum_{j=0}^{k}a_{k,j}=0$. It follows that
\begin{align}
 \alpha_ke_{\bm{u}}^{n+1}-\beta_k(e_{\bm{u}}^{n}) &=  \sum_{j=1}^{k}
  a_{k,j}
  \left(e_{\bm u}^{n+1-j}-e_{\bm u}^{n+1}\right)
  =
  -
  \sum_{j=1}^{k}a_{k,j}
  \int_{t^{n+1-j}}^{t^{n+1}}
  \partial_t e_{\bm u}(s)\,ds.
  \label{eq:projection-difference-integral}
\end{align}
Since $k\leq6$ and all the coefficients are fixed, applying the Cauchy-Schwarz inequality together with the approximation property of the mixed Stokes projection yields
\begin{align*}
  \| E_{k,3}\|_0^2\leq&C\tau\int_{t^{n+1-k}}^{t^{n+1}}\|\partial_t e_{\bm u}(s)\|_0^2\,ds\\
 \leq& Ch^{2(l+1)}\tau \int_{t^{n+1-k}}^{t^{n+1}} \left( \left\|\frac{\partial\bm u}{\partial t}(s)\right\|_{l+1,2}^2+ \left\|\frac{\partial p}{\partial t}(s)\right\|_{l,2}^2\right)\,ds.
\end{align*}
This completes the estimate for $E_{k,3}$.
\end{proof}

\subsection{Proof of optimal error estimates for velocity}
We now prove optimal $L^2$- and $H^1$-error estimates for the velocity of the fully discrete numerical solution to the system (\ref{weak space discrete:eq1})-(\ref{weak space discrete:eq2}). 
\begin{proof}[Proof of Theorem \ref{theoremmain}] 
\textbf{Step I}:  This proof is divided into four steps. Firstly, we denote
\begin{align}
C_{H1}=\max_{0\leq t\leq T}\|\nabla\bm{u}(\cdot, t) \|_0\quad \text{and}\,\,\,C_{\Diamond}=C_{H1}+2.\label{induction0}
\end{align}
We need to prove a uniform bound of $\nabla\bm{u}_h^n$ by induction,
\begin{align}
\|\nabla\bm{u}_h^n \|_0\leq C_{\Diamond}, \quad \forall\, n\leq N.\label{induction1}
\end{align}
Under Hypothesis \ref{initial value estimate}, (\ref{induction1}) certainly holds for $n=0,\cdots,k-1$. Now suppose we have 
\begin{align}
\|\nabla\bm{u}_h^n \|_0\leq C_{\Diamond}, \quad \forall\, n\leq m.\label{induction2}
\end{align}
we shall prove below
\begin{align}
\|\nabla\bm{u}_h^{m+1} \|_0\leq C_{\Diamond} \label{induction3}
\end{align}
for the same constant $C_{\Diamond}$.

For $1\leq k\leq 5$, choosing  $\bm{v}_h=\mathcal{A}_{h}\tilde{\bm{u}}_h^{n+1}\in\bm{X}_{0h}^{l}$ in (\ref{weak space discrete:eq1}),  where $\tilde{\bm{u}}_h^{n+1}={\bm{u}}_h^{n+1}-\mu_k{\bm{u}}_h^{n}$, and taking  $q_h=0$ in (\ref{weak space discrete:eq2}), then adding the two equations, there holds
\begin{align}
&(\alpha_k{\bm{u}}_h^{n+1}-\beta_k({\bm{u}}_h^{n}), \mathcal{A}_{h}\tilde{\bm{u}}_h^{n+1})+\nu\tau (\mathcal{A}_{h}{\bm{u}}_h^{n+1}, \mathcal{A}_{h}\tilde{\bm{u}}_h^{n+1})\nonumber\\[2mm]
&+\tau  a(\gamma_k(\bm{u}_h^n),\gamma_k(\bm{u}_h^n), \mathcal{A}_{h}\tilde{\bm{u}}_h^{n+1})=\tau (\bm{f}^{n+1}, \mathcal{A}_{h}\tilde{\bm{u}}_h^{n+1}).\label{fully:uh0}
\end{align}
For the first two terms on the left-hand side of (\ref{fully:uh0}), we invoke Lemma \ref{lemma_symmetricmatrix} to derive the following identity
\begin{align}
&(\alpha_k{\bm{u}}_h^{n+1}-\beta_k({\bm{u}}_h^{n}), \mathcal{A}_{h}\tilde{\bm{u}}_h^{n+1})%\nonumber\\
=(\nabla(\alpha_k{\bm{u}}_h^{n+1}-\beta_k({\bm{u}}_h^{n})), \nabla({\bm{u}}_h^{n+1}-\mu_k{\bm{u}}_h^{n}))\nonumber\\[2mm]
=&\sum_{i,j=1}^kg_{ij}(\nabla{\bm{u}}_h^{n+1+i-k}, \nabla{\bm{u}}_h^{n+1+j-k})%\nonumber\\
-\sum_{i,j=1}^kg_{ij}(\nabla{\bm{u}}_h^{n+i-k}, \nabla{\bm{u}}_h^{n+j-k})+\left\|\sum_{i=0}^k\delta_{i}\nabla{\bm{u}}_h^{n+1+i-k}\right\|_0^2.\label{fully:uh1}
\end{align}
By performing a direct calculation, we derive the following result
\begin{align}
&\nu(\mathcal{A}_{h}{\bm{u}}_h^{n+1}, \mathcal{A}_{h}({\bm{u}}_h^{n+1}-\mu_k{\bm{u}}_h^{n}))
=\nu\|\mathcal{A}_{h}{\bm{u}}_h^{n+1}\|_0^2-\nu\mu_k(\mathcal{A}_{h}{\bm{u}}_h^{n+1}, \mathcal{A}_{h}{\bm{u}}_h^{n}).\label{fully:uh2}
\end{align}
Inserting (\ref{fully:uh1})-(\ref{fully:uh2}) into (\ref{fully:uh0}) can be inferred  
\begin{align}
&\sum_{i,j=1}^kg_{ij}(\nabla{\bm{u}}_h^{n+1+i-k}, \nabla{\bm{u}}_h^{n+1+j-k})-\sum_{i,j=1}^kg_{ij}(\nabla{\bm{u}}_h^{n+i-k}, \nabla{\bm{u}}_h^{n+j-k})\nonumber\\[2mm]
&+\left\|\sum_{i=0}^k\delta_{i}\nabla{\bm{u}}_h^{n+1+i-k}\right\|_0^2+\nu\tau \|\mathcal{A}_{h}{\bm{u}}_h^{n+1}\|_0^2-\nu\tau \mu_k(\mathcal{A}_{h}{\bm{u}}_h^{n+1}, \mathcal{A}_{h}{\bm{u}}_h^{n})\nonumber\\[2mm]
&+\tau  a(\gamma_k(\bm{u}_h^n), \gamma_k(\bm{u}_h^n), \mathcal{A}_{h}\tilde{\bm{u}}_h^{n+1})=\tau (\bm{f}^{n+1}, \mathcal{A}_{h}\tilde{\bm{u}}_h^{n+1}).\label{fully:uh00}
\end{align}
%For the nonlinear term on the left-hand side of (\ref{eq:GSAV1}), we rewrite it as
Applying Lemma \ref{lemmaSobolev} to control the nonlinear convection term, we obtain 
\begin{align}
&\left| a(\gamma_k(\bm{u}_h^n), \gamma_k(\bm{u}_h^n), \mathcal{A}_{h}({\bm{u}}_h^{n+1}-\mu_k{\bm{u}}_h^{n}))\right|\nonumber\\[2mm]
\leq& C\|\gamma_k({\bm{u}}_h^n)\|_{0,6}\|\nabla\gamma_k({\bm{u}}_h^n)\|_{0,3}\|\mathcal{A}_{h}({\bm{u}}_h^{n+1}-\mu_k{\bm{u}}_h^{n})\|_0\nonumber\\[2mm]
\leq& C\|\nabla\gamma_k({\bm{u}}_h^n)\|_0\|\nabla\gamma_k({\bm{u}}_h^n)\|_0^{1/2}\|\gamma_k({\bm{u}}_h^n)\|_{2,h}^{1/2}(\|\mathcal{A}_{h}{\bm{u}}_h^{n+1}\|_0+\mu_k\|\mathcal{A}_{h}{\bm{u}}_h^{n}\|_0)\nonumber\\[2mm]
\leq&C(\varepsilon\nu)^{-1}\|\nabla\gamma_k({\bm{u}}_h^n)\|_0^3\|\gamma_k({\bm{u}}_h^n)\|_{2,h}+\varepsilon\nu\|\mathcal{A}_{h}{\bm{u}}_h^{n+1}\|_0^2+\varepsilon\nu\mu_k^2\|\mathcal{A}_{h}{\bm{u}}_h^{n}\|_0^2\nonumber\\[2mm]
\leq&C(\varepsilon\nu)^{-2}\|\nabla\gamma_k({\bm{u}}_h^n)\|_0^6+\varepsilon\nu\|\mathcal{A}_{h}\gamma_k({\bm{u}}_h^n)\|_{0}^2+\varepsilon\nu\|\mathcal{A}_{h}{\bm{u}}_h^{n+1}\|_0^2+\varepsilon\nu\mu_k^2\|\mathcal{A}_{h}{\bm{u}}_h^{n}\|_0^2,\label{fully:uh3a}
\end{align}
%where we apply the elliptic regularity estimate $\|\gamma_k({\bm{u}}_h^n)\|_{2,h}^2\leq \|\mathcal{A}_{h}(\gamma_k({\bm{u}}_h^n))\|_{0}^2$ in the last inequality above.  
By invoking Young's inequality, there holds
\begin{align}
&\left|\nu\mu_k(\mathcal{A}_{h}{\bm{u}}_h^{n+1}, \mathcal{A}_{h}{\bm{u}}_h^{n})\right|\leq \frac{1}{2}\nu\|\mathcal{A}_{h}{\bm{u}}_h^{n+1}\|_0^2+\frac{1}{2}\nu\mu_k^2\|\mathcal{A}_{h}{\bm{u}}_h^{n}\|_0^2.\label{fully:uh3}
\end{align}
For the right-hand side of  (\ref{fully:uh00}), it implies
\begin{align}
&\left|(\bm{f}^{n+1}, \mathcal{A}_{h}\tilde{\bm{u}}_h^{n+1})\right|\leq \varepsilon\nu\|\mathcal{A}_{h}{\bm{u}}_h^{n+1}\|_0^2+\varepsilon\nu\mu_k^2\|\mathcal{A}_{h}{\bm{u}}_h^{n}\|_0^2+C(\varepsilon\nu)^{-1}\|\bm{f}^{n+1}\|_0^2. \label{fully:uh4}
\end{align}
Now, combining (\ref{fully:uh3a})-(\ref{fully:uh4}) with (\ref{fully:uh00}), we obtain
\begin{align}
&\sum_{i,j=1}^kg_{ij}(\nabla{\bm{u}}_h^{n+1+i-k}, \nabla{\bm{u}}_h^{n+1+j-k})-\sum_{i,j=1}^kg_{ij}(\nabla{\bm{u}}_h^{n+i-k}, \nabla{\bm{u}}_h^{n+j-k})\nonumber\\[2mm]
&+\left\|\sum_{i=0}^k\delta_{i}\nabla{\bm{u}}_h^{n+1+i-k}\right\|_0^2+\nu\tau (1/2-2\varepsilon)\|\mathcal{A}_{h}{\bm{u}}_h^{n+1}\|_0^2\nonumber\\[2mm]
%\end{align}
%\begin{align}
\leq&\nu\tau (1/2+2\varepsilon)\mu_k^2\|\mathcal{A}_{h}{\bm{u}}_h^{n}\|_0^2+C(\varepsilon\nu)^{-1} \tau \|\bm{f}^{n+1}\|_0^2\nonumber\\[2mm]
&+C(\varepsilon\nu)^{-2} \tau \|\nabla\gamma_k({\bm{u}}_h^n)\|_0^6+2\varepsilon\nu\tau \|\mathcal{A}_{h}\gamma_k({\bm{u}}_h^n)\|_{0}^2.\label{fully:uh5}
\end{align}
We can choose $\varepsilon$  small enough, such that 
%\begin{align*}
$\frac{1}{2}-2\varepsilon>(\frac{1}{2}+2\varepsilon)\mu_k^2+2\varepsilon$. 
%\end{align*}
Noting that $G=(g_{ij})$  is a symmetric positive definite matrix with smallest eigenvalue $\lambda_g$ and  largest eigenvalue $\lambda_G$. Applying (\ref{induction2}),  then take the sum on (\ref{fully:uh5}) for $n$ from $k-1$ to $m$ and drop some unnecessary terms:
\begin{align}
&\lambda_g\|\nabla{\bm{u}}_h^{m+1}\|_0^2+\nu\tau \frac{1-\mu_k^2}{2}\sum_{n=k-1}^{m}\|\mathcal{A}_{h}{\bm{u}}_h^{n+1}\|_0^2\nonumber\\[2mm]
%\leq &\sum_{i,j=1}^kg_{ij}(\nabla{\bm{u}}_h^{m+i-k}, \nabla{\bm{u}}_h^{m+j-k})+\nu\tau \frac{1-\mu_k^2}{2}\sum_{n=k-1}^{m}\|\mathcal{A}_{h}{\bm{u}}_h^{n+1}\|_0^2\nonumber\\
\leq&C\tau \sum_{n=k-1}^{m}\left(\|\nabla\gamma_k({\bm{u}}_h^n)\|_0^6+\|\nabla\gamma_k({\bm{u}}_h^n)\|_0^2\right)%\nonumber\\
+C\tau \sum_{n=k-1}^{m}\|\bm{f}^{n+1}\|_0^2+\lambda_G\sum_{i=0}^{k-1}\|\nabla{\bm{u}}_h^i\|_0^2\nonumber\\[2mm]
\leq&C\tau \sum_{n=k-1}^{m}\left(\|\nabla{\bm{u}}_h^n\|_0^6+\|\nabla{\bm{u}}_h^{n}\|_0^2\right)+CTC_f^2+M_0\leq CT(C_{\Diamond}^6+C_{\Diamond}^2+C_f^2)+M_0\leq C_*,\label{m0}
\end{align}
where $M_0$ is a constant dependent only on the initial data, and the  estimate $\|\bm{f}\|_0^2\leq C_f^2$ has been applied. 

For $k=6$, choosing  $\bm{v}_h=\mathcal{A}_{h}\hat{\bm{u}}_h^{n+1}\in\bm{X}_{0h}^{l}$ in (\ref{weak space discrete:eq1}),  where $\hat{\bm{u}}_h^{n+1}=\bm{u}_h^{n+1}-\frac{13}{9}\bm{u}_h^{n}+\frac{25}{36}\bm{u}_h^{n-1}-\frac{1}{9}\bm{u}_h^{n-2}$, and taking  $q_h=0$ in (\ref{weak space discrete:eq2}), then adding the two equations, there holds
\begin{align}
&(\alpha_6{\bm{u}}_h^{n+1}-\beta_6({\bm{u}}_h^{n}), \mathcal{A}_{h}\hat{\bm{u}}_h^{n+1})+\nu\tau(\mathcal{A}_{h}{\bm{u}}_h^{n+1}, \mathcal{A}_{h}\hat{\bm{u}}_h^{n+1})\nonumber\\[2mm]
&+\tau a(\gamma_6(\bm{u}_h^n),\gamma_6(\bm{u}_h^n), \mathcal{A}_{h}\hat{\bm{u}}_h^{n+1})=\tau(\bm{f}^{n+1}, \mathcal{A}_{h}\hat{\bm{u}}_h^{n+1}).\label{fully:uh0k6}
\end{align}
For the first term on the left-hand side of (\ref{fully:uh0k6}), we invoke Lemma \ref{lemma_symmetricmatrix1} to derive the following identity
\begin{align}
&(\alpha_6{\bm{u}}_h^{n+1}-\beta_6({\bm{u}}_h^{n}), \mathcal{A}_{h}\hat{\bm{u}}_h^{n+1})%\nonumber\\
=(\nabla(\alpha_6{\bm{u}}_h^{n+1}-\beta_6({\bm{u}}_h^{n})), \nabla\hat{\bm{u}}_h^{n+1})\nonumber\\[2mm]
\ge&\sum_{i,j=1}^6g_{ij}(\nabla{\bm{u}}_h^{n+i-5}, \nabla{\bm{u}}_h^{n+j-5})%\nonumber\\
-\sum_{i,j=1}^6g_{ij}(\nabla{\bm{u}}_h^{n+i-6}, \nabla{\bm{u}}_h^{n+j-6})\nonumber\\[2mm]
=& |\nabla \mathbb{U}_h^{n+1}|_{G}^2-| \nabla\mathbb{U}_h^{n}|_{G}^2.\label{fully:uh1k6}
\end{align}
Here the norm $|\nabla \mathbb{U}_h^{n+1}|_{G}^2$ is given by
\begin{align*}
| \nabla\mathbb{U}_h^{n+1}|_{G}^2=\sum_{i,j=1}^6g_{ij}(\nabla{\bm{u}}_h^{n+i-5},\nabla{\bm{u}}_h^{n+j-5}),
\end{align*}
for the notation $\mathbb{U}_h^{n+1}=(\bm{u}_h^{n-4}, \bm{u}_h^{n-3}, \bm{u}_h^{n-2}, \bm{u}_h^{n-1}, \bm{u}_h^{n}, \bm{u}_h^{n+1})^{\top}$.  Inserting (\ref{fully:uh1k6}) into (\ref{fully:uh0k6}) can be inferred  
\begin{align}
&|\nabla \mathbb{U}_h^{n+1}|_{G}^2-|\nabla  \mathbb{U}_h^{n}|_{G}^2+\nu\tau(\mathcal{A}_{h}{\bm{u}}_h^{n+1}, \mathcal{A}_{h}\hat{\bm{u}}_h^{n+1})\nonumber\\[2mm]
&+\tau a(\gamma_6(\bm{u}_h^n), \gamma_6(\bm{u}_h^n), \mathcal{A}_{h}\hat{\bm{u}}_h^{n+1})\leq\tau(\bm{f}^{n+1}, \mathcal{A}_{h}\hat{\bm{u}}_h^{n+1}).\label{fully:uh00k6}
\end{align}
%For the nonlinear term on the left-hand side of (\ref{eq:GSAV1}), we rewrite it as
 Similar to (\ref{fully:uh3a}), we have the following estimate
%For the  nonlinear convection term involving  $\gamma(\bm{u}_h^n)\cdot\nabla\gamma(\bm{u}_h^n)$, we obtain 
\begin{align}
&| a(\gamma_6(\bm{u}_h^n), \gamma_6(\bm{u}_h^n), \mathcal{A}_{h}\hat{\bm{u}}_h^{n+1})|\nonumber\\[2mm]
%\leq& C\|\nabla\gamma_6({\bm{u}}_h^n)\|_0\|\nabla\gamma_6({\bm{u}}_h^n)\|_0^{1/2}\|\gamma_6({\bm{u}}_h^n)\|_{2,h}^{1/2}\|\mathcal{A}_{h}\hat{\bm{u}}_h^{n+1}\|_0\nonumber\\
%\leq&C(\varepsilon\nu)^{-1}\|\nabla\gamma_6({\bm{u}}_h^n)\|_0^3\|\gamma_6({\bm{u}}_h^n)\|_{2,h}+\varepsilon\nu\|\mathcal{A}_{h}\hat{\bm{u}}_h^{n+1}\|_0^2\nonumber\\
\leq&C(\varepsilon\nu)^{-2}\|\nabla\gamma_6({\bm{u}}_h^n)\|_0^6+\varepsilon\nu\|\mathcal{A}_{h}\gamma_6({\bm{u}}_h^n)\|_{0}^2+\varepsilon\nu\|\mathcal{A}_{h}\hat{\bm{u}}_h^{n+1}\|_0^2.\label{fully:uh3ak6}
\end{align}
%where we apply the elliptic regularity estimate $\|\gamma({\bm{u}}_h^n)\|_{2,h}^2\leq \|\mathcal{A}_{h}(\gamma({\bm{u}}_h^n))\|_{0}^2$ in the last inequality above.  
By invoking Young's inequality, for the right-hand side of  (\ref{fully:uh00k6}), it implies
\begin{align}
&|(\bm{f}^{n+1}, \mathcal{A}_{h}\hat{\bm{u}}_h^{n+1})|\leq \varepsilon\nu\|\mathcal{A}_{h}\hat{\bm{u}}_h^{n+1}\|_0^2+C(\varepsilon\nu)^{-1}\|\bm{f}^{n+1}\|_0^2. \label{fully:uh4k6}
\end{align}
Now, combining (\ref{fully:uh3ak6})-(\ref{fully:uh4k6}) with (\ref{fully:uh00k6}),  and then taking the sum for $n$ from $5$ to $m$, we obtain
\begin{align}
%&\sum_{i,j=1}^6g_{ij}(\nabla{\bm{u}}_h^{n+i-5}, \nabla{\bm{u}}_h^{n+j-5})-\sum_{i,j=1}^6g_{ij}(\nabla{\bm{u}}_h^{n+i-6}, \nabla{\bm{u}}_h^{n+j-6})\nonumber\\
&|\nabla  \mathbb{U}_h^{m+1}|_{G}^2+\nu\tau\sum_{n=5}^{m}(\mathcal{A}_{h}{\bm{u}}_h^{n+1}, \mathcal{A}_{h}\hat{\bm{u}}_h^{n+1})%\nonumber\\
%\end{align}
%\begin{align}
\leq |\nabla  \mathbb{U}_h^{5}|_{G}^2+2\varepsilon\nu\tau\sum_{n=5}^{m}\|\mathcal{A}_{h}\hat{\bm{u}}_h^{n+1}\|_0^2\nonumber\\[2mm]
&+C(\varepsilon\nu)^{-2} \tau\sum_{n=5}^{m}\|\nabla\gamma_6({\bm{u}}_h^n)\|_0^6+\varepsilon\nu\tau\sum_{n=5}^{m}\|\mathcal{A}_{h}\gamma_6({\bm{u}}_h^n)\|_{0}^2+C(\varepsilon\nu)^{-1} \sum_{n=5}^{m}\tau\|\bm{f}^{n+1}\|_0^2.\label{fully:uh50k6}
\end{align}
In section 3.1 of  \cite{Akrivis2021sinum}, the authors were able to bound the sum expression on the left-hand side of (\ref{fully:uh00k6}) from below by recasting the sums into a weighted double sum of inner products where the weights stem from a Toeplitz matrix which turns out to be nonnegative thanks to relaxed positivity assumption $1-\frac{13}{9}\cos(x)+\frac{25}{36}\cos(2x)-\frac{1}{9}\cos(3x)>0$. For the velocity the resulting estimate reads
\begin{align}
\nu\sum_{n=5}^m(\mathcal{A}_{h}{\bm{u}}_h^{n+1}, \mathcal{A}_{h}\hat{\bm{u}}_h^{n+1})\ge& \nu\frac{1}{32}\sum_{n=5}^m\|\mathcal{A}_{h}{\bm{u}}_h^{n+1}\|_0^2%\nonumber\\
-\nu\left(\mathcal{A}_{h}{\bm{u}}_h^{6},  \frac{13}{9}\mathcal{A}_{h}{\bm{u}}_h^{5}-\frac{25}{36}\mathcal{A}_{h}{\bm{u}}_h^{4}+\frac{1}{9}\mathcal{A}_{h}{\bm{u}}_h^{3}\right)\nonumber\\[2mm]
&-\nu\left(\mathcal{A}_{h}{\bm{u}}_h^{7}, -\frac{25}{36}\mathcal{A}_{h}{\bm{u}}_h^{5}+\frac{1}{9}\mathcal{A}_{h}{\bm{u}}_h^{4}\right)-\nu\left(\mathcal{A}_{h}{\bm{u}}_h^{8}, \frac{1}{9}\mathcal{A}_{h}{\bm{u}}_h^{5}\right).\label{fully:uh3k3}
\end{align}
 By invoking Young's inequality, there holds
\begin{align}
&\left|\nu\left(\mathcal{A}_{h}{\bm{u}}_h^{6},  \frac{13}{9}\mathcal{A}_{h}{\bm{u}}_h^{5}-\frac{25}{36}\mathcal{A}_{h}{\bm{u}}_h^{4}+\frac{1}{9}\mathcal{A}_{h}{\bm{u}}_h^{3}\right)\right|\nonumber\\[2mm]
\leq & 3\varepsilon\nu\|\mathcal{A}_{h}{\bm{u}}_h^{6}\|_0^2+C(\varepsilon\nu)^{-1}\left[\|\mathcal{A}_{h}{\bm{u}}_h^{5}\|_0^2+\|\mathcal{A}_{h}{\bm{u}}_h^{4}\|_0^2+\|\mathcal{A}_{h}{\bm{u}}_h^{3}\|_0^2\right],  \label{fully:uh3k300}\\[2mm]
&\left|\nu\left(\mathcal{A}_{h}{\bm{u}}_h^{7}, -\frac{25}{36}\mathcal{A}_{h}{\bm{u}}_h^{5}+\frac{1}{9}\mathcal{A}_{h}{\bm{u}}_h^{4}\right)\right|\nonumber\\[2mm]
\leq& 2 \varepsilon\nu\|\mathcal{A}_{h}{\bm{u}}_h^{7}\|_0^2+C(\varepsilon\nu)^{-1}\left[\|\mathcal{A}_{h}{\bm{u}}_h^{5}\|_0^2+\|\mathcal{A}_{h}{\bm{u}}_h^{4}\|_0^2\right], \label{fully:uh3k301}\\[2mm]
&\left|\nu\left(\mathcal{A}_{h}{\bm{u}}_h^{8}, \frac{1}{9}\mathcal{A}_{h}{\bm{u}}_h^{5}\right)\right|\leq \varepsilon\nu\|\mathcal{A}_{h}{\bm{u}}_h^{8}\|_0^2+C(\varepsilon\nu)^{-1}\|\mathcal{A}_{h}{\bm{u}}_h^{5}\|_0^2.\label{fully:uh3k302}
\end{align}
Now, combining (\ref{fully:uh3k3})-(\ref{fully:uh3k302}) with (\ref{fully:uh50k6}), we obtain
\begin{align}
%&\sum_{i,j=1}^6g_{ij}(\nabla{\bm{u}}_h^{n+i-5}, \nabla{\bm{u}}_h^{n+j-5})-\sum_{i,j=1}^6g_{ij}(\nabla{\bm{u}}_h^{n+i-6}, \nabla{\bm{u}}_h^{n+j-6})\nonumber\\
&|\nabla  \mathbb{U}_h^{m+1}|_{G}^2+ \nu\tau\frac{1}{32}\sum_{n=5}^m\|\mathcal{A}_{h}{\bm{u}}_h^{n+1}\|_0^2\nonumber\\[2mm]
%\end{align}
%\begin{align}
\leq&|\nabla  \mathbb{U}_h^{5}|_{G}^2+2\varepsilon\nu\tau\sum_{n=5}^{m}\|\mathcal{A}_{h}\hat{\bm{u}}_h^{n+1}\|_0^2+\varepsilon\nu\tau\left[\|\mathcal{A}_{h}{\bm{u}}_h^{8}\|_0^2+2\|\mathcal{A}_{h}{\bm{u}}_h^{7}\|_0^2+3\|\mathcal{A}_{h}{\bm{u}}_h^{6}\|_0^2 \right]\nonumber\\[2mm]
&+C(\varepsilon\nu)^{-2} \tau\sum_{n=5}^{m}\|\nabla\gamma_6({\bm{u}}_h^n)\|_0^6+\varepsilon\nu\tau\sum_{n=5}^{m}\|\mathcal{A}_{h}\gamma_6({\bm{u}}_h^n)\|_{0}^2+C(\varepsilon\nu)^{-1} \sum_{n=5}^{m}\tau\|\bm{f}^{n+1}\|_0^2\nonumber\\[2mm]
&+3C(\varepsilon\nu)^{-1} \tau\|\mathcal{A}_{h}{\bm{u}}_h^{5}\|_0^2+2C(\varepsilon\nu)^{-1} \tau\|\mathcal{A}_{h}{\bm{u}}_h^{4}\|_0^2+C(\varepsilon\nu)^{-1} \tau\|\mathcal{A}_{h}{\bm{u}}_h^{3}\|_0^2.\label{fully:uh5k6}
\end{align}
We can choose $\varepsilon$  small enough, such that
%\begin{align*}
$\frac{1}{64}>9\varepsilon$.
%\end{align*}
Noting that $G=(g_{ij})$  is a symmetric positive definite matrix with  smallest eigenvalue $\lambda_{1g}$  and largest eigenvalue $\lambda_{1G}$,  there holds %Then take the sum on (\ref{fully:uh5k6}) for $n$ from $5$ to $m$ and drop some unnecessary terms:
\begin{align*}
&\lambda_{1g}\|\nabla{\bm{u}}_h^{m+1}\|_0^2+\nu\tau\frac{1}{64}\sum_{n=5}^{m}\|\mathcal{A}_{h}{\bm{u}}_h^{n+1}\|_0^2\nonumber\\[2mm]
%\leq &\sum_{i,j=1}^6g_{ij}(\nabla{\bm{u}}_h^{m+i-6}, \nabla{\bm{u}}_h^{m+j-6})+\nu\tau\frac{1}{64}\sum_{n=5}^{m}\|\mathcal{A}_{h}{\bm{u}}_h^{n+1}\|_0^2\nonumber\\
\leq&C\tau\sum_{n=5}^{m}\left(\|\nabla\gamma_6({\bm{u}}_h^n)\|_0^6+\|\nabla\gamma_6({\bm{u}}_h^n)\|_0^2\right)%\nonumber\\
+C\tau\sum_{n=5}^{m}\|\bm{f}^{n+1}\|_0^2+C\tau\sum_{i=3}^{5}\|\mathcal{A}_{h}{\bm{u}}_h^i\|_0^2+\lambda_{1G}\sum_{i=0}^5\|\nabla \bm{u}_h^i \|_0^2\\[2mm]
\leq&C\tau\sum_{n=5}^{m}\left(\|\nabla{\bm{u}}_h^n\|_0^6+\|\nabla{\bm{u}}_h^{n}\|_0^2\right)+CTC_f^2+M_0,
\end{align*}
where $M_0$ is a constant dependent only on the initial data $\bm{u}_h^0, \bm{u}_h^1, \bm{u}_h^2, \bm{u}_h^3, \bm{u}_h^4, \bm{u}_h^5$ and we used $\|\bm{f}\|_0^2\leq C_f^2$.
%Now, if we define $\Phi$  as $\Phi(x)=x^3+x$ in Lemma \ref{lemma:con-in} and let
%$$0<T_*<\int_{CTC_f^2+M_0}^{\infty}dz/\Phi(z),$$
Under the induction assumption in (\ref{induction2}), there exists $C_*>0$ independent of $\tau$ such that
\begin{align}
&\lambda_{1g}\|\nabla{\bm{u}}_h^{m+1}\|_0^2+\tau\frac{1}{64}\sum_{n=5}^{m}\nu\|\mathcal{A}_{h}{\bm{u}}_h^{n+1}\|_0^2\leq  CT(C_{\diamond}^6+C_{\diamond}^2+C_f^2) +M_0 \leq C_*.\label{m06}
\end{align}
%This completes the proof of  (\ref{lemmadeltah}).\\

%Now, if we define $\Phi$  as $\Phi(x)=x^3+x$ in Lemma \ref{lemma:con-in} and let
%$$0<T_*<\int_{CTC_f^2+M_0}^{\infty}dz/\Phi(z),$$
%then Lemma \ref{lemma:con-in} implies that there exists $C_*>0$ independent of $\tau $ such that
%\begin{align*}
%&\lambda_G\|\nabla{\bm{u}}_h^m\|_0^2+\tau \frac{1-\mu_k^2}{2}\sum_{n=k}^{m}\nu\|\mathcal{A}_{h}{\bm{u}}_h^n\|_0^2\leq C_*.
%\end{align*}
%This completes the proof of   (\ref{lemmadeltah}).
%\end{proof}

%Building upon  Hypothesis \ref{assum:solution2}, we are ready to show the main convergence result for the BDF-$k$ in full discretization.
%\begin{theorem}\label{space discrete : theorem 1}
%Suppose that the continuous system (\ref{weak:eq1})-(\ref{weak:eq2}) has a unique solution  $(\bm{u}^n, p^n)$ satisfying  %Hypothesis \ref{assum:right1} and 
%Hypothesis \ref{assum:solution2}. Then there exist a positive constant $\tau _0$ such that when $\tau \leq \tau _0$,  the finite element problem (\ref{weak space discrete:eq1})-(\ref{weak space discrete:eq2}) admits a unique solution  $(\bm{u}_h^n,p_h^n)$, with  $k=1,\cdots,5$,  which satisfies, for all $m+1\leq N$
%\begin{equation*}
%\begin{split}
%\Vert\bm{u}^{m+1}-\bm{u}_h^{m+1}\Vert_0%+\tau \frac{1-\mu_k^2}{2}\sum_{n=1}^m\nu\Vert\nabla(\bm{u}^n-\bm{u}_h^n)\Vert_0^2
%\leq  C_u\left(h^{l+1}+(\tau )^{k}\right),
%\end{split}
%\end{equation*}
%where the constant $C_u$ depends on $T,\nu,\Omega$  and the exact solution $\bm{u}$, but is independent
%of $\tau $ and $h$.
%\end{theorem}
%\begin{proof}
\textbf{Step II}: For $1\leq k\leq 5$, from equation (\ref{eq:erroreq2}),  it can be deduced that
 \begin{align}
&\tau (q_h,\nabla\cdot({\eta}_{\bm{u}}^{n+1}-\mu_k{\eta}_{\bm{u}}^{n}))=0. \label{eq:erroreq02}
\end{align}
 Choosing $\bm{v}_h=\tilde{\eta}_{\bm{u}}^{n+1}$ in (\ref{eq:erroreq1}) with $\tilde{\eta}_{\bm{u}}^{n+1}={\eta}_{\bm{u}}^{n+1}-\mu_k{\eta}_{\bm{u}}^{n}$,  $q_h={\eta}_{p}^{n+1}$ in (\ref{eq:erroreq02}), and adding up these two  equations yields the following
\begin{align}
&(\alpha_k\eta_{\bm{u}}^{n+1}-\beta_k(\eta_{\bm{u}}^{n}),\tilde{\eta}_{\bm{u}}^{n+1})+\nu\tau (\nabla\eta_{\bm{u}}^{n+1},\nabla\tilde{\eta}_{\bm{u}}^{n+1})%\nonumber\\
=(\tau  E_{k,1}+E_{k,2}+\tau E_{k,3},\tilde{\eta}_{\bm{u}}^{n+1}). \label{fully:esti0}
\end{align}
For the left-hand side of (\ref{fully:esti0}), with the help of Lemma \ref{lemma_symmetricmatrix}, it can be concluded that 
\begin{align}
&(\alpha_k\eta_{\bm{u}}^{n+1}-\beta_k(\eta_{\bm{u}}^{n}), \tilde{\eta}_{\bm{u}}^{n+1})%\nonumber\\
=(\alpha_k\eta_{\bm{u}}^{n+1}-\beta_k(\eta_{\bm{u}}^{n}), \eta_{\bm{u}}^{n+1}-\mu_k\eta_{\bm{u}}^{n})\nonumber\\[2mm]
=&\sum_{i,j=1}^kg_{ij}(\eta_{\bm{u}}^{n+1+i-k}, \eta_{\bm{u}}^{n+1+j-k})-\sum_{i,j=1}^kg_{ij}(\eta_{\bm{u}}^{n+i-k}, \eta_{\bm{u}}^{n+j-k})%\nonumber\\
+\left\|\sum_{i=0}^k\delta_{i}\eta_{\bm{u}}^{n+1+i-k}\right\|_0^2.\label{fully:esti1}
\end{align}
%and
%\begin{align}
%&(-\nu\Delta\bar{e}_{\bm{u}}^{n+1}, \bar{e}_{\bm{u}}^{n+1}-\mu_k\bar{e}_{\bm{u}}^{n})
%=\nu\|\nabla\bar{e}_{\bm{u}}^{n+1}\|_0^2-\nu(\nabla\bar{e}_{\bm{u}}^{n+1}, \mu_k\nabla\bar{e}_{\bm{u}}^{n}).\label{fully:esti2}
%\end{align}
By inserting (\ref{fully:esti1}) into  (\ref{fully:esti0}), we obtain
\begin{align}
&\sum_{i,j=1}^kg_{ij}(\eta_{\bm{u}}^{n+1+i-k}, \eta_{\bm{u}}^{n+1+j-k})-\sum_{i,j=1}^kg_{ij}(\eta_{\bm{u}}^{n+i-k}, \eta_{\bm{u}}^{n+j-k})%\nonumber\\
+\left\|\sum_{i=0}^k\delta_{i}\eta_{\bm{u}}^{n+1+i-k}\right\|_0^2+\nu\tau \|\nabla\eta_{\bm{u}}^{n+1}\|_0^2\nonumber\\[2mm]
=&\nu\tau (\nabla\eta_{\bm{u}}^{n+1}, \mu_k\nabla\eta_{\bm{u}}^{n})+(\tau  E_{k,1}+E_{k,2}+\tau E_{k,3},\tilde{\eta}_{\bm{u}}^{n+1}). \label{fully:esti00}
\end{align}
We bound the terms on the right-hand side of (\ref{fully:esti00}) using (\ref{bestimate}). Then 
\begin{align*}
&\left|a(-\bm{u}^{n+1}+\gamma_k(\bm{u}^n)-\gamma_k(\bm{u}^n-\bm{u}_h^n),\bm{u}^{n+1}, \eta_{\bm{u}}^{n+1}-\mu_k\eta_{\bm{u}}^{n})\right|\\[2mm]
\leq&(\|\bm{u}^{n+1}-\gamma_k(\bm{u}^n)\|_{0}+\| \gamma_k(\bm{u}^n-\bm{u}_h^n)\|_{0})\|\bm{u}^{n+1}\|_{2,2}\|\eta_{\bm{u}}^{n+1}-\mu_k\eta_{\bm{u}}^{n}\|_{1,2}\\[2mm]
\leq&C(\varepsilon\nu)^{-1}\|\bm{u}^{n+1}\|_{2,2}^2 \left\|\sum_{i=1}^ko_i\int_{t^{n+1-i}}^{t^{n+1}}(s-t^{n+1-i})^{k-1}\frac{\partial^{k} \bm{u}}{\partial t^{k}}(s)\,ds\right\|_{0}^2\\[2mm]
&+2\varepsilon\nu\|\nabla\eta_{\bm{u}}^{n+1}\|_0^2+2\mu_k^2\varepsilon\nu\|\nabla\eta_{\bm{u}}^{n}\|_0^2+C(\varepsilon\nu)^{-1}(\| \gamma_k(\eta_{\bm{u}}^n)\|_0^2+\| \gamma_k(e_{\bm{u}}^n)\|_{0}^2)\|\bm{u}^{n+1}\|_{2,2}^2\\[2mm]
\leq&C(\varepsilon\nu)^{-1}\|\bm{u}^{n+1}\|_{2,2}^2 \tau^{2k-1} \sum_{i=1}^ko_i\int_{t^{n+1-i}}^{t^{n+1}}\left\|\frac{\partial^{k} \bm{u}}{\partial t^{k}}(s)\right\|_0^2\,ds\\[2mm]
&+2\varepsilon\nu\|\nabla\eta_{\bm{u}}^{n+1}\|_0^2+2\mu_k^2\varepsilon\nu\|\nabla\eta_{\bm{u}}^{n}\|_0^2+C(\varepsilon\nu)^{-1}\left(\| \gamma_k(\eta_{\bm{u}}^n)\|_0^2+\| \gamma_k(e_{\bm{u}}^n)\|_{0}^2\right)\|\bm{u}^{n+1}\|_{2,2}^2,
\end{align*}
where $o_i$ are some fixed and bounded constants determined by the truncation error.
 As an illustration, in the case $k=4$, we have
\begin{align*}
\bm{u}^{n+1}-\gamma_4(\bm{u}^n)
%=4(\bm{u}^{n+1}-\bm{u}^n)-6(\bm{u}^{n+1}-\bm{u}^{n-1})+4(\bm{u}^{n+1}-\bm{u}^{n-2})-(\bm{u}^{n+1}-\bm{u}^{n-3})\\
=&-\frac{2}{3}\int_{t^{n}}^{t^{n+1}}(s-t^{n})^{3}\frac{\partial^{4} \bm{u}}{\partial t^{4}}(s)\,ds+\int_{t^{n-1}}^{t^{n+1}}(s-t^{n-1})^{3}\frac{\partial^{4} \bm{u}}{\partial t^{4}}(s)\,ds\\[2mm]
&-\frac{2}{3}\int_{t^{n-2}}^{t^{n+1}}(s-t^{n-2})^{3}\frac{\partial^{4} \bm{u}}{\partial t^{4}}(s)\,ds+\frac{1}{6}\int_{t^{n-3}}^{t^{n+1}}(s-t^{n-3})^{3}\frac{\partial^{4} \bm{u}}{\partial t^{4}}(s)\,ds. %\label{un1andgammaunk4}
\end{align*}
Similarly, we can derive that
\begin{align*}
&\left|a(\gamma_k(\bm{u}_h^{n}),\bm{u}^{n+1}-\gamma_k(\bm{u}^n), \eta_{\bm{u}}^{n+1}-\mu_k\eta_{\bm{u}}^{n})\right|\\[2mm]
\leq&\|\gamma_k(\bm{u}_h^{n})\|_{1,2}\|\bm{u}^{n+1}-\gamma_k(\bm{u}^n)\|_{1,2}\|\eta_{\bm{u}}^{n+1}-\mu_k\eta_{\bm{u}}^{n}\|_{1,2}\\[2mm]
%\leq& C\|\gamma_k(\bm{u}_h^{n})\|_{1,2}^2\left\|\sum_{i=1}^kb_i\int_{t^{n+1-i}}^{t^{n+1}}(t^{n+1-i}-s)^{k-1}\frac{\partial^{k} \bm{u}}{\partial t^{k}}(s)\,ds\right\|_{1,2}^2
%+\varepsilon\nu\left(\|\nabla\eta_{\bm{u}}^{n+1}\|_0^2+\mu_k^2\|\nabla\eta_{\bm{u}}^{n}\|_0^2\right)\\
\leq& C\tau^{2k-1}(\varepsilon\nu)^{-1} \|\gamma_k(\bm{u}_h^{n})\|_{1,2}^2 \sum_{i=1}^ko_i \int_{t^{n+1-i}}^{t^{n+1}}\left\|\frac{\partial^{k} \bm{u}}{\partial t^{k}}(s)\right\|_{1,2}^2\,ds
+\varepsilon\nu\|\nabla\eta_{\bm{u}}^{n+1}\|_0^2+\mu_k^2\varepsilon\nu\|\nabla\eta_{\bm{u}}^{n}\|_0^2,
\end{align*}
and
\begin{align*}
&\left|a(\gamma_k(\bm{u}_h^{n}),\gamma_k(\bm{u}^n-\bm{u}_h^{n}), \eta_{\bm{u}}^{n+1}-\mu_k\eta_{\bm{u}}^{n})\right|\\[2mm]
\leq&\|\mathcal{A}_{h}(\gamma_k(\bm{u}_h^{n}))\|_{0}\|\gamma_k(\eta_{\bm{u}}^{n})\|_0\|\eta_{\bm{u}}^{n+1}-\mu_k\eta_{\bm{u}}^{n}\|_{1,2}
+\|\mathcal{A}_{h}(\gamma_k(\bm{u}_h^{n}))\|_{0}\|\gamma_k(e_{\bm{u}}^{n})\|_0\|\eta_{\bm{u}}^{n+1}-\mu_k\eta_{\bm{u}}^{n}\|_{1,2}\\[2mm]
\leq&C(\varepsilon\nu)^{-1}\left(\|\mathcal{A}_{h}(\gamma_k(\bm{u}_h^{n}))\|_{0}^2\|\gamma_k(\eta_{\bm{u}}^n)\|_0^2+\|\mathcal{A}_{h}(\gamma_k(\bm{u}_h^{n}))\|_{0}^2\|\gamma_k(e_{\bm{u}}^{n})\|_0^2\right)
+\varepsilon\nu\|\nabla\eta_{\bm{u}}^{n+1}\|_0^2+\mu_k^2\varepsilon\nu\|\nabla\eta_{\bm{u}}^{n}\|_0^2.
\end{align*}
Combining these estimates, we can conclude that
\begin{align}
\left|(E_{k,1}, \tilde{\eta}_{\bm{u}}^{n+1})\right|\leq&4\varepsilon\nu\|\nabla\eta_{\bm{u}}^{n+1}\|_0^2+4\mu_k^2\varepsilon\nu\|\nabla\eta_{\bm{u}}^{n}\|_0^2\nonumber\\[2mm]
&+ C(\varepsilon\nu)^{-1} \tau^{2k-1}\|\bm{u}^{n+1}\|_{2,2}^2\sum_{i=1}^ko_i \int_{t^{n+1-i}}^{t^{n+1}}\left\|\frac{\partial^{k} \bm{u}}{\partial t^{k}}(s)\right\|_{0}^2\,ds\nonumber\\[2mm]
&+ C(\varepsilon\nu)^{-1} \tau^{2k-1}\|\gamma_k(\bm{u}_h^{n})\|_{1,2}^2\sum_{i=1}^ko_i \int_{t^{n+1-i}}^{t^{n+1}}\left\|\frac{\partial^{k} \bm{u}}{\partial t^{k}}(s)\right\|_{1,2}^2\,ds\nonumber\\[2mm]
&+C(\varepsilon\nu)^{-1}(\|\mathcal{A}_{h}(\gamma_k(\bm{u}_h^{n}))\|_{0}^2+\|\bm{u}^{n+1}\|_{2,2}^2)\|\gamma_k(\eta_{\bm{u}}^n)\|_0^2\nonumber\\[2mm]
&+C(\varepsilon\nu)^{-1}(\|\mathcal{A}_{h}(\gamma_k(\bm{u}_h^{n}))\|_{0}^2+\|\bm{u}^{n+1}\|_{2,2}^2)\|\gamma_k(e_{\bm{u}}^{n})\|_0^2.\label{fully:esti4}
\end{align}
We continue to estimate
\begin{align}
\left|(E_{k,2}, \tilde{\eta}_{\bm{u}}^{n+1})\right|\leq &\varepsilon\nu\tau \|\nabla\eta_{\bm{u}}^{n+1}\|_0^2+\mu_k^2\varepsilon\nu\tau \|\nabla\eta_{\bm{u}}^{n}\|_0^2\nonumber\\[2mm]
 &+C(\varepsilon\nu)^{-1}\tau^{-1}\tau^{2k+1} \sum_{i=1}^ko_i \int_{t^{n+1-i}}^{t^{n+1}}\left\|\frac{\partial^{k+1} \mathcal{P}_h\bm{u}}{\partial t^{k+1}}(s)\right\|_0^2\,ds,\label{fully:esti5}
 \end{align}
and
\begin{align}
&\left|(E_{k,3}, \tilde{\eta}_{\bm{u}}^{n+1})\right|\leq C(\varepsilon\nu)^{-1}h^{2(l+1)}\tau \int_{t^{n+1-k}}^{t^{n+1}} \left( \left\|\frac{\partial\bm u}{\partial t}(s)\right\|_{l+1,2}^2+ \left\|\frac{\partial p}{\partial t}(s)\right\|_{l,2}^2\right)\,ds+\varepsilon\nu\|\nabla\eta_{\bm{u}}^{n+1}\|_0^2+\mu_k^2\varepsilon\nu\|\nabla\eta_{\bm{u}}^{n}\|_0^2.\label{fully:esti6}
\end{align}
By applying Young's inequality, we can easily derive the following 
\begin{align}
&\left|\nu(\nabla\eta_{\bm{u}}^{n+1}, \mu_k\nabla\eta_{\bm{u}}^{n})\right|\leq \frac{1}{2}\nu\|\nabla\eta_{\bm{u}}^{n+1}\|_0^2+\frac{1}{2}\mu_k^2\nu\|\nabla\eta_{\bm{u}}^{n}\|_0^2. \label{fully:esti9}
\end{align}
Now, combining (\ref{fully:esti4})-(\ref{fully:esti9}) with (\ref{fully:esti00}), we derive
\begin{align}
&\sum_{i,j=1}^kg_{ij}(\eta_{\bm{u}}^{n+1+i-k}, \eta_{\bm{u}}^{n+1+j-k})-\sum_{i,j=1}^kg_{ij}(\eta_{\bm{u}}^{n+i-k}, \eta_{\bm{u}}^{n+j-k})%\nonumber\\[2mm]
+\left\|\sum_{i=0}^k\delta_{i}\eta_{\bm{u}}^{n+1+i-k}\right\|_0^2+\nu\tau ({1}/{2}-6\varepsilon)\|\nabla\eta_{\bm{u}}^{n+1}\|_0^2\nonumber\\[2mm]
%\end{align}
%\begin{align}
\leq&\nu\tau (1/2+6\varepsilon)\mu_k^2\|\nabla\eta_{\bm{u}}^{n}\|_0^2+ C\tau^{2k}\|\bm{u}^{n+1}\|_{2,2}^2\sum_{i=1}^ko_i \int_{t^{n+1-i}}^{t^{n+1}}\left\|\frac{\partial^{k} \bm{u}}{\partial t^{k}}(s)\right\|_0^2\,ds\nonumber\\[2mm]
&+C\tau \left(\|\mathcal{A}_{h}(\gamma_k(\bm{u}_h^n))\|_{0}^2+\|\bm{u}^{n+1}\|_{2,2}^2\right)\left(\|\gamma_k(\eta_{\bm{u}}^n)\|_0^2+\|\gamma_k(e_{\bm{u}}^n)\|_0^2\right)\nonumber\\[2mm]
%&+C\tau (\|\mathcal{A}_{h}(\gamma_k(\bm{u}_h^n))\|_{0}^2+\|\bm{u}^{n+1}\|_{2,2}^2)\|\gamma_k(e_{\bm{u}}^n)\|_0^2\nonumber\\
%&+ C(\tau )^{2k}\|\bm{u}^{n+1}\|_{2,2}^2\int_{t^{n+1-i}}^{t^{n+1}}\|\frac{\partial^{k} \bm{u}}{\partial t^{k}}(s)\|_0^2\,ds\nonumber\\
&+ C\tau^{2k}\|\gamma_k(\bm{u}_h^{n})\|_{1,2}^2\sum_{i=1}^ko_i \int_{t^{n+1-i}}^{t^{n+1}}\left\|\frac{\partial^{k} \bm{u}}{\partial t^{k}}(s)\right\|_{1,2}^2\,ds\nonumber\\[2mm]
%&+ C(\tau )^{2k}\Big(\int_{t^{n+1-k}}^{t^{n+1}}\|\frac{\partial^{k} \bm{u}}{\partial t^{k}}(s)\|_{1,2}^2\,ds\Big)\nonumber\\
&+C\tau^{2k}\sum_{i=1}^ko_i\int_{t^{n+1-i}}^{t^{n+1}}\left\|\frac{\partial^{k+1} \bm{u}}{\partial t^{k+1}}(s)\right\|_0^2\,ds+Ch^{2(l+1)}\tau^2 \int_{t^{n+1-k}}^{t^{n+1}} \left( \left\|\frac{\partial\bm u}{\partial t}(s)\right\|_{l+1,2}^2+ \left\|\frac{\partial p}{\partial t}(s)\right\|_{l,2}^2\right)\,ds. \label{fully:esti11}
\end{align}
Next, with $0\leq \mu_k<1$, we can choose $\varepsilon$ small enough such that 
$\frac{1}{2}-6\varepsilon>(\frac{1}{2}+6\varepsilon)\mu_k^2.$
Summing both sides of inequality (\ref{fully:esti11}) for $n$ from $k-1$ to $m$, and noting  that $G=(g_{ij})$  is a symmetric positive definite matrix with smallest eigenvalue $\lambda_{2g}$ and  largest eigenvalue $\lambda_{2G}$,  we then obtain, after dropping some unnecessary terms
\begin{align}
&\lambda_{2g}\|\eta_{\bm{u}}^{m+1}\|_0^2+\tau \frac{1-\mu_k^2}{2}\sum_{n=k-1}^{m}\nu\|\nabla\eta_{\bm{u}}^{n+1}\|_0^2\nonumber\\[2mm]
%\leq &\sum_{i,j=1}^kg_{ij}(\eta_{\bm{u}}^{m+1+i-k}, \eta_{\bm{u}}^{m+1+j-k})+\tau \frac{1-\mu_k^2}{2}\sum_{n=0}^{m+1}\nu\|\nabla\eta_{\bm{u}}^{n}\|_0^2\nonumber\\
\leq&C\tau \sum_{n=k-1}^m\left(\|\mathcal{A}_{h}(\gamma_k(\bm{u}_h^n))\|_{0}^2+\|\bm{u}^{n+1}\|_{2,2}^2\right)\|\gamma_k(\eta_{\bm{u}}^n)\|_0^2\nonumber\\[2mm]
&+C\tau \sum_{n=k-1}^m\left(\|\mathcal{A}_{h}(\gamma_k(\bm{u}_h^n))\|_{0}^2+\|\bm{u}^{n+1}\|_{2,2}^2\right)\|\gamma_k(e_{\bm{u}}^n)\|_0^2\nonumber\\[2mm]
&+ C\tau^{2k}\int_{0}^{T}\left(\left\|\frac{\partial^{k} \bm{u}}{\partial t^{k}}(s)\right\|_0^2+\left\|\frac{\partial^{k} \bm{u}}{\partial t^{k}}(s)\right\|_{1,2}^2+\left\|\frac{\partial^{k+1} \bm{u}}{\partial t^{k+1}}(s)\right\|_0^2\right)\,ds\nonumber\\[2mm]
&+Ch^{2(l+1)}\tau^2  \int_{0}^{T} \left( \left\|\frac{\partial\bm u}{\partial t}(s)\right\|_{l+1,2}^2+ \left\|\frac{\partial p}{\partial t}(s)\right\|_{l,2}^2\right)\,ds+\lambda_{2G}\sum_{i=0}^{k-1}\|\eta_{\bm{u}}^{i}\|_0^2.\label{fully:esti12}
\end{align}
By invoking (\ref{m0}) and the assumptions in  Hypothesis \ref{assum:solution2} imposed on the exact solution, we derive the following
\begin{align*}
&\|\nabla\bm{u}_h^m\|_0\leq C_*,\,\, \sum_{n=k-1}^{m}\nu \tau \|\mathcal{A}_{h}\bm{u}_h^{n}\|_0^2\leq C_*, \,\, \|\bm{u}^{n+1}\|_{2,2}^2\leq C_1.
\end{align*}
Applying Hypothesis \ref{initial value estimate} and Gr$\ddot{\text{o}}$nwall Lemma \ref{def:gronwall2} to (\ref{fully:esti12}) then yields
\begin{align}
&\|\eta_{\bm{u}}^{m+1}\|_0^2+\tau \frac{1-\mu_k^2}{2}\sum_{n=k-1}^{m}\nu\|\nabla\eta_{\bm{u}}^{n+1}\|_0^2\nonumber\\[2mm]
%\leq&C\tau \sum_{n=1}^md_n\Big(\|\nabla \hat{e}_{\bm{u}}^n\|_0^2+\|\nabla \bar{e}_{\bm{u}}^n\|_0^2+\|\nabla \bar{e}_{\bm{u}}^{n-1}\|_0^2+\|\nabla\tilde{e}_{\bm{u}}^{n}\|_0^2\Big)\\
%&+ CT(\tau )^4\\
\leq& C\exp\left(CC_*+CC_1\right)\left(T\tau^{2k}+h^{2(l+1)}\right)\leq C_u\left(\tau^{2k}+h^{2(l+1)}\right). \label{fully:esti13}
\end{align}

 For $k=6$, from equation (\ref{eq:erroreq2}),  it can be deduced that
 \begin{align}
&\tau\left(q_h,\nabla\cdot\left({\eta}_{\bm{u}}^{n+1}-\frac{13}{9}\eta_{\bm{u}}^{n}+\frac{25}{36}\eta_{\bm{u}}^{n-1}-\frac{1}{9}\eta_{\bm{u}}^{n-2}\right)\right)=0. \label{eq:erroreq02k6}
\end{align}
 Choosing $\bm{v}_h=\hat{\eta}_{\bm{u}}^{n+1}$ in (\ref{eq:erroreq1}) with $\hat{\eta}_{\bm{u}}^{n+1}={\eta}_{\bm{u}}^{n+1}-\frac{13}{9}\eta_{\bm{u}}^{n}+\frac{25}{36}\eta_{\bm{u}}^{n-1}-\frac{1}{9}\eta_{\bm{u}}^{n-2}$,  $q_h={\eta}_{p}^{n+1}$ in (\ref{eq:erroreq02k6}), and adding up these two  equations yields the following
\begin{align}
&(\alpha_6\eta_{\bm{u}}^{n+1}-\beta_6(\eta_{\bm{u}}^{n}),\hat{\eta}_{\bm{u}}^{n+1})+\nu\tau(\nabla\eta_{\bm{u}}^{n+1},\nabla\hat{\eta}_{\bm{u}}^{n+1})%\nonumber\\
=(\tau E_{6,1}+E_{6,2}+\tau E_{6,3},\hat{\eta}_{\bm{u}}^{n+1}). \label{fully:esti0k6}
\end{align}
For the left-hand side of (\ref{fully:esti0k6}), with the help of Lemma \ref{lemma_symmetricmatrix1}, it can be concluded that 
\begin{align}
(\alpha_6\eta_{\bm{u}}^{n+1}-\beta_6(\eta_{\bm{u}}^{n}), \hat{\eta}_{\bm{u}}^{n+1})
\ge&\sum_{i,j=1}^6g_{ij}(\eta_{\bm{u}}^{n+i-5}, \eta_{\bm{u}}^{n+j-5})
-\sum_{i,j=1}^6g_{ij}(\eta_{\bm{u}}^{n+i-6}, \eta_{\bm{u}}^{n+j-6})\nonumber\\[2mm]
=& | \mathbb{U}_{\eta}^{n+1}|_{G}^2-| \mathbb{U}_{\eta}^{n}|_{G}^2.\label{fully:esti1k6}
\end{align}
%and
%\begin{align}
%&(-\nu\Delta\bar{e}_{\bm{u}}^{n+1}, \bar{e}_{\bm{u}}^{n+1}-\mu_k\bar{e}_{\bm{u}}^{n})
%=\nu\|\nabla\bar{e}_{\bm{u}}^{n+1}\|_0^2-\nu(\nabla\bar{e}_{\bm{u}}^{n+1}, \mu_k\nabla\bar{e}_{\bm{u}}^{n}).\label{fully:esti2}
%\end{align}
Here the norm $| \mathbb{U}_{\eta}^{n+1}|_{G}^2$ is given by
\begin{align*}
 | \mathbb{U}_{\eta}^{n+1}|_{G}^2=\sum_{i,j=1}^6g_{ij}(\eta_{\bm{u}}^{n+i-5},\eta_{\bm{u}}^{n+j-5}),
\end{align*}
for the notation $\mathbb{U}_{\eta}^{n+1}=(\eta_{\bm{u}}^{n-4},\eta_{\bm{u}}^{n-3}, \eta_{\bm{u}}^{n-2}, \eta_{\bm{u}}^{n-1}, \eta_{\bm{u}}^{n},\eta_{\bm{u}}^{n+1})^{\top}$.  By inserting (\ref{fully:esti1k6}) into  (\ref{fully:esti0k6}), we obtain
\begin{align}
& | \mathbb{U}_{\eta}^{n+1}|_{G}^2-| \mathbb{U}_{\eta}^{n}|_{G}^2+\nu\tau(\nabla\eta_{\bm{u}}^{n+1},\nabla\hat{\eta}_{\bm{u}}^{n+1})%\nonumber\\
\leq (\tau E_{6,1}+E_{6,2}+\tau E_{6,3},\hat{\eta}_{\bm{u}}^{n+1}). \label{fully:esti00k6}
\end{align}
We bound the terms on the right-hand side of (\ref{fully:esti00k6}) using (\ref{bestimate}).  % and (\ref{un1andgammaun}), 
Then
\begin{align*}
&\left|a(-\bm{u}^{n+1}+\gamma_6(\bm{u}^n)-\gamma_6(\bm{u}^n-\bm{u}_h^n),\bm{u}^{n+1}, \hat{\eta}_{\bm{u}}^{n+1})\right|\\[2mm]
\leq&\left(\|\bm{u}^{n+1}-\gamma_6(\bm{u}^n)\|_{0}+\| \gamma_6(\bm{u}^n-\bm{u}_h^n)\|_{0}\right)\|\bm{u}^{n+1}\|_{2,2}\|\hat{\eta}_{\bm{u}}^{n+1}\|_{1,2}\\[2mm]
\leq&C(\varepsilon\nu)^{-1}\tau^{11}\|\bm{u}^{n+1}\|_{2,2}^2 \sum_{i=1}^6 \int_{t^{n+1-i}}^{t^{n+1}}\left\|\frac{\partial^{6} \bm{u}}{\partial t^{6}}(s)\right\|_0^2\,ds+C(\varepsilon\nu)^{-1}\left(\| \gamma_6(\eta_{\bm{u}}^n)\|_0^2+\| \gamma_6(e_{\bm{u}}^n)\|_{0}^2\right)\|\bm{u}^{n+1}\|_{2,2}^2\\[2mm]
&+\varepsilon\nu\|\nabla\eta_{\bm{u}}^{n+1}\|_0^2+\varepsilon\nu\left[\|\nabla\eta_{\bm{u}}^{n}\|_0^2+\|\nabla\eta_{\bm{u}}^{n-1}\|_0^2+\|\nabla\eta_{\bm{u}}^{n-2}\|_0^2\right],
\end{align*}
%where $b_i$ are some fixed and bounded constants determined by the truncation error.
 here, we have used the following identity
%\begin{align*}
%&\bm{u}^{n+1}-\gamma_k(\bm{u}^n)
%=3(\bm{u}^{n+1}-\bm{u}^n)-3(\bm{u}^{n+1}-\bm{u}^{n-1})+(\bm{u}^{n+1}-\bm{u}^{n-2})\\
%=&\frac{3}{2}\int_{t^{n}}^{t^{n+1}}(t^{n}-s)^{2}\frac{\partial^{3} \bm{u}}{\partial t^{3}}(s)\,ds-\frac{3}{2}\int_{t^{n-1}}^{t^{n+1}}(t^{n-1}-s)^{2}\frac{\partial^{3} \bm{u}}{\partial t^{3}}(s)\,ds\\
%&+\frac{1}{2}\int_{t^{n-2}}^{t^{n+1}}(t^{n-2}-s)^{2}\frac{\partial^{3} \bm{u}}{\partial t^{3}}(s)\,ds.
%\end{align*}
%To facilitate the subsequent error analysis, we derive the following high-order temporal identity, which will be repeatedly employed throughout the analysis and plays a crucial role in the derivation of the error estimates:
%for the case $k=6$, there holds
\begin{align}
&\bm{u}^{n+1}-\gamma_6(\bm{u}^n)\nonumber\\[2mm]
%=&6(\bm{u}^{n+1}-\bm{u}^n)-15(\bm{u}^{n+1}-\bm{u}^{n-1})+20(\bm{u}^{n+1}-\bm{u}^{n-2})\nonumber\\
%&-15(\bm{u}^{n+1}-\bm{u}^{n-3})+6(\bm{u}^{n+1}-\bm{u}^{n-4})-(\bm{u}^{n+1}-\bm{u}^{n-5})\nonumber\\
=&-\frac{1}{20}\int_{t^{n}}^{t^{n+1}}(s-t^{n})^{5}\frac{\partial^{6} \bm{u}}{\partial t^{6}}(s)\,ds+\frac{1}{8}\int_{t^{n-1}}^{t^{n+1}}(s-t^{n-1})^{5}\frac{\partial^{6} \bm{u}}{\partial t^{6}}(s)\,ds\nonumber\\[2mm]
&-\frac{1}{6}\int_{t^{n-2}}^{t^{n+1}}(s-t^{n-2})^{5}\frac{\partial^{6} \bm{u}}{\partial t^{6}}(s)\,ds+\frac{1}{8}\int_{t^{n-3}}^{t^{n+1}}(s-t^{n-3})^{5}\frac{\partial^{6} \bm{u}}{\partial t^{6}}(s)\,ds\nonumber\\[2mm]
&-\frac{1}{20}\int_{t^{n-4}}^{t^{n+1}}(s-t^{n-4})^{5}\frac{\partial^{6} \bm{u}}{\partial t^{6}}(s)\,ds+\frac{1}{120}\int_{t^{n-5}}^{t^{n+1}}(s-t^{n-5})^{5}\frac{\partial^{6} \bm{u}}{\partial t^{6}}(s)\,ds.\label{un1andgammaun}
\end{align}
Similarly, we can derive that
\begin{align*}
&\left|a(\gamma_6(\bm{u}_h^{n}),\bm{u}^{n+1}-\gamma_6(\bm{u}^n),\hat{\eta}_{\bm{u}}^{n+1})\right|\\[2mm]
%\leq&\|\gamma_k(\bm{u}_h^{n})\|_{1,2}\|\bm{u}^{n+1}-\gamma_k(\bm{u}^n)\|_{1,2}\|\eta_{\bm{u}}^{n+1}-\mu_k\eta_{\bm{u}}^{n}\|_{1,2}\\
%\leq& C\|\gamma_k(\bm{u}_h^{n})\|_{1,2}^2\Big\|\sum_{i=1}^kb_i\int_{t^{n+1-i}}^{t^{n+1}}(t^{n+1-i}-s)^{k-1}\frac{\partial^{k} \bm{u}}{\partial t^{k}}(s)\,ds\Big\|_{1,2}^2
%+\varepsilon\nu(\|\nabla\eta_{\bm{u}}^{n+1}\|_0^2+\mu_k^2\|\nabla\eta_{\bm{u}}^{n}\|_0^2)\\
\leq& C(\varepsilon\nu)^{-1} \tau^{11}\|\gamma_6(\bm{u}_h^{n})\|_{1,2}^2\sum_{i=1}^6\int_{t^{n+1-i}}^{t^{n+1}}\left\|\frac{\partial^{6} \bm{u}}{\partial t^{6}}(s)\right\|_{1,2}^2\,ds
+\varepsilon\nu\|\nabla\eta_{\bm{u}}^{n+1}\|_0^2\\[2mm]
&+\varepsilon\nu\left[\|\nabla\eta_{\bm{u}}^{n}\|_0^2+\|\nabla\eta_{\bm{u}}^{n-1}\|_0^2+\|\nabla\eta_{\bm{u}}^{n-2}\|_0^2\right],
\end{align*}
and
\begin{align*}
&\left|a(\gamma_6(\bm{u}_h^{n}),\gamma_6(\bm{u}^n-\bm{u}_h^{n}), \hat{\eta}_{\bm{u}}^{n+1})\right|\\[2mm]
\leq&\|\mathcal{A}_{h}(\gamma_6(\bm{u}_h^{n}))\|_{0}\|\gamma_6(\eta_{\bm{u}}^{n})\|_0\|\hat{\eta}_{\bm{u}}^{n+1}\|_{1,2}
+\|\mathcal{A}_{h}(\gamma_6(\bm{u}_h^{n}))\|_{0}\|\gamma_6(e_{\bm{u}}^{n})\|_0\|\hat{\eta}_{\bm{u}}^{n+1}\|_{1,2}\\[2mm]
\leq&C(\varepsilon\nu)^{-1}\left(\|\mathcal{A}_{h}(\gamma_6(\bm{u}_h^{n}))\|_{0}^2\|\gamma_6(\eta_{\bm{u}}^n)\|_0^2+\|\mathcal{A}_{h}(\gamma_6(\bm{u}_h^{n}))\|_{0}^2\|\gamma_6(e_{\bm{u}}^{n})\|_0^2\right)\\[2mm]
&+\varepsilon\nu\|\nabla\eta_{\bm{u}}^{n+1}\|_0^2+\varepsilon\nu\left[\|\nabla\eta_{\bm{u}}^{n}\|_0^2+\|\nabla\eta_{\bm{u}}^{n-1}\|_0^2+\|\nabla\eta_{\bm{u}}^{n-2}\|_0^2\right].
\end{align*}
Combining these estimates, we can conclude that
\begin{align}
\left|(E_{6,1}, \hat{\eta}_{\bm{u}}^{n+1})\right|\leq&3\varepsilon\nu\left[\|\nabla\eta_{\bm{u}}^{n+1}\|_0^2+\|\nabla\eta_{\bm{u}}^{n}\|_0^2+\|\nabla\eta_{\bm{u}}^{n-1}\|_0^2+\|\nabla\eta_{\bm{u}}^{n-2}\|_0^2\right]\nonumber\\[2mm]
&+ C(\varepsilon\nu)^{-1} \tau^{11}\|\bm{u}^{n+1}\|_{2,2}^2\sum_{i=1}^6\int_{t^{n+1-i}}^{t^{n+1}}\left\|\frac{\partial^{6} \bm{u}}{\partial t^{6}}(s)\right\|_{0}^2\,ds\nonumber\\[2mm]
&+ C(\varepsilon\nu)^{-1} \tau^{11}\|\gamma_6(\bm{u}_h^{n})\|_{1,2}^2\sum_{i=1}^6 \int_{t^{n+1-i}}^{t^{n+1}}\left\|\frac{\partial^{6} \bm{u}}{\partial t^{6}}(s)\right\|_{1,2}^2\,ds\nonumber\\[2mm]
&+C(\varepsilon\nu)^{-1}\left(\|\mathcal{A}_{h}(\gamma_6(\bm{u}_h^{n}))\|_{0}^2+\|\bm{u}^{n+1}\|_{2,2}^2\right)\|\gamma_6(\eta_{\bm{u}}^n)\|_0^2\nonumber\\[2mm]
&+C(\varepsilon\nu)^{-1}\left(\|\mathcal{A}_{h}(\gamma_6(\bm{u}_h^{n}))\|_{0}^2+\|\bm{u}^{n+1}\|_{2,2}^2\right)\|\gamma_6(e_{\bm{u}}^{n})\|_0^2.\label{fully:esti4k6}
\end{align}
We continue to estimate
\begin{align}
\left|(E_{6,2}, \hat{\eta}_{\bm{u}}^{n+1})\right|\leq &\varepsilon\nu\tau\left[\|\nabla\eta_{\bm{u}}^{n+1}\|_0^2+\|\nabla\eta_{\bm{u}}^{n}\|_0^2+\|\nabla\eta_{\bm{u}}^{n-1}\|_0^2+\|\nabla\eta_{\bm{u}}^{n-2}\|_0^2\right]\nonumber\\[2mm]
 &+C(\varepsilon\nu)^{-1}\tau^{-1}\tau^{13}\sum_{i=1}^6\int_{t^{n+1-i}}^{t^{n+1}}\left\|\frac{\partial^{7} \mathcal{P}_h\bm{u}}{\partial t^{7}}(s)\right\|_0^2\,ds,\label{fully:esti5k6}
 \end{align}
and
\begin{align}
\left|(E_{6,3}, \hat{\eta}_{\bm{u}}^{n+1})\right|\leq &C(\varepsilon\nu)^{-1}\tau h^{2(l+1)} \int_{t^{n-5}}^{t^{n+1}} \left( \left\|\frac{\partial\bm u}{\partial t}(s)\right\|_{l+1,2}^2+ \left\|\frac{\partial p}{\partial t}(s)\right\|_{l,2}^2\right)\,ds\nonumber\\[2mm]
&+\varepsilon\nu\left[\|\nabla\eta_{\bm{u}}^{n+1}\|_0^2+\|\nabla\eta_{\bm{u}}^{n}\|_0^2+\|\nabla\eta_{\bm{u}}^{n-1}\|_0^2+\|\nabla\eta_{\bm{u}}^{n-2}\|_0^2\right].\label{fully:esti6k6}
\end{align}
Now, combining (\ref{fully:esti4k6})-(\ref{fully:esti6k6}) with (\ref{fully:esti00k6}),  and then taking the sum for $n$ from $5$ to $m$, we obtain
\begin{align}
%&\sum_{i,j=1}^6g_{ij}(\nabla{\bm{u}}_h^{n+i-5}, \nabla{\bm{u}}_h^{n+j-5})-\sum_{i,j=1}^6g_{ij}(\nabla{\bm{u}}_h^{n+i-6}, \nabla{\bm{u}}_h^{n+j-6})\nonumber\\
&| \mathbb{U}_{\eta}^{m+1}|_{G}^2+\nu\tau\sum_{n=5}^{m}(\nabla\eta_{\bm{u}}^{n+1},\nabla\hat{\eta}_{\bm{u}}^{n+1})%\nonumber\\
%\end{align}
%\begin{align}
\leq C h^{2(l+1)}\tau^2 \int_{0}^{T} \left( \left\|\frac{\partial\bm u}{\partial t}(s)\right\|_{l+1,2}^2+ \left\|\frac{\partial p}{\partial t}(s)\right\|_{l,2}^2\right)\,ds \nonumber\\[2mm]
&+C\tau\sum_{n=5}^{m}\left(\|\mathcal{A}_{h}(\gamma_6(\bm{u}_h^{n}))\|_{0}^2+\|\bm{u}^{n+1}\|_{2,2}^2\right)\left(\|\gamma_6(\eta_{\bm{u}}^n)\|_0^2+\|\gamma_6(e_{\bm{u}}^n)\|_0^2\right)\nonumber \\[2mm]
&+5\varepsilon\nu\tau\sum_{n=5}^{m} \left[\|\nabla\eta_{\bm{u}}^{n+1}\|_0^2+\|\nabla\eta_{\bm{u}}^{n}\|_0^2+\|\nabla\eta_{\bm{u}}^{n-1}\|_0^2+\|\nabla\eta_{\bm{u}}^{n-2}\|_0^2\right]\nonumber\\[2mm]
&+C\tau^{12}\int_{0}^{T}\left(\left\|\frac{\partial^{6} \bm{u}}{\partial t^{6}}(s)\right\|_0^2+\left\|\frac{\partial^{6} \bm{u}}{\partial t^{6}}(s)\right\|_{1,2}^2+\left\|\frac{\partial^{7} \bm{u}}{\partial t^{7}}(s)\right\|_0^2\right)\,ds+|  \mathbb{U}_{\eta}^{5}|_{G}^2.  \label{fully:esti110k6}
\end{align}
Similar to (\ref{fully:uh3k3}), we can easily derive the following
\begin{align}
\nu\sum_{n=5}^m(\nabla\eta_{\bm{u}}^{n+1},\nabla\hat{\eta}_{\bm{u}}^{n+1})\ge& \nu\frac{1}{32}\sum_{n=5}^m\|\nabla\eta_{\bm{u}}^{n+1}\|_0^2%\nonumber\\
-\nu\left(\nabla\eta_{\bm{u}}^{6},  \frac{13}{9}\nabla\eta_{\bm{u}}^{5}-\frac{25}{36}\nabla\eta_{\bm{u}}^{4}+\frac{1}{9}\nabla\eta_{\bm{u}}^{3}\right)\nonumber\\[2mm]
&-\nu\left(\nabla\eta_{\bm{u}}^{7}, -\frac{25}{36}\nabla\eta_{\bm{u}}^{5}+\frac{1}{9}\nabla\eta_{\bm{u}}^{4}\right)-\nu\left(\nabla\eta_{\bm{u}}^{8}, \frac{1}{9}\nabla\eta_{\bm{u}}^{5}\right). \label{fully:esti9k6}
\end{align}
 By invoking Young's inequality, there holds
\begin{align}
&\left|\nu\left(\nabla\eta_{\bm{u}}^{6},  \frac{13}{9}\nabla\eta_{\bm{u}}^{5}-\frac{25}{36}\nabla\eta_{\bm{u}}^{4}+\frac{1}{9}\nabla\eta_{\bm{u}}^{3}\right)\right|\nonumber\\[2mm]
\leq & 3\varepsilon\nu\|\nabla\eta_{\bm{u}}^{6}\|_0^2+C(\varepsilon\nu)^{-1}\left[\|\nabla\eta_{\bm{u}}^{5}\|_0^2+\|\nabla\eta_{\bm{u}}^{4}\|_0^2+\|\nabla\eta_{\bm{u}}^{3}\|_0^2\right]  \label{fully:esti9k600}\\[2mm]
&\left|\nu\left(\nabla\eta_{\bm{u}}^{7}, -\frac{25}{36}\nabla\eta_{\bm{u}}^{5}+\frac{1}{9}\nabla\eta_{\bm{u}}^{4}\right)\right|\nonumber\\[2mm]
\leq& 2 \varepsilon\nu\|\nabla\eta_{\bm{u}}^{7}\|_0^2+C(\varepsilon\nu)^{-1}\left[\|\nabla\eta_{\bm{u}}^{5}\|_0^2+\|\nabla\eta_{\bm{u}}^{4}\|_0^2\right] \label{fully:esti9k601}\\[2mm]
&\left|\nu\left(\nabla\eta_{\bm{u}}^{8}, \frac{1}{9}\nabla\eta_{\bm{u}}^{5}\right)\right|\leq \varepsilon\nu\|\nabla\eta_{\bm{u}}^{8}\|_0^2+C(\varepsilon\nu)^{-1}\|\nabla\eta_{\bm{u}}^{5}\|_0^2.\label{fully:esti9k602}
\end{align}
Now, combining (\ref{fully:esti9k6})-(\ref{fully:esti9k602}) with (\ref{fully:esti110k6}), we derive
\begin{align}
&| \mathbb{U}_{\eta}^{m+1}|_{G}^2+\nu\tau\frac{1}{32} \sum_{n=5}^{m}\|\nabla\eta_{\bm{u}}^{n+1}\|_0^2\nonumber\\[2mm]
\leq &|  \mathbb{U}_{\eta}^{5}|_{G}^2+Ch^{2(l+1)} \tau^2\int_{0}^{T} \left( \left\|\frac{\partial\bm u}{\partial t}(s)\right\|_{l+1,2}^2+ \left\|\frac{\partial p}{\partial t}(s)\right\|_{l,2}^2\right)\,ds \nonumber\\[2mm]
&+C\tau\sum_{n=5}^{m}\left(\|\mathcal{A}_{h}(\gamma_6(\bm{u}_h^{n}))\|_{0}^2+\|\bm{u}^{n+1}\|_{2,2}^2\right)\left(\|\gamma_6(\eta_{\bm{u}}^n)\|_0^2+\|\gamma_6(e_{\bm{u}}^n)\|_0^2\right)\nonumber \\[2mm]
&+5\varepsilon\nu\tau\sum_{n=5}^{m} \left[\|\nabla\eta_{\bm{u}}^{n+1}\|_0^2+\|\nabla\eta_{\bm{u}}^{n}\|_0^2+\|\nabla\eta_{\bm{u}}^{n-1}\|_0^2+\|\nabla\eta_{\bm{u}}^{n-2}\|_0^2\right]\nonumber\\[2mm]
&+C\tau^{12}\int_{0}^{T}\left(\left\|\frac{\partial^{6} \bm{u}}{\partial t^{6}}(s)\right\|_0^2+\left\|\frac{\partial^{6} \bm{u}}{\partial t^{6}}(s)\right\|_{1,2}^2+\left\|\frac{\partial^{7} \bm{u}}{\partial t^{7}}(s)\right\|_0^2\right)\,ds \nonumber\\[2mm]
&+\varepsilon\nu\tau \left[ \|\nabla\eta_{\bm{u}}^{8}\|_0^2+2\|\nabla\eta_{\bm{u}}^{7}\|_0^2+3\|\nabla\eta_{\bm{u}}^{6}\|_0^2\right] +C(\varepsilon\nu)^{-1}\tau\left[3\|\nabla\eta_{\bm{u}}^{5}\|_0^2+2\|\nabla\eta_{\bm{u}}^{4}\|_0^2+\|\nabla\eta_{\bm{u}}^{3}\|_0^2 \right]. \label{fully:esti11k6}
\end{align}
Next,  we can choose $\varepsilon$ small enough such that 
$\frac{1}{64}>26\varepsilon$.
Using Hypothesis \ref{initial value estimate},  and noting  that $G=(g_{ij})$  is a symmetric positive definite matrix with smallest  eigenvalue $\lambda_{3g}$ and largest eigenvalue $\lambda_{3G}$, we then obtain
\begin{align}
&\lambda_{3g}\|\eta_{\bm{u}}^{m+1}\|_0^2+\tau\frac{1}{64}\sum_{n=5}^{m}\nu\|\nabla\eta_{\bm{u}}^{n+1}\|_0^2\nonumber\\[2mm]
%\leq &\sum_{i,j=1}^kg_{ij}(\eta_{\bm{u}}^{m+1+i-k}, \eta_{\bm{u}}^{m+1+j-k})+\tau\frac{1-\mu_k^2}{2}\sum_{n=0}^{m+1}\nu\|\nabla\eta_{\bm{u}}^{n}\|_0^2\nonumber\\
\leq&C\tau\sum_{n=5}^m\left(\|\mathcal{A}_{h}(\gamma_6(\bm{u}_h^n))\|_{0}^2+\|\bm{u}^{n+1}\|_{2,2}^2\right)\|\gamma_6(\eta_{\bm{u}}^n)\|_0^2\nonumber\\[2mm]
&+C\tau\sum_{n=5}^m\left(\|\mathcal{A}_{h}(\gamma_6(\bm{u}_h^n))\|_{0}^2+\|\bm{u}^{n+1}\|_{2,2}^2\right)\|\gamma_6(e_{\bm{u}}^n)\|_0^2\nonumber\\[2mm]
&+ C\tau^{12}\int_{0}^{T}\left(\left\|\frac{\partial^{6} \bm{u}}{\partial t^{6}}(s)\right\|_0^2+\left\|\frac{\partial^{6} \bm{u}}{\partial t^{6}}(s)\right\|_{1,2}^2+\left\|\frac{\partial^{7} \bm{u}}{\partial t^{7}}(s)\right\|_0^2\right)\,ds+Ch^{2(l+1)}+C\tau^{12}\nonumber\\[2mm]
&+Ch^{2(l+1)}\tau^2 \int_{0}^{T} \left( \left\|\frac{\partial\bm u}{\partial t}(s)\right\|_{l+1,2}^2+ \left\|\frac{\partial p}{\partial t}(s)\right\|_{l,2}^2\right)\,ds+\lambda_{3G}\sum_{i=0}^5\|\eta_{\bm{u}}^{i}\|_0^2.\label{fully:esti12k6}
\end{align}
By invoking (\ref{lemmadeltah}) and the assumptions in  Hypothesis \ref{assum:solution2} imposed on the exact solution, we derive the following
\begin{align*}
&\|\nabla\bm{u}_h^{m+1}\|_0\leq C_*,\,\, \sum_{n=5}^{m}\nu \tau\|\mathcal{A}_{h}\bm{u}_h^{n+1}\|_0^2\leq C_*, \,\, \|\bm{u}^{n+1}\|_{2,2}^2\leq C_1.
\end{align*}
Applying Gr$\ddot{\text{o}}$nwall Lemma \ref{def:gronwall2} to (\ref{fully:esti12k6}) then yields
\begin{align}
&\|\eta_{\bm{u}}^{m+1}\|_0^2+\tau\frac{1}{64}\sum_{n=5}^{m}\nu\|\nabla\eta_{\bm{u}}^{n+1}\|_0^2\nonumber\\
%\leq&C\tau\sum_{n=1}^md_n\Big(\|\nabla \hat{e}_{\bm{u}}^n\|_0^2+\|\nabla \bar{e}_{\bm{u}}^n\|_0^2+\|\nabla \bar{e}_{\bm{u}}^{n-1}\|_0^2+\|\nabla\tilde{e}_{\bm{u}}^{n}\|_0^2\Big)\\
%&+ CT(\tau)^4\\
\leq& C\exp\left(CC_*+CC_1\right)\left(T\tau^{12}+h^{2(l+1)}\right)\leq C_u\left(\tau^{12}+h^{2(l+1)}\right). \label{fully:esti13k6}
\end{align}
By applying the triangle inequality on (\ref{fully:esti13}) and (\ref{fully:esti13k6}), and the approximation properties of the projection operators, the desired result is now the conclusion of  (\ref{space discrete : theorem 1}). 
%\end{proof}

%By virtue of Hypothesis \ref{assum:solution2}  and Hypothesis \ref{assum:solution3}, we are ready to show the optimal $H^1$ convergence result for our BDF-$k$ in full  discrete setting.
%\begin{theorem}\label{space discrete : theorem 2}
%Suppose that the continuous system (\ref{weak:eq1})-(\ref{weak:eq2}) has a unique solution  $(\bm{u}^n, p^n)$ satisfying  Hypothesis \ref{assum:solution2}  and Hypothesis \ref{assum:solution3}. Then there exist a positive constant $\tau _0$ such that when $\tau \leq \tau _0$,  the finite element problem (\ref{weak space discrete:eq1})-(\ref{weak space discrete:eq2}) admits a unique solution  $(\bm{u}_h^n,p_h^n)$, which satisfies, for all $m+1\leq N$ and $k=1,\cdots,5$
%\begin{equation*}
%\begin{split}
%&\Vert\nabla\bm{u}^{m+1}-\nabla\bm{u}_h^{m+1}\Vert_0
%\leq  C_u\left(h^{l}+(\tau )^{k}\right),
%\end{split}
%\end{equation*}
%where the constant $C_u$ depends on $T,\nu,\Omega$  and the exact solution $\bm{u}$, but is independent
%of $\tau $ and $h$.
%\end{theorem}
%\begin{proof}
\textbf{Step III}: According to the definitions of $\alpha_k$ and $\beta_k$,  it can be deduced that by  equation (\ref{eq:erroreq2}) 
 \begin{align}
&\tau (q_h,\nabla\cdot(\alpha_k{\eta}_{\bm{u}}^{n+1}-\beta_k(\eta_{\bm{u}}^{n})))=0.\label{eq:erroreq22}
%&\tau (q_h,\nabla\cdot(\alpha_k{\eta}_{\bm{u}}^{n}-\beta_k(\eta_{\bm{u}}^{n-1})))=0. \label{eq:erroreq22new}
\end{align}
%In addition to Hypotheses~2.1--2.2, we assume that
%\begin{equation}
%  \partial_t p\in L^2(0,T;H^l(\Omega)).
%  \label{eq:additional-pressure-regularity}
%\end{equation}
%This assumption is needed because the velocity component of the Stokes
%projection depends on both $u$ and $p$.
For $1\leq k\leq 5$,  we test (\ref{eq:erroreq1}) with $\bm{v}_h=\frac{\alpha_k\eta_{\bm{u}}^{n+1}-\beta_k(\eta_{\bm{u}}^{n})}{\tau }$ and test  (\ref{eq:erroreq22}) with $q_h=\frac{1}{\tau }{\eta}_{p}^{n+1}$. Adding the resulting identities and canceling the pressure terms, we obtain
\begin{align}
&\left\|\frac{\alpha_k\eta_{\bm{u}}^{n+1}-\beta_k(\eta_{\bm{u}}^{n})}{\tau }\right\|_0^2+\frac{\nu}{\tau }(\nabla\eta_{\bm{u}}^{n+1},\nabla(\alpha_k\eta_{\bm{u}}^{n+1}-\beta_k(\eta_{\bm{u}}^{n})))\nonumber\\[2mm]
&=\left( E_{k,1}+\frac{1}{\tau }E_{k,2}+E_{k,3},\frac{\alpha_k\eta_{\bm{u}}^{n+1}-\beta_k(\eta_{\bm{u}}^{n})}{\tau }\right). \label{fully:esti000new0}
\end{align}
Next, at time level $t_n$, we test (\ref{eq:erroreq1}) with $\bm{v}_h=-\mu_k\frac{\alpha_k\eta_{\bm{u}}^{n+1}-\beta_k(\eta_{\bm{u}}^{n})}{\tau }$. Using the corresponding discrete divergence relation to eliminate the
pressure term, we obtain
\begin{align}
&-\mu_k\left( \frac{\alpha_k\eta_{\bm{u}}^{n}-\beta_k(\eta_{\bm{u}}^{n-1})}{\tau }, \frac{\alpha_k\eta_{\bm{u}}^{n+1}-\beta_k(\eta_{\bm{u}}^{n})}{\tau }\right)-\frac{\nu}{\tau }(\mu_k\nabla\eta_{\bm{u}}^{n},\nabla(\alpha_k\eta_{\bm{u}}^{n+1}-\beta_k(\eta_{\bm{u}}^{n})))\nonumber\\[2mm]
&=-\mu_k\left( \tilde{E}_{k,1}+\frac{1}{\tau } \tilde{E}_{k,2}+ \tilde{E}_{k,3},\frac{\alpha_k\eta_{\bm{u}}^{n+1}-\beta_k(\eta_{\bm{u}}^{n})}{\tau }\right). \label{fully:esti000new1}
\end{align}
Here, the residual terms at time $t_n$ are defined by
\begin{align*}
\tilde{E}_{k,1}=&-\bm{u}^{n}\cdot\nabla\bm{u}^{n}+\gamma_k(\bm{u}_h^{n-1})\cdot\nabla\gamma_k(\bm{u}_h^{n-1}),
\end{align*}
\begin{align*}
\tilde{E}_{k,2}=&\alpha_k\bm{u}^{n}-\beta_k(\bm{u}^{n-1})-\tau \frac{\partial\bm{u}^{n}}{\partial t},
\end{align*}
and
\begin{align*}
\tilde{E}_{k,3}=\alpha_k\mathcal{P}_h\bm{u}^{n}-\beta_k(\mathcal{P}_h\bm{u}^{n-1})-[\alpha_k\bm{u}^{n}-\beta_k(\bm{u}^{n-1})].
\end{align*}
For the sake of simplicity, set $d_{\bm{u}}^{n+1}=\frac{\alpha_k\eta_{\bm{u}}^{n+1}-\beta_k(\eta_{\bm{u}}^{n})}{\tau }$. Combining (\ref{fully:esti000new0}) and (\ref{fully:esti000new1}) gives
\begin{align}
&\left\|d_{\bm{u}}^{n+1}\right\|_0^2-\mu_k\left( d_{\bm{u}}^{n}, d_{\bm{u}}^{n+1}\right)%\nonumber\\[2mm]
+\frac{\nu}{\tau }(\nabla(\eta_{\bm{u}}^{n+1}-\mu_k\eta_{\bm{u}}^{n}),\nabla(\alpha_k\eta_{\bm{u}}^{n+1}-\beta_k(\eta_{\bm{u}}^{n})))\nonumber\\[2mm]
&=\left( E_{k,1}+\frac{1}{\tau }E_{k,2}+E_{k,3}-\mu_k\left( \tilde{E}_{k,1}+\frac{1}{\tau } \tilde{E}_{k,2}+ \tilde{E}_{k,3}\right),d_{\bm{u}}^{n+1}\right). \label{fully:esti000}
\end{align}
For the left-hand side of (\ref{fully:esti000}), with the help of Lemma \ref{lemma_symmetricmatrix}, it can be concluded that 
\begin{align}
&(\nabla(\alpha_k\eta_{\bm{u}}^{n+1}-\beta_k(\eta_{\bm{u}}^{n})), \nabla(\eta_{\bm{u}}^{n+1}-\mu_k\eta_{\bm{u}}^{n}))\nonumber\\[2mm]
=&\sum_{i,j=1}^kg_{ij}(\nabla\eta_{\bm{u}}^{n+1+i-k}, \nabla\eta_{\bm{u}}^{n+1+j-k})%\nonumber\\
-\sum_{i,j=1}^kg_{ij}(\nabla\eta_{\bm{u}}^{n+i-k}, \nabla\eta_{\bm{u}}^{n+j-k})+\left\|\sum_{i=0}^k\delta_{i}\nabla\eta_{\bm{u}}^{n+1+i-k}\right\|_0^2.\label{fully:esti11a}
\end{align}
Moreover, Young's inequality yields
\begin{align*}
 \left| -\mu_k\left(d_{\bm{u}}^{n}, d_{\bm{u}}^{n+1}\right)\right|\leq \frac{1}{2}\left\|d_{\bm{u}}^{n+1}\right\|_0^2+\frac{\mu_k^2}{2}\left\|d_{\bm{u}}^{n}\right\|_0^2.
 \end{align*}
We bound the terms on the right-hand side of  (\ref{fully:esti000}) using (\ref{bestimate}). Then  
\begin{align*}
&\left|a\left(-\bm{u}^{n+1}+\gamma_k(\bm{u}^n)-\gamma_k(\bm{u}^n-\bm{u}_h^n),\bm{u}^{n+1}, d_{\bm{u}}^{n+1}\right)\right|\\[2mm]
\leq&\left(\|\bm{u}^{n+1}-\gamma_k(\bm{u}^n)\|_{1,2}+\| \gamma_k(\bm{u}^n-\bm{u}_h^n)\|_{1,2}\right)\|\bm{u}^{n+1}\|_{2,2} \left\|d_{\bm{u}}^{n+1}\right\|_{0}\\[2mm]
%\leq&C\varepsilon^{-1}\|\bm{u}^{n+1}\|_{2,2}^2 \left\|\sum_{i=1}^kb_i\int_{t^{n+1-i}}^{t^{n+1}}(t^{n+1-i}-s)^{k-1}\frac{\partial^{k} \bm{u}}{\partial t^{k}}(s)\,ds\right\|_{1,2}^2\\
%&+2\varepsilon \left\|\frac{\alpha_k\eta_{\bm{u}}^{n+1}-\beta_k(\eta_{\bm{u}}^{n})}{\tau }\right\|_{0}^2+C\varepsilon^{-1}\left(\| \gamma_k(\eta_{\bm{u}}^n)\|_{1,2}^2+\| \gamma_k(e_{\bm{u}}^n)\|_{1,2}^2\right)\|\bm{u}^{n+1}\|_{2,2}^2\\
\leq&C\varepsilon^{-1}\|\bm{u}^{n+1}\|_{2,2}^2 \tau^{2k-1} \sum_{i=1}^ko_i\int_{t^{n+1-i}}^{t^{n+1}}\left\|\frac{\partial^{k} \bm{u}}{\partial t^{k}}(s)\right\|_{1,2}^2\,ds\\[2mm]
&+2\varepsilon \left\|d_{\bm{u}}^{n+1}\right\|_{0}^2+C\varepsilon^{-1}\left(\| \gamma_k(\eta_{\bm{u}}^n)\|_{1,2}^2+\| \gamma_k(e_{\bm{u}}^n)\|_{1,2}^2\right)\|\bm{u}^{n+1}\|_{2,2}^2.
\end{align*}
%where the coefficients $b_i$ denote certain  fixed and bounded constants  determined by the truncation error. 
 Following similar techniques, one can derive that
\begin{align*}
&\left|a\left(\gamma_k(\bm{u}_h^{n}),\bm{u}^{n+1}-\gamma_k(\bm{u}^n), d_{\bm{u}}^{n+1}\right)\right|\\[2mm]
\leq&\|\gamma_k(\bm{u}_h^{n})\|_{1,2}\|\bm{u}^{n+1}-\gamma_k(\bm{u}^n)\|_{2,2} \left\|d_{\bm{u}}^{n+1}\right\|_{0}\\[2mm]
%\leq& C\varepsilon^{-1}\|\gamma_k(\bm{u}_h^{n})\|_{1,2}^2\|\sum_{i=1}^kb_i\int_{t^{n+1-i}}^{t^{n+1}}(t^{n+1-i}-s)^{k-1}\frac{\partial^{k} \bm{u}}{\partial t^{k}}(s)\,ds\|_{2,2}^2\\
%&+\varepsilon\|\frac{\alpha_k\eta_{\bm{u}}^{n+1}-\beta_k(\eta_{\bm{u}}^{n})}{\tau }\|_0^2\\
\leq& C\varepsilon^{-1}\tau^{2k-1} \|\gamma_k(\bm{u}_h^{n})\|_{1,2}^2\sum_{i=1}^ko_i\int_{t^{n+1-i}}^{t^{n+1}}\left\|\frac{\partial^{k} \bm{u}}{\partial t^{k}}(s)\right\|_{2,2}^2\,ds
+\varepsilon \left\|d_{\bm{u}}^{n+1} \right\|_0^2,
\end{align*}
and
\begin{align*}
&\left|a\left(\gamma_k(\bm{u}_h^{n}),\gamma_k(\bm{u}^n-\bm{u}_h^{n}), d_{\bm{u}}^{n+1}\right)\right|\\[2mm]
\leq&\|\mathcal{A}_{h}(\gamma_k(\bm{u}_h^{n}))\|_{0}\left(\|\gamma_k(\eta_{\bm{u}}^{n})\|_{1,2}+\|\gamma_k(e_{\bm{u}}^{n})\|_{1,2}\right) \left\|d_{\bm{u}}^{n+1}\right\|_{0}\\[2mm]
\leq&C\varepsilon^{-1}\|\mathcal{A}_{h}(\gamma_k(\bm{u}_h^{n}))\|_{0}^2\left(\|\gamma_k(\eta_{\bm{u}}^n)\|_{1,2}^2+\|\gamma_k(e_{\bm{u}}^{n})\|_{1,2}^2\right)
+\varepsilon\left\|d_{\bm{u}}^{n+1}\right\|_{0}^2.
\end{align*}
This implies 
\begin{align}
&\left|\left(E_{k,1}, d_{\bm{u}}^{n+1}\right)\right|\nonumber\\[2mm]
\leq&4\varepsilon \left\|d_{\bm{u}}^{n+1}\right\|_{0}^2%\nonumber\\
+ C\varepsilon^{-1} \tau^{2k-1}\|\bm{u}^{n+1}\|_{2,2}^2\sum_{i=1}^ko_i\int_{t^{n+1-i}}^{t^{n+1}}\left\|\frac{\partial^{k} \bm{u}}{\partial t^{k}}(s)\right\|_{1,2}^2\,ds\nonumber\\[2mm]
&+ C\varepsilon^{-1} \tau^{2k-1}\|\gamma_k(\bm{u}_h^{n})\|_{1,2}^2\sum_{i=1}^ko_i \int_{t^{n+1-i}}^{t^{n+1}}\left\|\frac{\partial^{k} \bm{u}}{\partial t^{k}}(s)\right\|_{2,2}^2\,ds\nonumber\\[2mm]
%&+C\varepsilon^{-1}\left(\|\mathcal{A}_{h}(\gamma_k(\bm{u}_h^{n}))\|_{0}^2+\|\bm{u}^{n+1}\|_{2,2}^2\right)\|\gamma_k(\eta_{\bm{u}}^n)\|_{1,2}^2\nonumber\\
&+C\varepsilon^{-1}\left(\|\mathcal{A}_{h}(\gamma_k(\bm{u}_h^{n}))\|_{0}^2+\|\bm{u}^{n+1}\|_{2,2}^2\right)\left[\|\gamma_k(\eta_{\bm{u}}^n)\|_{1,2}^2+\|\gamma_k(e_{\bm{u}}^{n})\|_{1,2}^2\right].\label{fully:esti44}
\end{align}
We continue to estimate
\begin{align}
\left|\left(E_{k,2},d_{\bm{u}}^{n+1}\right)\right|\leq& \varepsilon \left\|d_{\bm{u}}^{n+1}\right\|_0^2%\nonumber\\
 +C\varepsilon^{-1}\tau^{2k} \sum_{i=1}^ko_i\int_{t^{n+1-k}}^{t^{n+1}}\left\|\frac{\partial^{k+1} \mathcal{P}_h\bm{u}}{\partial t^{k+1}}(s)\right\|_0^2\,ds,\label{fully:esti55}
 \end{align}
and
\begin{align}
&\left|\left(E_{k,3}, d_{\bm{u}}^{n+1}\right) \right|\leq C\varepsilon^{-1}h^{2(l+1)} \tau\int_{t^{n+1-k}}^{t^{n+1}} \left( \left\|\frac{\partial\bm u}{\partial t}(s)\right\|_{l+1,2}^2+ \left\|\frac{\partial p}{\partial t}(s)\right\|_{l,2}^2\right)\,ds+\varepsilon\left\|d_{\bm{u}}^{n+1}\right\|_0^2.\label{fully:esti66}
\end{align}
The residual terms at the preceding time level are treated in the same manner. More precisely,
\begin{align}
&\left|\mu_k\left(\tilde{E}_{k,1}, d_{\bm{u}}^{n+1}\right)\right|\nonumber\\[2mm]
\leq&4\varepsilon \left\|d_{\bm{u}}^{n+1}\right\|_{0}^2%\nonumber\\
+ C\varepsilon^{-1}\mu_k^2 \tau^{2k-1}\|\bm{u}^{n}\|_{2,2}^2\sum_{i=1}^ko_i\int_{t^{n-i}}^{t^{n}}\left\|\frac{\partial^{k} \bm{u}}{\partial t^{k}}(s)\right\|_{1,2}^2\,ds\nonumber\\[2mm]
&+ C\varepsilon^{-1} \tau^{2k-1}\|\gamma_k(\bm{u}_h^{n-1})\|_{1,2}^2\sum_{i=1}^ko_i \int_{t^{n-i}}^{t^{n}}\left\|\frac{\partial^{k} \bm{u}}{\partial t^{k}}(s)\right\|_{2,2}^2\,ds\nonumber\\[2mm]
%&+C\varepsilon^{-1}\left(\|\mathcal{A}_{h}(\gamma_k(\bm{u}_h^{n}))\|_{0}^2+\|\bm{u}^{n+1}\|_{2,2}^2\right)\|\gamma_k(\eta_{\bm{u}}^n)\|_{1,2}^2\nonumber\\
&+C\varepsilon^{-1}\left(\|\mathcal{A}_{h}(\gamma_k(\bm{u}_h^{n-1}))\|_{0}^2+\|\bm{u}^{n}\|_{2,2}^2\right)\left[\|\gamma_k(\eta_{\bm{u}}^{n-1})\|_{1,2}^2+\|\gamma_k(e_{\bm{u}}^{n-1})\|_{1,2}^2\right],\label{fully:esti44new}
\end{align}
\begin{align}
\left|\mu_k\left(\tilde{E}_{k,2}, d_{\bm{u}}^{n+1}\right)\right|\leq& \varepsilon \left\|d_{\bm{u}}^{n+1}\right\|_0^2%\nonumber\\[2mm]
 +C\varepsilon^{-1}\mu_k^2\tau^{2k} \sum_{i=1}^ko_i\int_{t^{n-k}}^{t^{n}}\left\|\frac{\partial^{k+1} \mathcal{P}_h\bm{u}}{\partial t^{k+1}}(s)\right\|_0^2\,ds,\label{fully:esti55new}
 \end{align}
and
\begin{align}
&\left|\mu_k\left(\tilde{E}_{k,3}, d_{\bm{u}}^{n+1}\right) \right|\leq C\varepsilon^{-1}\mu_k^2\tau h^{2(l+1)} \int_{t^{n-k}}^{t^{n}} \left( \left\|\frac{\partial\bm u}{\partial t}(s)\right\|_{l+1,2}^2+ \left\|\frac{\partial p}{\partial t}(s)\right\|_{l,2}^2\right)\,ds+\varepsilon\left\|d_{\bm{u}}^{n+1}\right\|_0^2.\label{fully:esti66new}
\end{align}
 By combining inequalities (\ref{fully:esti44})-(\ref{fully:esti66new}) with  (\ref{fully:esti000}), we derive the following estimate
\begin{align}
&\nu\sum_{i,j=1}^kg_{ij}(\nabla\eta_{\bm{u}}^{n+1+i-k}, \nabla\eta_{\bm{u}}^{n+1+j-k})-\nu\sum_{i,j=1}^kg_{ij}(\nabla\eta_{\bm{u}}^{n+i-k}, \nabla\eta_{\bm{u}}^{n+j-k})%\nonumber\\[2mm]
+\tau (1/2-12\varepsilon)\left\|\frac{\alpha_k\eta_{\bm{u}}^{n+1}-\beta_k(\eta_{\bm{u}}^{n})}{\tau }\right\|_0^2\nonumber\\[2mm]
%\end{align}
%\begin{align}
\leq&\tau 1/2\mu_k^2\left\|\frac{\alpha_k\eta_{\bm{u}}^{n}-\beta_k(\eta_{\bm{u}}^{n-1})}{\tau }\right\|_0^2+C\tau \left(\|\mathcal{A}_{h}(\gamma_k(\bm{u}_h^n))\|_{0}^2+\|\bm{u}^{n+1}\|_{2,2}^2+\|\mathcal{A}_{h}(\gamma_k(\bm{u}_h^{n-1}))\|_{0}^2\right.\nonumber\\[2mm]
&\left.+\|\bm{u}^{n}\|_{2,2}^2\right)\left(\|\gamma_k(\eta_{\bm{u}}^n)\|_{1,2}^2+\|\gamma_k(e_{\bm{u}}^n)\|_{1,2}^2+\|\gamma_k(\eta_{\bm{u}}^{n-1})\|_{1,2}^2+\|\gamma_k(e_{\bm{u}}^{n-1})\|_{1,2}^2\right)\nonumber\\[2mm]
%&+C\tau \left(\|\mathcal{A}_{h}(\gamma_k(\bm{u}_h^n))\|_{0}^2+\|\bm{u}^{n+1}\|_{2,2}^2\right)\|\gamma_k(e_{\bm{u}}^n)\|_{1,2}^2\nonumber\\
&+ C\tau^{2k}\left(\|\bm{u}^{n+1}\|_{2,2}^2\sum_{i=1}^ko_i\int_{t^{n+1-i}}^{t^{n+1}}\left\|\frac{\partial^{k} \bm{u}}{\partial t^{k}}(s)\right\|_{1,2}^2\,ds+\mu_k^2\|\bm{u}^{n}\|_{2,2}^2\sum_{i=1}^ko_i\int_{t^{n-i}}^{t^{n}}\left\|\frac{\partial^{k} \bm{u}}{\partial t^{k}}(s)\right\|_{1,2}^2\,ds\right)\nonumber\\[2mm]
&+ C\tau^{2k}\left(\|\gamma_k(\bm{u}_h^{n})\|_{1,2}^2\sum_{i=1}^ko_i \int_{t^{n+1-i}}^{t^{n+1}}\left\|\frac{\partial^{k} \bm{u}}{\partial t^{k}}(s)\right\|_{2,2}^2\,ds+\mu_k^2\|\gamma_k(\bm{u}_h^{n-1})\|_{1,2}^2\sum_{i=1}^ko_i \int_{t^{n-i}}^{t^{n}}\left\|\frac{\partial^{k} \bm{u}}{\partial t^{k}}(s)\right\|_{2,2}^2\,ds\right)\nonumber\\[2mm]
%&+ C(\tau )^{2k}\Big(\int_{t^{n+1-k}}^{t^{n+1}}\|\frac{\partial^{k} \bm{u}}{\partial t^{k}}(s)\|_{1,2}^2\,ds\Big)\nonumber\\
&+C\tau^{2k}\sum_{i=1}^ko_i\left(\int_{t^{n+1-i}}^{t^{n+1}}\left\|\frac{\partial^{k+1} \bm{u}}{\partial t^{k+1}}(s)\right\|_0^2\,ds+\int_{t^{n-i}}^{t^{n}}\left\|\frac{\partial^{k+1} \bm{u}}{\partial t^{k+1}}(s)\right\|_0^2\,ds\right)\nonumber\\[2mm]
&+Ch^{2(l+1)}\tau^2 \left[ \int_{t^{n+1-k}}^{t^{n+1}}\left(\left\|\frac{\partial\bm{u}^{n+1}}{\partial t}\right\|_{l+1,2}^2+\left\|\frac{\partial p^{n+1}}{\partial t}\right\|_{l,2}^2\right)\,ds+ \int_{t^{n-k}}^{t^{n}}\left(\left\|\frac{\partial\bm{u}^{n}}{\partial t}\right\|_{l+1,2}^2+\left\|\frac{\partial p^{n}}{\partial t}\right\|_{l,2}^2\right)\,ds\right]. \label{fully:esti1111}
\end{align}
Since $0\leq\mu_k<1$, we choose $\varepsilon>0$ sufficiently small so that $1/2-12\varepsilon\ge\mu_k^2/2$. 
%Given $0\leq \mu_k<1$, we choose $\varepsilon$ sufficiently small such that $\frac{1}{2}-6\varepsilon>(\frac{1}{2}+6\varepsilon)\mu_k^2.$
Summing both sides of equation (\ref{fully:esti1111})  over $n$ from $k$ to $m$, and invoking Hypothesis \ref{assum:solution2} on the exact solution  along with the fact that $G=(g_{ij})$  is a symmetric positive definite matrix with smallest eigenvalue $\lambda_{4g}$ and  largest eigenvalue $\lambda_{4G}$, we obtain the following estimate %after removing some non-essential terms
\begin{align}
&\lambda_{4g}\|\nabla\eta_{\bm{u}}^{m+1}\|_0^2+\tau \sum_{n=k}^{m}(1/2-12\varepsilon-1/2\mu_k^2)\left\|\frac{\alpha_k\eta_{\bm{u}}^{n+1}-\beta_k(\eta_{\bm{u}}^{n})}{\tau }\right\|_0^2\nonumber\\[2mm]
%\leq &\sum_{i,j=1}^kg_{ij}(\nabla\eta_{\bm{u}}^{m+1+i-k}, \nabla\eta_{\bm{u}}^{m+1+j-k})+\tau \sum_{n=0}^{m+1}(1-6\varepsilon)\Big\|\frac{\alpha_k\eta_{\bm{u}}^{n+1}-\beta_k(\eta_{\bm{u}}^{n})}{\tau }\Big\|_0^2\nonumber\\
\leq&C\tau \sum_{n=k}^{m}d_{n}\|\nabla\eta_{\bm{u}}^{n}\|_{0}^2+C\tau \sum_{n=k-1}^m\left(\|\bm{u}^{n+1}\|_{2,2}^2\|\gamma_k(\eta_{\bm{u}}^n)\|_{1,2}^2+\|\bm{u}^{n}\|_{2,2}^2\|\gamma_k(\eta_{\bm{u}}^{n-1})\|_{1,2}^2\right)\nonumber\\[2mm]
&+C\tau \sum_{n=k}^m\left(\|\mathcal{A}_{h}(\gamma_k(\bm{u}_h^n))\|_{0}^2+\|\bm{u}^{n+1}\|_{2,2}^2\right)\|\gamma_k(e_{\bm{u}}^n)\|_{1,2}^2\nonumber\\[2mm]
&+C\tau \sum_{n=k}^m\left(\|\mathcal{A}_{h}(\gamma_k(\bm{u}_h^{n-1}))\|_{0}^2+\|\bm{u}^{n}\|_{2,2}^2\right)\|\gamma_k(e_{\bm{u}}^{n-1})\|_{1,2}^2\nonumber\\[2mm]
&+ C\tau^{2k}\int_{0}^{T}\left(\left\|\frac{\partial^{k} \bm{u}}{\partial t^{k}}(s)\right\|_{1,2}^2+\left\|\frac{\partial^{k} \bm{u}}{\partial t^{k}}(s)\right\|_{2,2}^2+\left\|\frac{\partial^{k+1} \bm{u}}{\partial t^{k+1}}(s)\right\|_0^2\right)\,ds\nonumber\\[2mm]
&+Ch^{2(l+1)}\tau^2\int_{0}^{T}\left(\left\|\frac{\partial\bm{u}^{n+1}}{\partial t}\right\|_{l+1,2}^2+\left\|\frac{\partial\bm{u}^{n}}{\partial t}\right\|_{l+1,2}^2+\left\|\frac{\partial p^{n+1}}{\partial t}\right\|_{l,2}^2+\left\|\frac{\partial p^{n}}{\partial t}\right\|_{l,2}^2\right)\,ds+\lambda_{4G}\sum_{i=0}^{k-1}\|\nabla\eta_{\bm{u}}^{i}\|_0^2,\label{fully:esti1222}
\end{align}
with $d_n=\|\mathcal{A}_{h}(\gamma_k(\bm{u}_h^n))\|_{0}^2+\|\mathcal{A}_{h}(\gamma_k(\bm{u}_h^{n-1}))\|_{0}^2$. %By (\ref{m0}) and the assumption $\tau d_n<1$, Gr\"onwall Lemma \ref{def:gronwall} is applicable.
%\begin{align*}
%\tau \|\mathcal{A}_{h}(\gamma_k(\bm{u}_h^n))\|_{0}^2+\frac{\nu}{2}\mu_k^2+\frac{\nu}{2}\alpha_k^2+\frac{\nu}{2}<1.
%\end{align*}
By invoking Hypothesis \ref{initial value estimate}, (\ref{fully:esti13}) and applying Gr$\ddot{\text{o}}$nwall Lemma \ref{def:gronwall2} to  (\ref{fully:esti1222}), we derive the following result:
\begin{align}
&\|\nabla\eta_{\bm{u}}^{m+1}\|_0^2+\tau \sum_{n=k}^{m}(1/2-12\varepsilon-1/2\mu_k^2)\left\|\frac{\alpha_k\eta_{\bm{u}}^{n+1}-\beta_k(\eta_{\bm{u}}^{n})}{\tau }\right\|_0^2%\leq&C\tau \sum_{n=1}^md_n\Big(\|\nabla \hat{e}_{\bm{u}}^n\|_0^2+\|\nabla \bar{e}_{\bm{u}}^n\|_0^2+\|\nabla \bar{e}_{\bm{u}}^{n-1}\|_0^2+\|\nabla\tilde{e}_{\bm{u}}^{n}\|_0^2\Big)\\
%&+ CT(\tau )^4\\
\leq C_u\left(\tau^{2k}+h^{2l}\right). \label{fully:0esti1333}
\end{align}
At the first BDF time level $t_k$, the required bound for $d_{\bm u}^{k}$ follows by testing the error equation corresponding to $n=k-1$ with $d_{\bm u}^{k}$ and proceeding as above; the details are omitted for brevity. Hence, we arrive at the following final estimate
\begin{align}
&\|\nabla\eta_{\bm{u}}^{m+1}\|_0^2+\tau \sum_{n=k-1}^{m}(1/2-12\varepsilon-1/2\mu_k^2)\left\|\frac{\alpha_k\eta_{\bm{u}}^{n+1}-\beta_k(\eta_{\bm{u}}^{n})}{\tau }\right\|_0^2%\leq&C\tau \sum_{n=1}^md_n\Big(\|\nabla \hat{e}_{\bm{u}}^n\|_0^2+\|\nabla \bar{e}_{\bm{u}}^n\|_0^2+\|\nabla \bar{e}_{\bm{u}}^{n-1}\|_0^2+\|\nabla\tilde{e}_{\bm{u}}^{n}\|_0^2\Big)\\
%&+ CT(\tau )^4\\
\leq C_u\left(\tau^{2k}+h^{2l}\right). \label{fully:esti1333}
\end{align}

We now consider the case $k=6$. %For $n\geq8$, set
%\begin{align}
%  \bm{v}_h=d_{\bm{u}}^{n+1}:=\frac{\alpha_6\eta_{\bm{u}}^{n+1}-\beta_6(\eta_{\bm{u}}^{n})}{\tau}.
%\end{align}
%For $n\geq5$, 
For brevity, let 
\begin{equation}
  d_{6\bm{u}}^{n+1}:=\frac{\alpha_6\eta_{\bm{u}}^{n+1} -\beta_6(\eta_{\bm{u}}^{n})}{\tau},
  \qquad
  R_6^{n+1}:=E_{6,1}^{n+1}+\frac{1}{\tau} E_{6,2}^{n+1}+E_{6,3}^{n+1}.
  \label{eq:k6-dR}
\end{equation}
Here and below, $E_{6,j}^{n+1}$ denotes the defect $E_{6,j}$ evaluated at $t_{n+1}$. Dividing (\ref{eq:erroreq1}) by $\tau$, we write the error equation at
time $t_{n+1}$ as
\begin{equation}
  (d_{6\bm {u}}^{n+1},\bm{v}_h)+\nu(\nabla\eta_{\bm{u}}^{n+1},\nabla\bm{v}_h)-(\eta_p^{n+1},\nabla\cdot\bm{v}_h)=(R_6^{n+1},\bm{v}_h).
  \label{eq:k6-normalized-error}
\end{equation}
Since (\ref{eq:erroreq2}) holds at every time level, linearity gives
\begin{equation}
  (q_h,\nabla\cdot d_{6\bm{u}}^{n+1})=0 \qquad\forall q_h\in Q_h^{l-1}.
  \label{eq:k6-d-divergence}
\end{equation}
Thus $d_{6\bm {u}}^{n+1}\in X_{0h}^l$, so it is an admissible test function and the pressure term in (\ref{eq:k6-normalized-error}) vanishes when
$\bm{v}_h=d_{6\bm{u}}^{n+1}$.

For any sequence $\{z^{n+1}\}$, introduce the BDF6 multiplier combination
\begin{equation}
  \widehat z^{n+1} :=z^{n+1}-\frac{13}{9}z^{n}+\frac{25}{36}z^{n-1}-\frac{1}{9}z^{n-2}.
  \label{eq:k6-hat}
\end{equation}
  For $n\geq 8$, we take the error equations (\ref{eq:k6-normalized-error}) at the four time
levels $t_{n+1}$, $t_n$, $t_{n-1}$, and $t_{n-2}$, multiply these equations by
\begin{equation}
  1,\qquad -\frac{13}{9},\qquad
  \frac{25}{36},\qquad -\frac{1}{9},
  \label{eq:k6-four-weights}
\end{equation}
respectively, and add them.  In the resulting equation we take
\begin{align}
  \bm{v}_h=d_{6\bm{u}}^{n+1}=\frac{\alpha_6\eta_{\bm{u}}^{n+1}-\beta_6(\eta_{\bm{u}}^{n})}{\tau}. \label{eq:k6-test-function}
\end{align}
 Using (\ref{eq:k6-d-divergence}) with
$q_h=\widehat\eta_p^{n+1}$, we obtain
\begin{equation}
  (\widehat d_{6\bm{u}}^{n+1},d_{6\bm{u}}^{n+1})
  +\nu(\nabla\widehat\eta_{\bm{u}}^{n+1},\nabla d_{6\bm{u}}^{n+1})=(\widehat R_6^{n+1},d_{6\bm{u}}^{n+1}),
  \qquad n\geq8.
  \label{eq:k6-key-identity}
\end{equation}
We next establish the two coercivity properties associated with the
left-hand side of \eqref{eq:k6-key-identity}.  First, consider
\begin{equation}
  q(\theta):=1-\frac{13}{9}\cos\theta
  +\frac{25}{36}\cos(2\theta)-\frac19\cos(3\theta).
  \label{eq:k6-pdf-symbol}
\end{equation}
With $x=\cos\theta$, we have
\begin{equation}
  q(\theta)
  =-\frac49x^3+\frac{25}{18}x^2
  -\frac{10}{9}x+\frac{11}{36}=:J(x).
\end{equation}
The only critical point of $J$ in $[-1,1]$ is
$x_*=(25-\sqrt{145})/24$, and
\begin{equation}
  J(1)=\frac5{36},\qquad
  J(-1)=\frac{13}{4},\qquad
  J(x_*)=\frac{2377-145\sqrt{145}}{15552}>\frac1{32}.
\end{equation}
Consequently,
\begin{equation}
  q(\theta)\geq\frac1{32}
  \qquad\forall\theta\in\mathbb R.
  \label{eq:k6-symbol-positive}
\end{equation}
Let $8\leq m\leq N-1$. To pass from (\ref{eq:k6-symbol-positive}) to a finite time interval, extend the
sequence $\{d_{6\bm u}^{n+1}\}_{n=8}^{m}$ by zero outside the index
interval $[8,m]$. Parseval's identity gives
\begin{align}
  Q_{8,m}^0
  &:=\frac1{2\pi}\int_{-\pi}^{\pi}
  q(\theta)
  \left\|
    \sum_{n=8}^{m}d_{6\bm u}^{n+1}
    e^{\mathrm i n\theta}
  \right\|_0^2\,\mathrm d\theta
 % \nonumber\\
  \geq
  \frac1{32}\sum_{n=8}^{m}
  \|d_{6\bm u}^{n+1}\|_0^2.
  \label{eq:k6-parseval}
\end{align}
Expanding the integral and restoring the three values preceding
$d_{6\bm u}^{9}$, namely
$d_{6\bm u}^{8}$, $d_{6\bm u}^{7}$, and $d_{6\bm u}^{6}$, gives
\begin{equation}
  \sum_{n=8}^{m}
  (\widehat d_{6\bm u}^{\,n+1},d_{6\bm u}^{n+1})
  =Q_{8,m}^0-\mathcal B_8,
  \label{eq:k6-boundary-identity}
\end{equation}
where
\begin{align}
  \mathcal B_8
  &:=
  \left(
    d_{6\bm u}^{9},
    \frac{13}{9}d_{6\bm u}^{8}
    -\frac{25}{36}d_{6\bm u}^{7}
    +\frac19d_{6\bm u}^{6}
  \right)
%  \nonumber\\
  %&\quad
  +
  \left(
    d_{6\bm u}^{10},
    -\frac{25}{36}d_{6\bm u}^{8}
    +\frac19d_{6\bm u}^{7}
  \right)
  +
  \left(
    d_{6\bm u}^{11},
    \frac19d_{6\bm u}^{8}
  \right).
  \label{eq:k6-boundary-term}
\end{align}
For $m=8$, the terms containing $d_{6\bm u}^{10}$ and
$d_{6\bm u}^{11}$ are omitted, whereas for $m=9$ only the term
containing $d_{6\bm u}^{11}$ is omitted. Therefore, Young's inequality
gives
\begin{equation}
  \sum_{n=8}^{m}
  (\widehat d_{6\bm u}^{\,n+1},d_{6\bm u}^{n+1})
  \geq
  \frac1{64}\sum_{n=8}^{m}
  \|d_{6\bm u}^{n+1}\|_0^2
  -C\sum_{j=5}^{7}\|d_{6\bm u}^{j+1}\|_0^2.
  \label{eq:k6-finite-coercivity}
\end{equation}
For the second term in (\ref{eq:k6-key-identity}),  Lemma \ref{lemma_symmetricmatrix1} gives a
symmetric positive-definite matrix
$G=(g_{ij})_{i,j=1}^{6}$ such that
\begin{equation}
  \nu(\nabla\widehat\eta_{\bm u}^{\,n+1},
  \nabla d_{6\bm u}^{n+1})
  \geq
  \frac{\nu}{\tau}
  \left(
    \|\nabla\eta_{\bm u}^{n+1}\|_G^2
    -\|\nabla\eta_{\bm u}^{n}\|_G^2
  \right),
  \label{eq:k6-G-identity}
\end{equation}
where
\begin{equation}
  \|\nabla\eta_{\bm u}^{n+1}\|_G^2
  :=
  \sum_{i,j=1}^{6}g_{ij}
  \left(
    \nabla\eta_{\bm u}^{n-5+i},
    \nabla\eta_{\bm u}^{n-5+j}
  \right).
  \label{eq:k6-G-norm}
\end{equation}
In particular, if  $\lambda_{5g}$ and $\lambda_{5G}$ denote the smallest and
largest eigenvalues of $G$, respectively, then
\begin{equation}
   \lambda_{5g}\sum_{j=0}^{5}
  \|\nabla\eta_{\bm u}^{n+1-j}\|_0^2
  \leq
  \|\nabla\eta_{\bm u}^{n+1}\|_G^2
  \leq
 \lambda_{5G}\sum_{j=0}^{5}
  \|\nabla\eta_{\bm u}^{n+1-j}\|_0^2.
  \label{eq:k6-G-equivalence}
\end{equation}
Summing \eqref{eq:k6-key-identity} from $n=8$ to $n=m$,
multiplying by $\tau$, and using
\eqref{eq:k6-finite-coercivity} and
\eqref{eq:k6-G-identity}, we obtain
\begin{align}
  &\nu\|\nabla\eta_{\bm u}^{m+1}\|_G^2
  +\frac{\tau}{64}
  \sum_{n=8}^{m}\|d_{6\bm u}^{n+1}\|_0^2
  \nonumber\\
  \leq&
  \nu\|\nabla\eta_{\bm u}^{8}\|_G^2
  +C\tau\sum_{j=5}^{7}\|d_{6\bm u}^{j+1}\|_0^2
  +\tau\sum_{n=8}^{m}
  |(\widehat R_6^{\,n+1},d_{6\bm u}^{n+1})|.
  \label{eq:k6-energy-before-force}
\end{align}
%Since the multiplier coefficients are fixed, Young's inequality gives
%\begin{equation}
%  \tau\sum_{n=8}^m
%  |(\widehat R_6^{\,n+1},d_{\bm u}^{n+1})|
 % \leq\frac{\tau}{128}\sum_{n=8}^m
%  \|d_{\bm u}^{n+1}\|_0^2
 % +C\tau\sum_{j=5}^m\|R_6^{j+1}\|_0^2.
%  \label{eq:k6-force}
%\end{equation}
It remains to control the three starting derivatives. At each of the
three starting indices $n=5,6,7$, corresponding to the time levels
$n+1=6,7,8$, respectively, we take
\begin{equation}
  \bm v_h=d_{6\bm u}^{n+1}
  =\frac{\alpha_6\eta_{\bm u}^{n+1}
  -\beta_6(\eta_{\bm u}^{n})}{\tau}
  \label{eq:k6-pdf-start-test}
\end{equation}
in \eqref{eq:k6-normalized-error}. The pressure term again vanishes
by \eqref{eq:k6-d-divergence}, and hence
\begin{equation}
  \|d_{6\bm u}^{n+1}\|_0^2
  +\nu(\nabla\eta_{\bm u}^{n+1},
  \nabla d_{6\bm u}^{n+1})
  =(R_6^{n+1},d_{6\bm u}^{n+1}).
  \label{eq:k6-start-identity}
\end{equation}
%Using Young's inequality on the right-hand side of (\ref{eq:k6-start-identity}), we have
%\begin{equation}
%  (R_6^{n+1},d_{\bm u}^{n+1})
%  \leq \frac{1}{2}\|d_{\bm u}^{n+1}\|_0^2
%  + \frac{1}{2}\|R_6^{n+1}\|_0^2.
 % \label{eq:k6-start-force}
%\end{equation}
Using $\alpha_6=147/60$ and Young's inequality, we have
\begin{align*}
  \nu(\nabla\eta_{\bm u}^{n+1},
  \nabla d_{6\bm u}^{n+1})
  &\geq\frac{147\nu}{120\tau}
  \|\nabla\eta_{\bm u}^{n+1}\|_0^2
  -\frac{30\nu}{147\tau}
  \|\nabla\beta_6(\eta_{\bm u}^{n})\|_0^2.
\end{align*}
Applying this estimate successively for $n=5,6,7$ yields
\begin{align}
  \tau\sum_{n=5}^{7}\|d_{6\bm u}^{n+1}\|_0^2
  +\max_{5\leq n\leq7}
  \|\nabla\eta_{\bm u}^{n+1}\|_0^2
  \leq C\left(
  \sum_{j=0}^{5}\|\nabla\eta_{\bm u}^j\|_0^2
  +\tau\sum_{n=5}^{7}|(R_6^{n+1},d_{6\bm u}^{n+1})|\right).
  \label{eq:k6-start-bound}
\end{align}

Combining (\ref{eq:k6-energy-before-force}) and (\ref{eq:k6-start-bound}),  and using (\ref{eq:k6-G-equivalence}), we find
\begin{align}
  &\max_{5\leq n\leq m}
  \|\nabla\eta_{\bm u}^{n+1}\|_0^2
  + \frac{\tau}{64}\sum_{n=5}^m\|d_{6\bm u}^{n+1}\|_0^2
  \leq C\left(
  \sum_{j=0}^{5}\|\nabla\eta_{\bm u}^j\|_0^2
  +\tau\sum_{n=5}^m|(R_6^{n+1},d_{6\bm u}^{n+1})|\right).
  \label{eq:k6-linear-error-bound}
\end{align}
We now estimate $R_6^{n+1}$, with the help of (\ref{bestimate}) and (\ref{un1andgammaun}), there holds 
\begin{align*}
&\left|a\left(-\bm{u}^{n+1}+\gamma_6(\bm{u}^n)-\gamma_6(\bm{u}^n-\bm{u}_h^n),\bm{u}^{n+1},d_{6\bm{u}}^{n+1}\right)\right|\\[2mm]
\leq&\left(\|\bm{u}^{n+1}-\gamma_6(\bm{u}^n)\|_{1,2}+\| \gamma_6(\bm{u}^n-\bm{u}_h^n)\|_{1,2}\right)\|\bm{u}^{n+1}\|_{2,2} \left\|d_{6\bm{u}}^{n+1}\right\|_{0}\\[2mm]
%\leq&C\varepsilon^{-1}\|\bm{u}^{n+1}\|_{2,2}^2 \left\|\sum_{i=1}^kb_i\int_{t^{n+1-i}}^{t^{n+1}}(t^{n+1-i}-s)^{k-1}\frac{\partial^{k} \bm{u}}{\partial t^{k}}(s)\,ds\right\|_{1,2}^2\\
%&+2\varepsilon \left\|\frac{\alpha_k\eta_{\bm{u}}^{n+1}-\beta_k(\eta_{\bm{u}}^{n})}{\tau}\right\|_{0}^2+C\varepsilon^{-1}\left(\| \gamma_k(\eta_{\bm{u}}^n)\|_{1,2}^2+\| \gamma_k(e_{\bm{u}}^n)\|_{1,2}^2\right)\|\bm{u}^{n+1}\|_{2,2}^2\\
\leq&C\varepsilon^{-1}\tau^{11}\|\bm{u}^{n+1}\|_{2,2}^2\sum_{i=1}^6 \int_{t^{n+1-i}}^{t^{n+1}}\left\|\frac{\partial^{6} \bm{u}}{\partial t^{6}}(s)\right\|_{1,2}^2\,ds\\[2mm]
&+2\varepsilon \left\|d_{6\bm{u}}^{n+1}\right\|_{0}^2+C\varepsilon^{-1}\left(\| \gamma_6(\eta_{\bm{u}}^n)\|_{1,2}^2+\| \gamma_6(e_{\bm{u}}^n)\|_{1,2}^2\right)\|\bm{u}^{n+1}\|_{2,2}^2.
\end{align*} 
%where the coefficients $b_i$ denote certain  fixed and bounded constants  determined by the truncation error. 
 Following similar techniques, one can derive that
\begin{align*}
&\left|a\left(\gamma_6(\bm{u}_h^{n}),\bm{u}^{n+1}-\gamma_6(\bm{u}^n),d_{6\bm{u}}^{n+1}\right)\right|\\[2mm]
\leq&\|\gamma_6(\bm{u}_h^{n})\|_{1,2}\|\bm{u}^{n+1}-\gamma_6(\bm{u}^n)\|_{2,2} \left\|d_{6\bm{u}}^{n+1}\right\|_{0}\\[2mm]
%\leq& C\varepsilon^{-1}\|\gamma_k(\bm{u}_h^{n})\|_{1,2}^2\|\sum_{i=1}^kb_i\int_{t^{n+1-i}}^{t^{n+1}}(t^{n+1-i}-s)^{k-1}\frac{\partial^{k} \bm{u}}{\partial t^{k}}(s)\,ds\|_{2,2}^2\\
%&+\varepsilon\|\frac{\alpha_k\eta_{\bm{u}}^{n+1}-\beta_k(\eta_{\bm{u}}^{n})}{\tau}\|_0^2\\
\leq& C\varepsilon^{-1}\tau^{11} \|\gamma_6(\bm{u}_h^{n})\|_{1,2}^2\sum_{i=1}^6\int_{t^{n+1-i}}^{t^{n+1}}\left\|\frac{\partial^{6} \bm{u}}{\partial t^{6}}(s)\right\|_{2,2}^2\,ds
+\varepsilon \left\|d_{6\bm{u}}^{n+1} \right\|_0^2,
\end{align*}
and
\begin{align*}
&\left|a\left(\gamma_6(\bm{u}_h^{n}),\gamma_6(\bm{u}^n-\bm{u}_h^{n}),d_{6\bm{u}}^{n+1}\right)\right|\\[2mm]
\leq&\|\mathcal{A}_{h}(\gamma_6(\bm{u}_h^{n}))\|_{0}\left(\|\gamma_6(\eta_{\bm{u}}^{n})\|_{1,2}+\|\gamma_6(e_{\bm{u}}^{n})\|_{1,2}\right) \left\|d_{6\bm{u}}^{n+1}\right\|_{0}\\[2mm]
\leq&C\varepsilon^{-1}\|\mathcal{A}_{h}(\gamma_6(\bm{u}_h^{n}))\|_{0}^2\left(\|\gamma_6(\eta_{\bm{u}}^n)\|_{1,2}^2+\|\gamma_6(e_{\bm{u}}^{n})\|_{1,2}^2\right)
+\varepsilon\left\|d_{6\bm{u}}^{n+1}\right\|_{0}^2.
\end{align*}
This implies 
\begin{align}
\left|\left(E_{6,1}^{n+1}, d_{6\bm{u}}^{n+1}\right)\right|\leq&4\varepsilon \left\|d_{6\bm{u}}^{n+1}\right\|_{0}^2%\nonumber\\[2mm]
+ C\varepsilon^{-1} \tau^{11}\|\bm{u}^{n+1}\|_{2,2}^2\sum_{i=1}^6\int_{t^{n+1-i}}^{t^{n+1}}\left\|\frac{\partial^{6} \bm{u}}{\partial t^{6}}(s)\right\|_{1,2}^2\,ds\nonumber\\[2mm]
&+ C\varepsilon^{-1} \tau^{11}\|\gamma_6(\bm{u}_h^{n})\|_{1,2}^2\sum_{i=1}^6 \int_{t^{n+1-i}}^{t^{n+1}}\left\|\frac{\partial^{6} \bm{u}}{\partial t^{6}}(s)\right\|_{2,2}^2\,ds\nonumber\\[2mm]
&+C\varepsilon^{-1}\left(\|\mathcal{A}_{h}(\gamma_6(\bm{u}_h^{n}))\|_{0}^2+\|\bm{u}^{n+1}\|_{2,2}^2\right)\|\gamma_6(\eta_{\bm{u}}^n)\|_{1,2}^2\nonumber\\[2mm]
&+C\varepsilon^{-1}\left(\|\mathcal{A}_{h}(\gamma_6(\bm{u}_h^{n}))\|_{0}^2+\|\bm{u}^{n+1}\|_{2,2}^2\right)\|\gamma_6(e_{\bm{u}}^{n})\|_{1,2}^2.\label{fully:esti44k6}
\end{align}
We continue to estimate
\begin{align}
\left|\left(E_{6,2}^{n+1}, d_{6\bm{u}}^{n+1}\right)\right|\leq& \varepsilon \left\|d_{6\bm{u}}^{n+1}\right\|_0^2%\nonumber\\[2mm]
 +C(\varepsilon\nu)^{-1}\tau^{-1}\tau^{13}\sum_{i=1}^6\int_{t^{n+1-i}}^{t^{n+1}}\left\|\frac{\partial^{7} \mathcal{P}_h\bm{u}}{\partial t^{7}}(s)\right\|_0^2\,ds,\label{fully:esti55k6}
 \end{align}
and
\begin{align}
&\left|\left(E_{6,3}^{n+1},d_{6\bm{u}}^{n+1}\right) \right|\leq C\varepsilon^{-1}h^{2(l+1)}\tau\int_{t^{n-5}}^{t^{n+1}}\left(\left\|\frac{\partial\bm{u}^{n+1}}{\partial t}\right\|_{l+1,2}^2+\left\|\frac{\partial p^{n+1}}{\partial t}\right\|_{l,2}^2\right)\,ds+\varepsilon\left\|d_{6\bm{u}}^{n+1}\right\|_0^2.\label{fully:esti66k6}
\end{align}
The residual terms at the three preceding time levels, $E_{6,j}^{n-\ell}$, $j=1,2,3$ and $\ell=0,1,2$, can be estimated by the same arguments as those used for $E_{6,j}^{n+1}$.  Since the BDF6 multiplier coefficients are fixed, their contributions are absorbed into the generic constant $C$.  By combining inequalities (\ref{fully:esti44k6})-(\ref{fully:esti66k6}) with (\ref{eq:k6-linear-error-bound}), we derive the following estimate
\begin{align}
  &\|\nabla\eta_{\bm u}^{m+1}\|_0^2
  +\tau(1/64-24\varepsilon)\sum_{n=5}^m\|d_{6\bm u}^{n+1}\|_0^2
  \nonumber\\
  \leq %C\tau\sum_{n=5}^{m} \bar{d}_n\|\nabla\eta_{\bm{u}}^{n}\|_{0}^2%\nonumber\\
%C\tau\sum_{n=5}^m\left(\|\mathcal{A}_{h}(\gamma_6(\bm{u}_h^n))\|_{0}^2+\|\bm{u}^{n+1}\|_{2,2}^2\right)\|\gamma_6(e_{\bm{u}}^n)\|_{1,2}^2\nonumber\\[2mm]
 &C\tau\sum_{n=5}^{m} \bar d_n\|\gamma_6(\eta_{\bm u}^{n})\|_{1,2}^2+C\tau\sum_{n=8}^{m}\sum_{\ell=0}^{2} \bar d_{n-\ell-1} \|\gamma_6(\eta_{\bm u}^{n-\ell-1})\|_{1,2}^2\nonumber\\[2mm]
&+C\tau\sum_{n=5}^{m}\bar d_n\|\gamma_6(e_{\bm u}^{n})\|_{1,2}^2+C\tau\sum_{n=8}^{m}\sum_{\ell=0}^{2} \bar d_{n-\ell-1} \|\gamma_6(e_{\bm u}^{n-\ell-1})\|_{1,2}^2\nonumber\\[2mm]
&+ C\tau^{12}\int_{0}^{T}\left(\left\|\frac{\partial^{6} \bm{u}}{\partial t^{6}}(s)\right\|_{1,2}^2+\left\|\frac{\partial^{6} \bm{u}}{\partial t^{6}}(s)\right\|_{2,2}^2+\left\|\frac{\partial^{7} \bm{u}}{\partial t^{7}}(s)\right\|_0^2\right)\,ds\nonumber\\[2mm]
&+Ch^{2(l+1)}\tau^2\int_0^T\left(\left\|\frac{\partial\bm{u}^{n+1}}{\partial t}\right\|_{l+1,2}^2+\left\|\frac{\partial p^{n+1}}{\partial t}\right\|_{l,2}^2\right)\,ds+C\sum_{i=0}^5\|\nabla\eta_{\bm{u}}^{i}\|_0^2\nonumber\\[2mm]
&+Ch^{2(l+1)}\tau^2\int_0^T\sum_{\ell=0}^{2}\left( \left\|\frac{\partial\bm u^{n-\ell}}{\partial t}\right\|_{l+1,2}^2+\left\|\frac{\partial p^{n-\ell}}{\partial t}\right\|_{l,2}^2 \right)\,ds,\label{fully:esti1222k6}
\end{align}
with $ \bar{d}_n=\|\mathcal{A}_{h}(\gamma_6(\bm{u}_h^n))\|_{0}^2+\|\bm{u}^{n+1}\|_{2,2}^2$. By invoking  Hypothesis \ref{initial value estimate} and Lemma \ref{def:gronwall2}, we derive %the following result:
\begin{align}
&\|\nabla\eta_{\bm{u}}^{m+1}\|_0^2+\tau\sum_{n=5}^{m}(1/64-24\varepsilon)\left\|\frac{\alpha_6\eta_{\bm{u}}^{n+1}-\beta_6(\eta_{\bm{u}}^{n})}{\tau}\right\|_0^2%\leq&C\tau\sum_{n=1}^md_n\Big(\|\nabla \hat{e}_{\bm{u}}^n\|_0^2+\|\nabla \bar{e}_{\bm{u}}^n\|_0^2+\|\nabla \bar{e}_{\bm{u}}^{n-1}\|_0^2+\|\nabla\tilde{e}_{\bm{u}}^{n}\|_0^2\Big)\\
%&+ CT(\tau)^4\\
\leq C_u\left(\tau^{12}+h^{2l}\right). \label{fully:esti1333k6}
\end{align}
By applying the triangle inequality on (\ref{fully:esti1333}) and (\ref{fully:esti1333k6}) together with the approximation properties of the projection operators,  the desired  result follows (\ref{space discrete : theorem 2}). 

\textbf{Step IV}: The estimate established in \textbf{Step III} and the projection error bound imply that, for $1\leq k\leq6$,
\begin{equation}
 \|\nabla(\bm u^{m+1}-\bm u_h^{m+1})\|_0 \leq C_u\left(\tau^{k}+h^{l}\right),
 \label{eq:H1-total-error-for-induction}
\end{equation}
where $C_{u}$ is independent of $m$, $h$, and $\tau$. Choose $h_0,\tau_0>0$ so that
\begin{align*}
 C_u(\tau_0^k+h_0^l)\leq1.
\end{align*}
Then, for $h\leq h_0$ and $\tau\leq\tau_0$,
\begin{align*}
 \|\nabla\bm u_h^{m+1}\|_0\leq\|\nabla\bm u(t^{m+1})\|_0
 +\|\nabla(\bm u(t^{m+1})-\bm u_h^{m+1})\|_0 \leq C_{H^1}+1<C_{\Diamond}.
\end{align*}
Thus the induction is complete, and (\ref{lemmadeltah}) follows. This concludes the proof of Theorem  \ref{theoremmain}.
\end{proof}

%\section{Error analysis of pressure}
%In order to estimate the pressure, we need to assume the higher regularity of magnetic induction $\bm{B}$. 
 \subsection{Proof of optimal error estimate for pressure}
Based on the results established in Theorem  \ref{theoremmain}, we derive the optimal error estimate for the pressure field.
%\begin{theorem}\label{space discrete : theorem 3}
%Under the same conditions as Theorem \ref{theoremmain}, with $k=1,\cdots,6$, the following estimate holds for all $m+1\leq N$
%Under the same conditions as Theorem \ref{space discrete : theorem 2}. %Then the continuous system (\ref{weak:eq1})-(\ref{weak:eq2}) and the finite element problem (\ref{weak space discrete:eq1})-(\ref{weak space discrete:eq2}) admit the unique solution $(\bm{u}^n, p^n)$ and $(\bm{u}_h^n,p_h^n)$, respectively. For all $1\leq m\leq N$, we  have the following estimate,
%\begin{equation*}
%\begin{split}
%\tau  \sum_{n=k-1}^{m}\Vert p^{n+1}-p_h^{n+1}\Vert_0^2\leq  C_{u,p}\left(h^{2l}+\tau^{2k}\right),
%\end{split}
%\end{equation*}
%where the constant $ C_{u,p}$ depends on $T,\nu,\Omega$  and the exact solution $\bm{u}, p$, but is independent
%of $\tau $ and $h$.
%\end{theorem}
%\begin{proof}
\begin{proof}[Proof of Theorem \ref{space discrete : theorem 3}] 
By analyzing the error equation for $p_h^{n+1}$,  we have
\begin{align}
(\eta_p^{n+1},\nabla\cdot\bm{v}_h)=&\left(\frac{\alpha_k{\eta}_{\bm{u}}^{n+1}-\beta_k({\eta}_{\bm{u}}^{n})}{\tau },\bm{v}_h\right)-\left( E_{k,1}+\frac{1}{\tau }E_{k,2}+E_{k,3},\bm{v}_h\right)%\nonumber\\
+\nu(\nabla{\eta}_{\bm{u}}^{n+1},\nabla\bm{v}_h).%\nonumber\\
\label{eq:errorp1}
\end{align}
By invoking the inf-sup condition (\ref{infsup}),  there holds
\begin{align}
&\chi^*\|\eta_p^{n+1}\|_0\leq \sup\limits_{\bm{0}\ne\bm{v}_h\in \bm{X}_h^{l}}\frac{(\eta_p^{n+1}, \nabla\cdot\bm{v}_h)}{\|\bm{v}_h\|_{1,2}}.\label{eq:errorp2}
\end{align}
Combining (\ref{eq:errorp1}) and (\ref{eq:errorp2}), we will continue to derive
\begin{align}
\tau \|\eta_p^{n+1}\|_0^2\leq &C\tau \left\|\frac{\alpha_k{\eta}_{\bm{u}}^{n+1}-\beta_k({\eta}_{\bm{u}}^{n})}{\tau }\right\|_0^2+C\nu^2\tau \|\nabla{\eta}_{\bm{u}}^{n+1}\|_0^2\nonumber\\[2mm]
&+C\tau \left(\|\nabla(\gamma_k(\bm{u}_h^n))\|_{0}^2+\|\bm{u}^{n+1}\|_{2,2}^2\right)\left(\|\gamma_k(\eta_{\bm{u}}^n)\|_{1,2}^2+\|\gamma_k(e_{\bm{u}}^n)\|_{1,2}^2\right)\nonumber\\[2mm]
%&+C\tau \left(\|\nabla(\gamma_k(\bm{u}_h^n))\|_{0}^2+\|\bm{u}^{n+1}\|_{2,2}^2\right)\|\gamma_k(e_{\bm{u}}^n)\|_{1,2}^2\nonumber\\
&+ C\tau^{2k}\|\bm{u}^{n+1}\|_{2,2}^2\sum_{i=1}^ko_i \int_{t^{n+1-i}}^{t^{n+1}}\left\|\frac{\partial^{k} \bm{u}}{\partial t^{k}}(s)\right\|_0^2\,ds\nonumber\\[2mm]
&+ C\tau^{2k}\|\gamma_k(\bm{u}_h^{n})\|_{1,2}^2\sum_{i=1}^ko_i \int_{t^{n+1-i}}^{t^{n+1}}\left\|\frac{\partial^{k} \bm{u}}{\partial t^{k}}(s)\right\|_{1,2}^2\,ds+C\tau^{2k}\sum_{i=1}^ko_i\int_{t^{n+1-i}}^{t^{n+1}}\left\|\frac{\partial^{k+1} \bm{u}}{\partial t^{k+1}}(s)\right\|_0^2\,ds\nonumber \\[2mm]
&+Ch^{2(l+1)} \tau^2\int_{t^{n+1-k}}^{t^{n+1}}\left(\left\|\frac{\partial\bm{u}^{n+1}}{\partial t}\right\|_{l+1,2}^2+\left\|\frac{\partial p^{n+1}}{\partial t}\right\|_{l,2}^2\right)\,ds.\label{eq:errorp3}
\end{align}
By combining (\ref{fully:esti13}), (\ref{fully:esti13k6}),   (\ref{fully:esti1333}), (\ref{fully:esti1333k6}) and (\ref{eq:errorp3}),  the following inequality is obtained
\begin{align}
&\tau  \sum_{n=k-1}^{m}\|\eta_p^{n+1}\|_0^2\leq C_{u,p}\left(\tau ^{2k}+h^{2l}\right).\label{eq:errorp4}
\end{align}
The desired result is obtained by  invoking the triangle inequality and the approximation properties of the projection operators.
\end{proof}

\section{Numerical experiments}
\label{section:numerical}

In this section, we conduct a series of  numerical experiments to validate the accuracy and stability of the proposed  IMEX‑BDF$k$ finite element scheme. For all tests, the Taylor-Hood ($P_2$-$P_1$) finite element pair is employed  to approximate  the velocity-pressure fields. The numerical implementation is based on the finite element discretization library MFEM \cite{Anderson2021} (version: 4.8.1), and all computations are performed on the LSSC-IV cluster of the State Key Laboratory of Scientific and Engineering Computing, Chinese Academy of Sciences.

\subsection{Convergence rate for temporal discretization}
\label{subsection1}
In this example, we assess the temporal accuracy of the proposed numerical scheme by considering a problem with a known exact solution.   The computational domain is chosen as  $ \Omega = [0, 1]^{3}$, with the exact solutions prescribed by 
\begin{align*}
\bm{u} = (y\,\sin^2(t), z\,\sin^2(t), x\,\sin^2(t) )^{\top}, \quad p = (x + y + z - 1.5) \sin^2(t).
\end{align*}
The boundary conditions and source terms are determined to align with
these exact solutions.  Given that the exact solutions exhibit linearity in space, the error is dominated by temporal discretization. The kinematic viscosity is set to  $\nu=1$, and the unit cube is uniformly partitioned into  $48\times 48 \times 48$  small cubes,  with each small cube subdivided into six tetrahedral elements. To assess convergence rates, we employ second- to sixth-order temporal schemes ($k=2,3,4,5,6$).   The final time is set to $T=5.0$ for $k=2, 4$, $T=7.5$ for $k=3, 5$ and $T=10.0$ for $k=6$. The time step $\tau $  is selected as $\{ T/20,\ T/40,\ T/80,\ T/160 \}$ for successive refinements.  The convergence behavior of the velocity errors in the  $H^{1}$- and $L^{2}$-norms,  along with the $L^{2}$-error of the pressure, is displayed in Figure \ref{fig:time_convergence}.  As demonstrated, this is in full agreement with the theoretical prediction and validates the temporal discretization approach adopted in this work.
%Figure \ref{fig:time_convergence} presents the convergence rates of velocity errors in the  $H^{1}$- and $L^{2}$-norms, along with the $L^{2}$-error of the pressure, for second- to fifth-order schemes.  The final time is set to $T=5.0$ for $k=2$ and  $k=4$, and $T=7.5$ for $k=3$ and  $k=5$.The time step $\tau $  is successively refined as $\{ T/20,\ T/40,\ T/80,\ T/160 \}$.  The numerical results exhibit the expected convergence rates for both velocity and pressure fields, confirming the theoretical predictions.
\begin{figure}[H]
    \centering
    \begin{subfigure}[b]{0.3\textwidth}
        \centering
        \includegraphics[width=\linewidth, height=\linewidth]{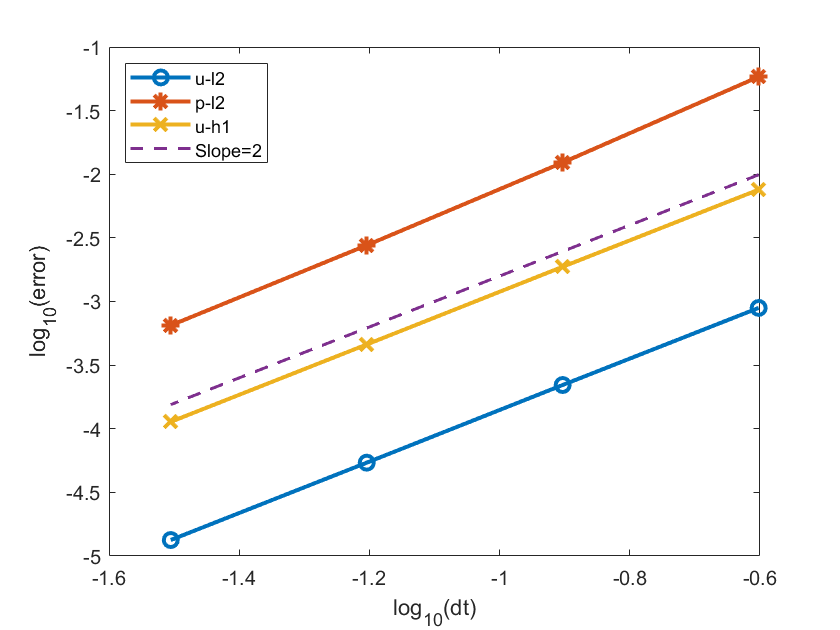}
        \caption{$k=2$}
        \label{fig:k2}
    \end{subfigure}
    \hfill
    \begin{subfigure}[b]{0.3\textwidth}
        \centering
        \includegraphics[width=\linewidth, height=\linewidth]{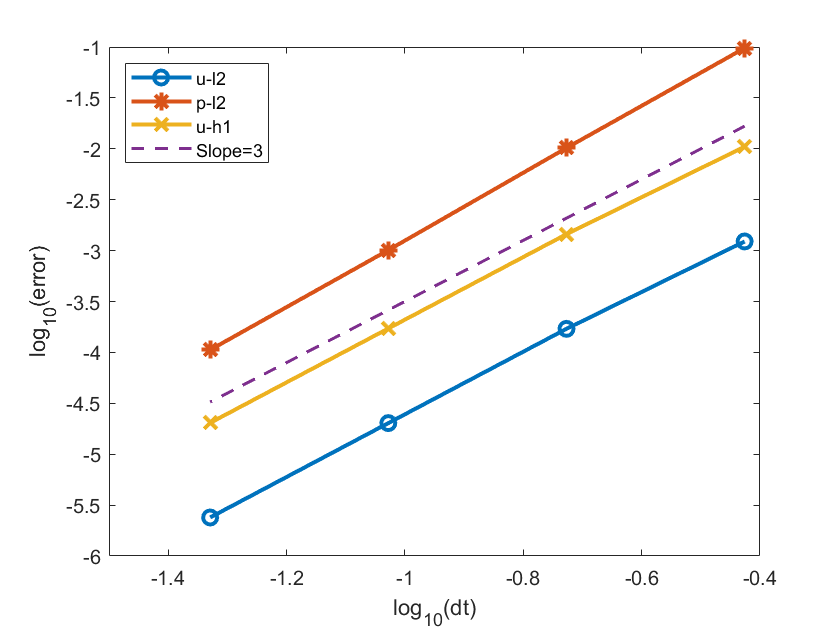}
        \caption{$k=3$}
        \label{fig:k3}
    \end{subfigure}
    \hfill
    \begin{subfigure}[b]{0.3\textwidth}
        \centering
        \includegraphics[width=\linewidth, height=\linewidth]{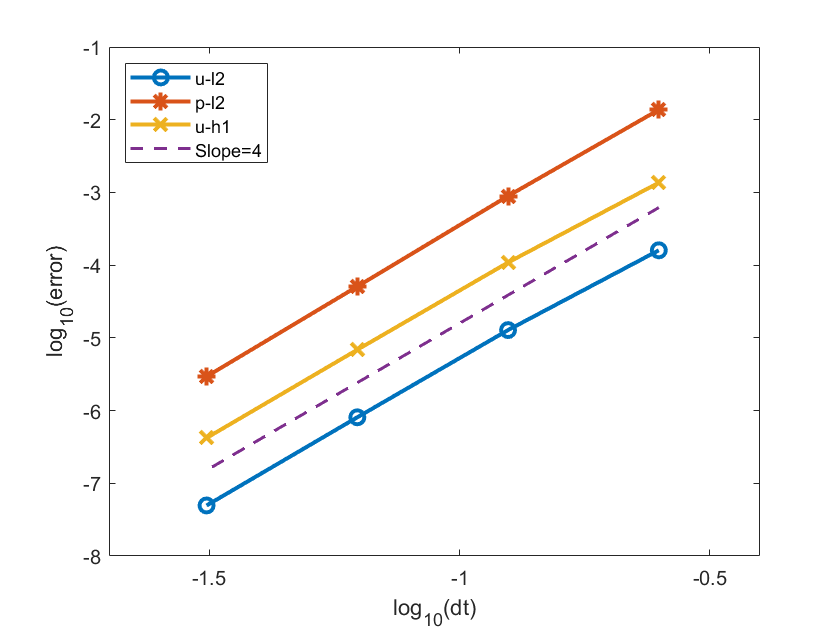}
        \caption{$k=4$}
        \label{fig:k4}
    \end{subfigure}
    \hfill
    \begin{subfigure}[b]{0.3\textwidth}
        \centering
       \includegraphics[width=\linewidth, height=\linewidth]{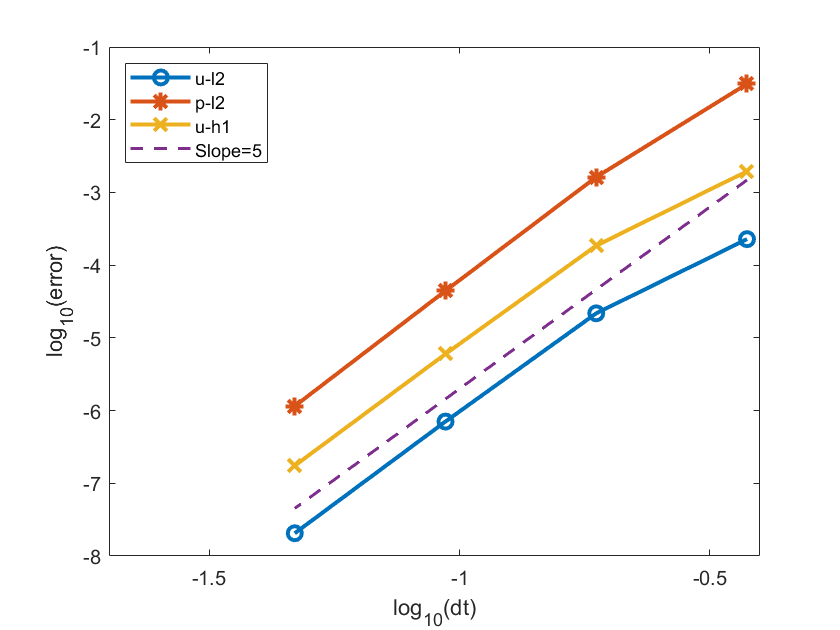}
        \caption{$k=5$}
        \label{fig:k5}
    \end{subfigure}
   %  \hfill
    \begin{subfigure}[b]{0.3\textwidth}
        \centering
        \includegraphics[width=\linewidth, height=\linewidth]{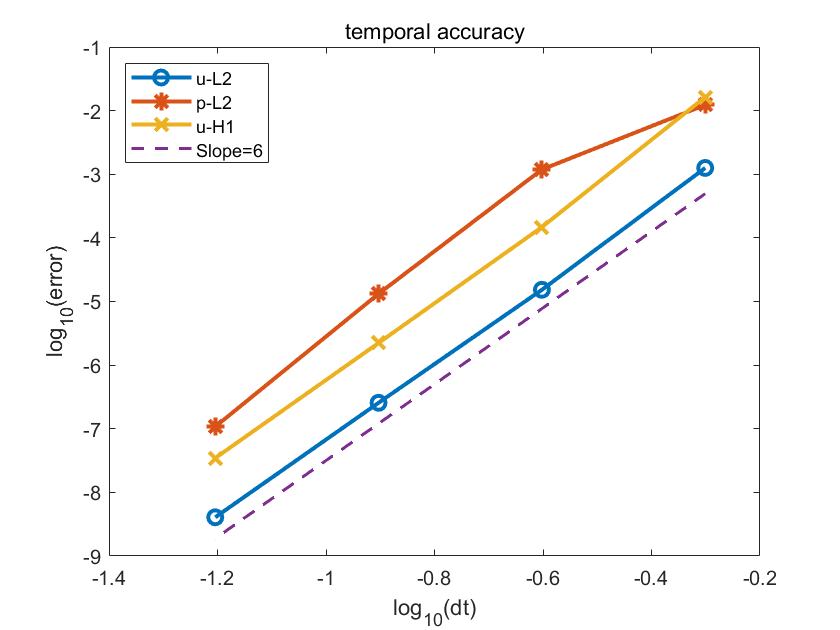}
        \caption{$k=6$}
        \label{fig:k6}
    \end{subfigure}
    \caption{Temporal convergence rates of the BDF$k$ ($k=2,3,4,5,6$)}
    \label{fig:time_convergence}
\end{figure}

\subsection{Convergence rate for spatial discretization}
%Similar to subsection \ref{subsection1}, 
This example  is designed to examine the convergence rates for spatial approximations. The computational domain is   taken as 
$ \Omega = [0, 1]^{3}$, and the physical parameter $\nu$  is set to $1$.  The true solutions are expressed as 
\begin{align*}
\bm{u}=(y(1 - y)z(1 - z)\cos(t), x(1 - x)z(1 - z)\cos(t), 0)^{\top},  \quad  p = (2x - 1)yz\cos(t).
\end{align*}
We fix the time step at $\tau  = 0.01$ to examine spatial convergence rates with respect to mesh refinement. %Table \ref{tab:convergence_rates5}, 
Table \ref{tab:convergence_rates1} and Table \ref{tab:convergence_rates2} present  the numerical errors and convergence orders for $k=1$ and $k=2$ schemes at the final time  $T=1.0$, respectively.  These results confirm that all schemes attain  the theoretical convergence orders.  When evaluating spatial accuracy, we fix $\tau$ sufficiently small to ensure error stabilization under temporal refinement, meaning the spatial error dominates the temporal error. At this point, using higher-order temporal schemes, the $\tau^{k}$ term decreases as the order $k$ increases, leading to an overall error that remains nearly unchanged. Consequently, the computed errors and convergence orders for $k=3,4,5,6$ are almost identical to those for $k=2$ and are omitted for brevity. %For brevity, these results are not shown here.

%For $k=3$ and $k=4$ cases, the error values under current settings  differ from those of $k=2$ and $k=5$ schemes mostly beyond the third decimal place. This  behavior arises  because the sufficiently small time step renders temporal errors dominated by spatial errors, where the $(\tau )^{k}$ term decays  rapidly with increasing $k$. Consequently, the error reduction becomes indistinguishable after rounding, yielding convergence orders virtually identical to the $k=2$ and $k=5$ cases. These results are omitted here due to space constraints.

%\begin{table}[H]
 %   \centering
  % % \caption{Spatial convergence order at \( t = T =1.0\)}
  %  \label{tab:convergence_rates}
   % %\begin{subtable}[t]{\textwidth}
%        \centering
 %       \caption{Spatial convergence order at \( T =1.0\), \( k = 5 \)}
 %       \label{tab:convergence_rates5}
 %       \begin{tabular}{@{}ccccc@{}}
  %          \toprule
  %          \( h \) & 0.281 & 0.140 & 0.070 & 0.035 \\ \midrule
  %          \( \|u - u_h\|_0 \) & 1.87E-04 & 2.40E-05 & 3.02E-06 & 3.79E-07 \\
 %           \textit{order} & - & 2.96 & 2.98 & 2.98 \\
  %          \( \|u - u_h\|_{1,2} \) & 6.36E-03 & 1.61E-03 & 4.05E-04 & 1.01E-04 \\
   %         \textit{order} & - & 1.98 & 1.99 & 1.99 \\
   %         \( \|p - p_h\|_0 \) & 4.23E-03 & 1.05E-03 & 2.61E-04 & 6.52E-05 \\
   %         \textit{order} & - & 2.02 & 2.01 & 2.01 \\ \bottomrule
   %     \end{tabular}
  %  %\end{subtable}
  %  %\vspace{1em} % 添加一些垂直间距
   % \end{table}

\begin{table}[H]
%  \begin{subtable}[t]{\textwidth}
        \centering
       % \label{tab:convergence_rates1}
        \begin{tabular}{@{}ccccc@{}}
            \toprule
            $h$ & 0.281 & 0.140 & 0.070 & 0.035 \\ \midrule
            $\|\bm{u}(T) - \bm{u}_h^N\|_0 $ & 1.87E-04 & 2.40E-05 & 3.05E-06 & 5.78E-07 \\
            \textit{order} & - - & 2.97 & 2.97 & 2.80 \\
            $\|\bm{u}(T) - \bm{u}_h^N\|_{1,2} $ & 6.36E-03 & 1.61E-03 & 4.06E-04 & 1.02E-04 \\
            \textit{order} & -  -& 1.98 & 1.99 & 1.99 \\
            $ \|p(T) - p_h^N\|_0 $ & 4.24E-03 & 1.05E-03 & 2.65E-04 & 7.77E-05 \\
            \textit{order} & - - & 2.02 & 2.00 & 1.93 \\ \bottomrule
        \end{tabular}
        \vspace{-0.5em} 
         \caption{Spatial convergence order at $T =1.0$, $k = 1$}
          \label{tab:convergence_rates1}
   % \end{subtable}
\end{table}
\vspace{-2em}   % 或者 -0.5\baselineskip、-5pt 等，根据实际效果微调
\begin{table}[H]
 %   \begin{subtable}[t]{\textwidth}
        \centering
%        \caption{Spatial convergence order at \( T =1.0\), \( k = 2 \)}
       % \label{tab:convergence_rates2}
        \begin{tabular}{@{}ccccc@{}}
            \toprule
            $h $ & 0.281 & 0.140 & 0.070 & 0.035 \\ \midrule
            $ \|\bm{u}(T) - \bm{u}_h^N\|_0 $ & 1.87E-04 & 2.40E-05 & 3.02E-06 & 3.79E-07 \\
            \textit{order} & -  -& 2.96 & 2.98 & 2.98 \\
            $ \|\bm{u}(T) - \bm{u}_h^N\|_{1,2} $ & 6.36E-03 & 1.61E-03 & 4.05E-04 & 1.02E-04 \\
            \textit{order} & - -& 1.98 & 1.99 & 1.99 \\
            $\|p(T) - p_h^N\|_0 $ & 4.23E-03 & 1.05E-03 & 2.61E-04 & 6.53E-05 \\
            \textit{order} & - - & 2.02 & 2.01 & 2.01 \\ \bottomrule
        \end{tabular}
        \vspace{-0.5em} 
         \caption{Spatial convergence order at $T =1.0$, $k = 2$}
           \label{tab:convergence_rates2}
  %  \end{subtable}
  %  \vspace{1em} % 添加一些垂直间距
  \end{table}
    \vspace{-1em} 
\subsection{Double shear layer problem}

We consider the  classical double shear layer problem for the Navier-Stokes equations in the domain $\Omega = [0,1]^2$, % \times (0,1)$, % subject to periodic boundary conditions, 
with the initial condition  given by
\begin{align*}
u_1(x,y,0) &= 
\begin{cases} 
\tanh(\rho(y - 0.25)), & y \leq 0.5, \\
\tanh(\rho(0.75 - y)), & y > 0.5,
\end{cases} \qquad
u_2(x,y,0) = \zeta \sin(2\pi x),
\end{align*}
where $\rho$ determines the slope of the shear layer and $\zeta$ represents the size of the perturbation. In all simulations, we set $\zeta = 0.05$ and $\bm{f}=\bm{0}$.
%The perturbation amplitude is fixed on $\zeta = 0.05$ and the external body force is taken as $\bm{f}=\bm{0}$ in our simulations.

To evaluate the performance of the high-order numerical scheme in capturing complex flow structures, we first simulate the thick shear layer problem using the BDF$4$ and BDF$6$ schemes  with physical parameters $\rho = 30.0$ and $\nu = 0.0001$. Spatial discretization employs a $128 \times 128$ uniform grid and the time step is set to $\tau  = 5 \times 10^{-4}$ for $k=4$ and $\tau  = 1.25 \times 10^{-4}$ for $k=6$.  The simulation is advanced until the final time $T = 1.2$, capturing the temporal evolution of the thick shear layer flow.  Figure \ref{fig:sequence} and Figure \ref{fig:sequence1}  present the snapshots of vorticity
at different times for the thick shear layer. %As shown, the vorticity contours obtained with the high-order scheme closely agree with those reported in the literature using the Fourier pseudo-spectral method \cite{Jicf2024}.  %Notably, the high-order scheme achieves comparable accuracy without requiring explicit decoupling strategies, underscoring its robustness in capturing fine-scale flow features. This highlights the method's potential for high-fidelity simulations of transitional and turbulent flows.
%\clearpage % 确保从新页面开始

%\vspace{-1em} 
\begin{figure}[H]
    \centering
    \begin{subfigure}{0.3\textwidth}
        \includegraphics[width=\linewidth]{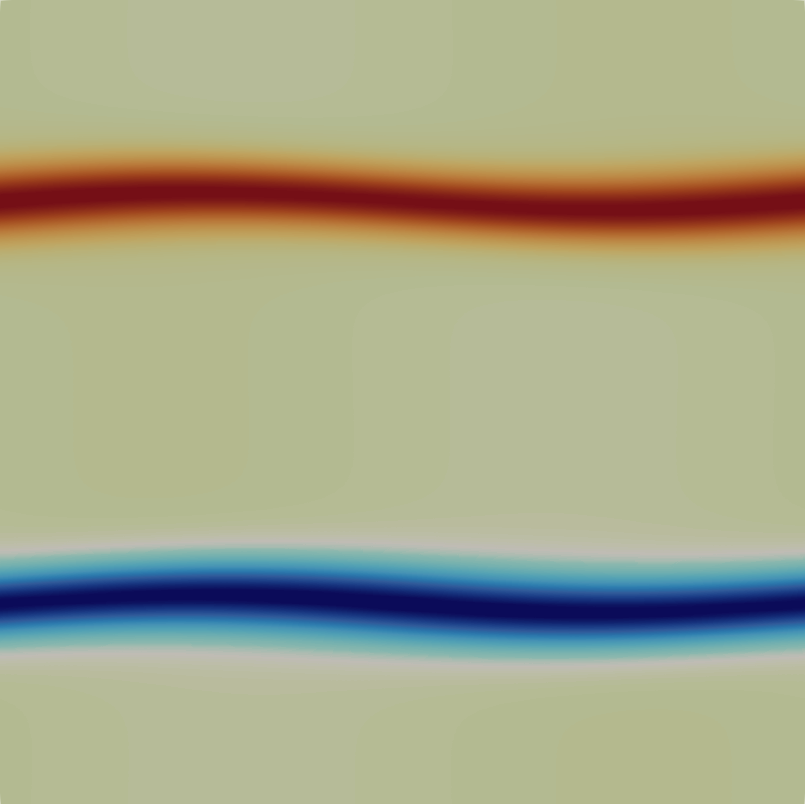}
        \caption{$t = 0.2$}
        \label{fig:sub1}
    \end{subfigure}
    \hfill
    \begin{subfigure}{0.3\textwidth}
        \includegraphics[width=\linewidth]{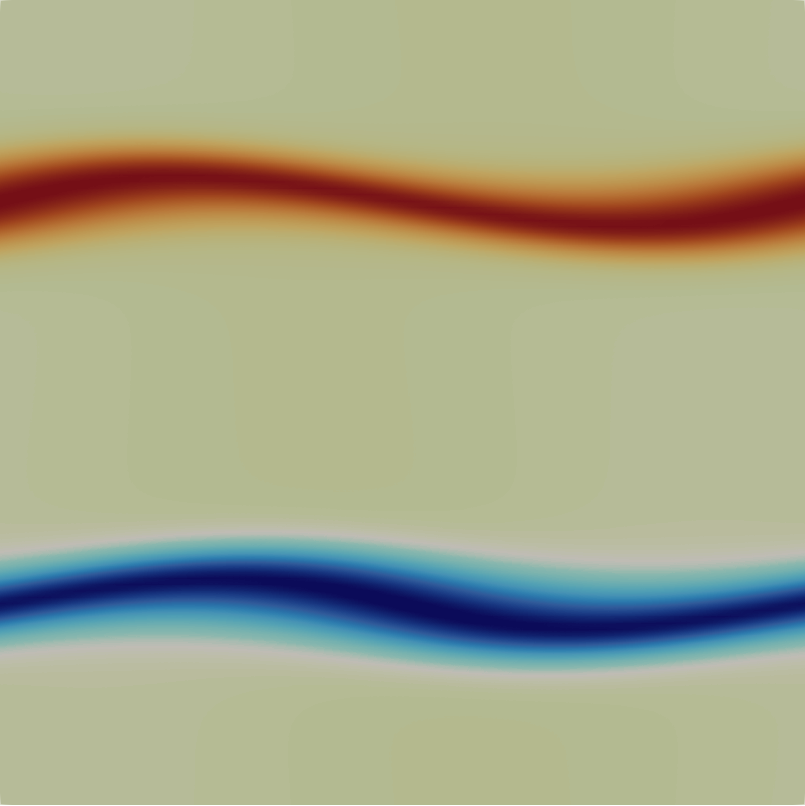}
        \caption{$t = 0.4$}
        \label{fig:sub2}
    \end{subfigure}
    \hfill
    \begin{subfigure}{0.3\textwidth}
        \includegraphics[width=\linewidth]{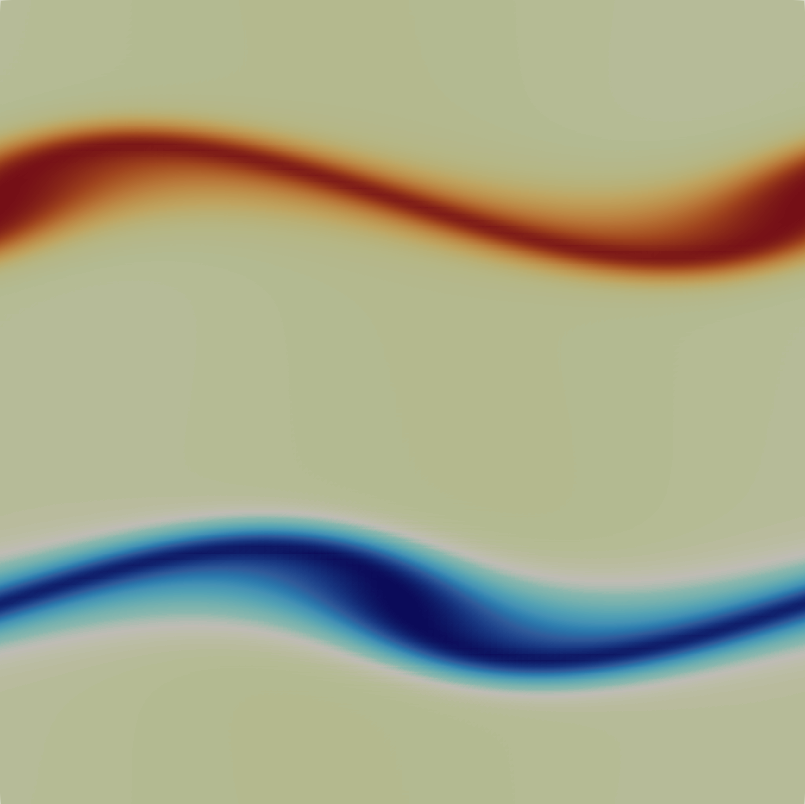}
        \caption{$t = 0.6$}
        \label{fig:sub3}
    \end{subfigure}
    \vspace{-0.2em} 
       % \vspace{0.5cm} % Add some vertical space between rows
    \begin{subfigure}{0.3\textwidth}
        \includegraphics[width=\linewidth]{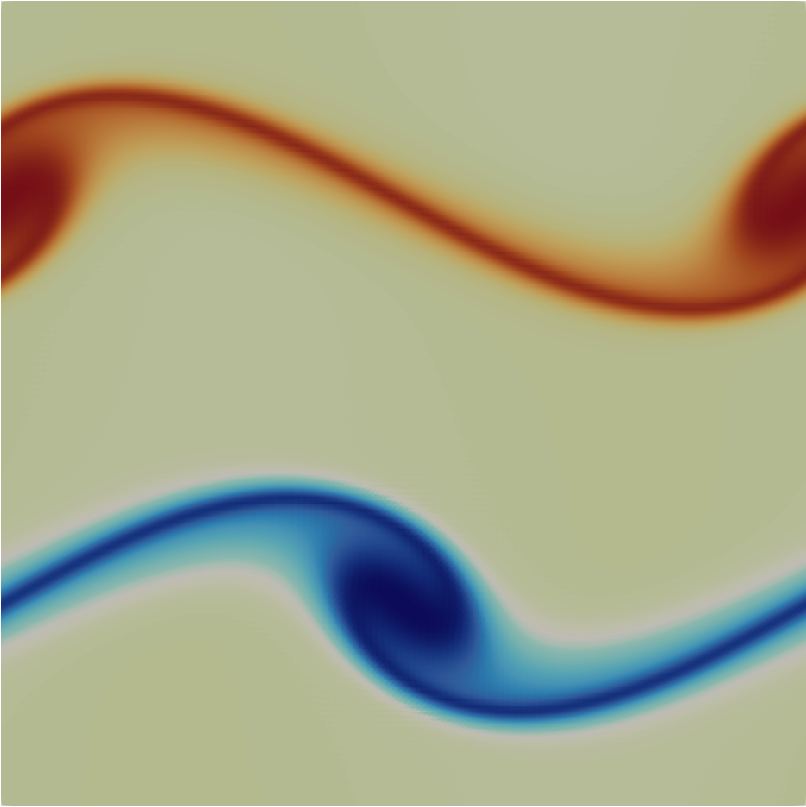}
        \caption{$t = 0.8$}
        \label{fig:sub4}
    \end{subfigure}
    \hfill
    \begin{subfigure}{0.3\textwidth}
        \includegraphics[width=\linewidth]{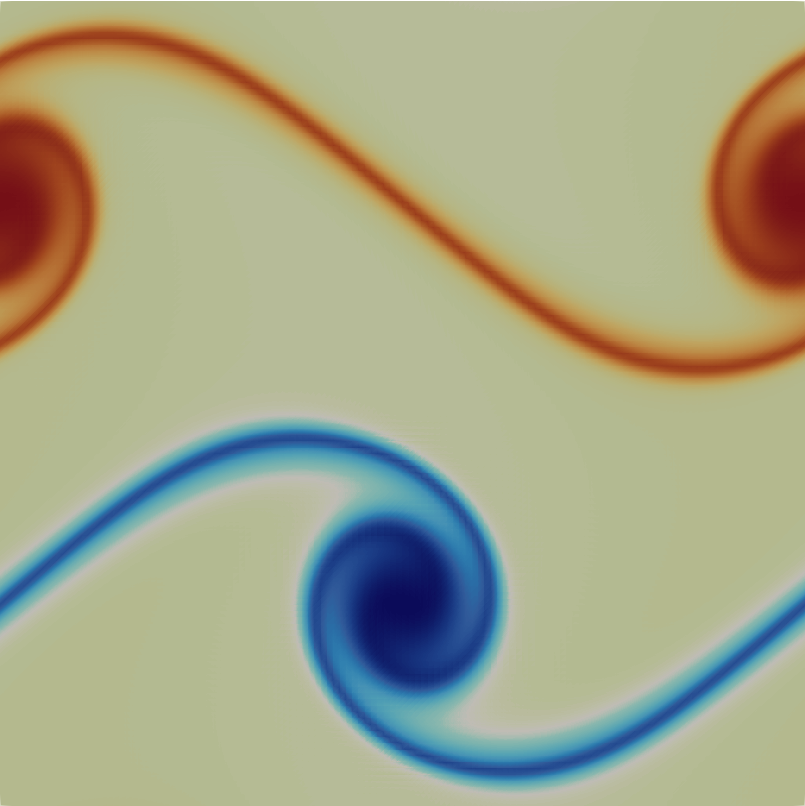}
        \caption{$t = 1.0$}
        \label{fig:sub5}
    \end{subfigure}
    \hfill
    \begin{subfigure}{0.3\textwidth}
        \includegraphics[width=\linewidth]{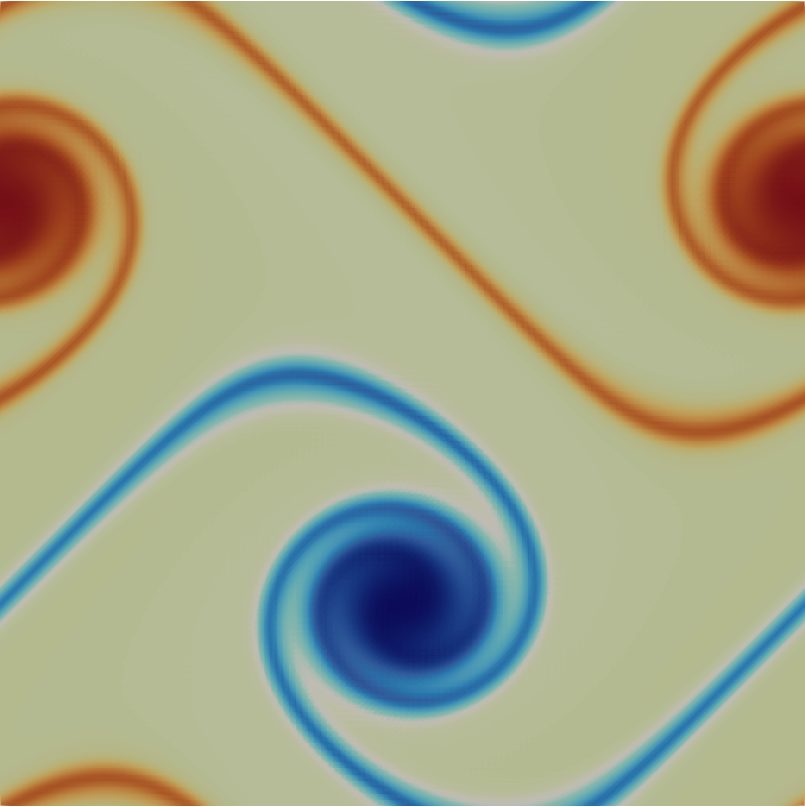}
        \caption{$t = 1.2$}
        \label{fig:sub6}
    \end{subfigure}
    \vspace{-0.5em} 
    \caption{Snapshots of vorticity for the thick shear layer problem computed by using the BDF$4$ scheme with $\nu = 0.0001$ at different times.}
    \label{fig:sequence}
\end{figure}
\vspace{-2em} 
\begin{figure}[H]
    \centering
    \begin{subfigure}[b]{0.3\textwidth}
        \includegraphics[width=\linewidth, height=\linewidth]{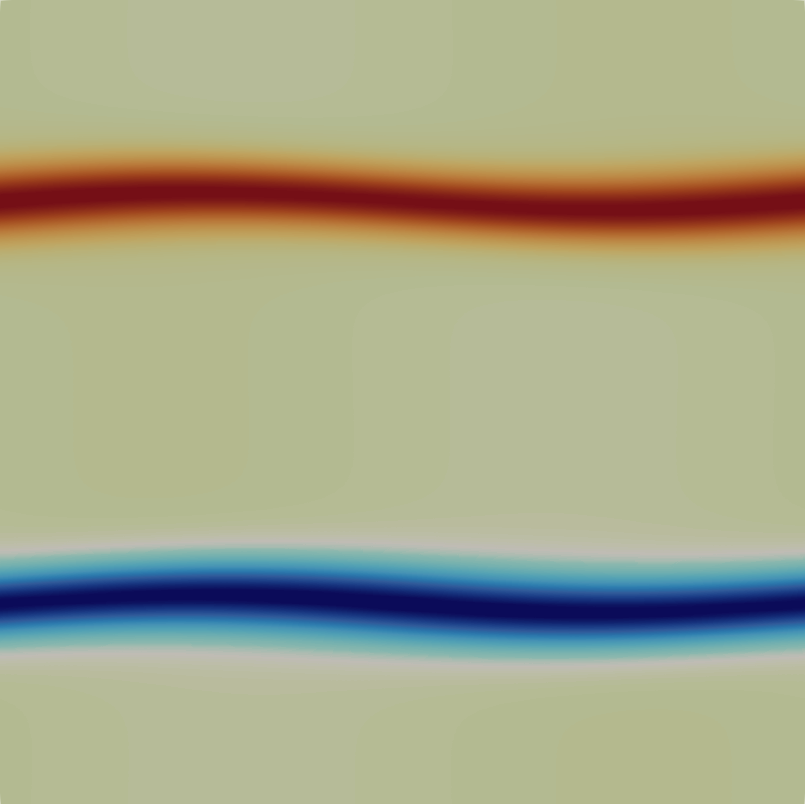}
        \caption{$t = 0.2$}
        \label{fig:thick_sub1}
    \end{subfigure}
    \hfill
    \begin{subfigure}[b]{0.3\textwidth}
        \includegraphics[width=\linewidth, height=\linewidth]{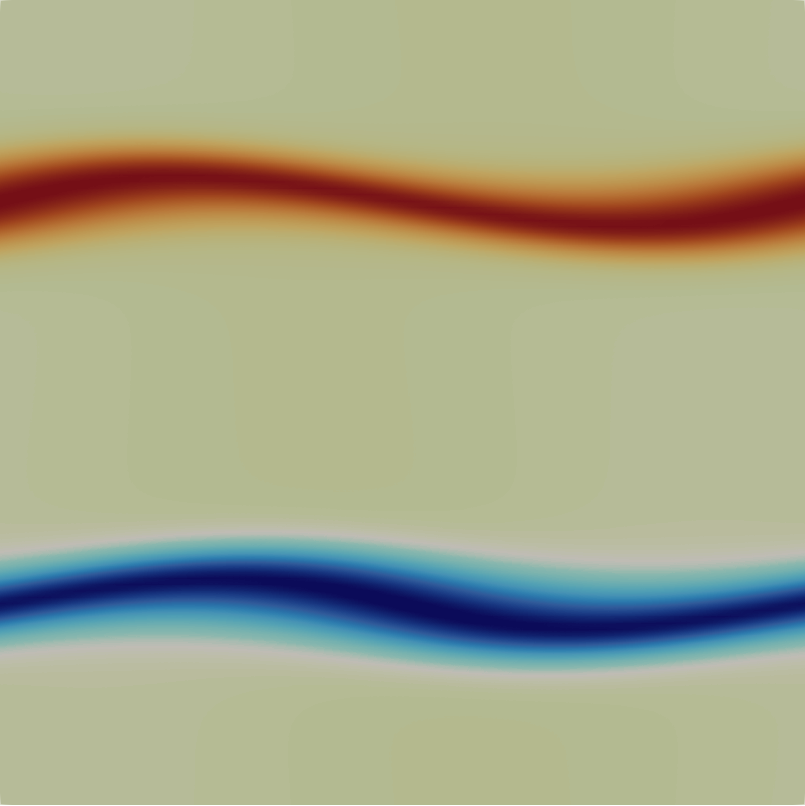}
        \caption{$t = 0.4$}
        \label{fig:thick_sub2}
    \end{subfigure}
    \hfill
    \begin{subfigure}[b]{0.3\textwidth}
        \includegraphics[width=\linewidth, height=\linewidth]{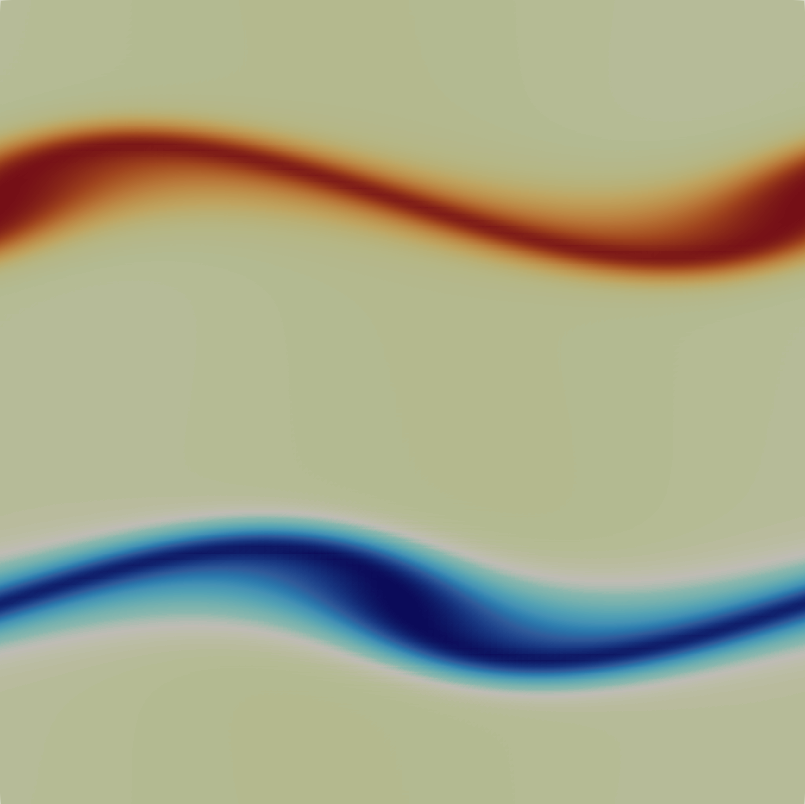}
        \caption{$t = 0.6$}
        \label{fig:thick_sub3}
    \end{subfigure}
    \vspace{-0.2em} 
        \begin{subfigure}[b]{0.3\textwidth}
        \includegraphics[width=\linewidth, height=\linewidth]{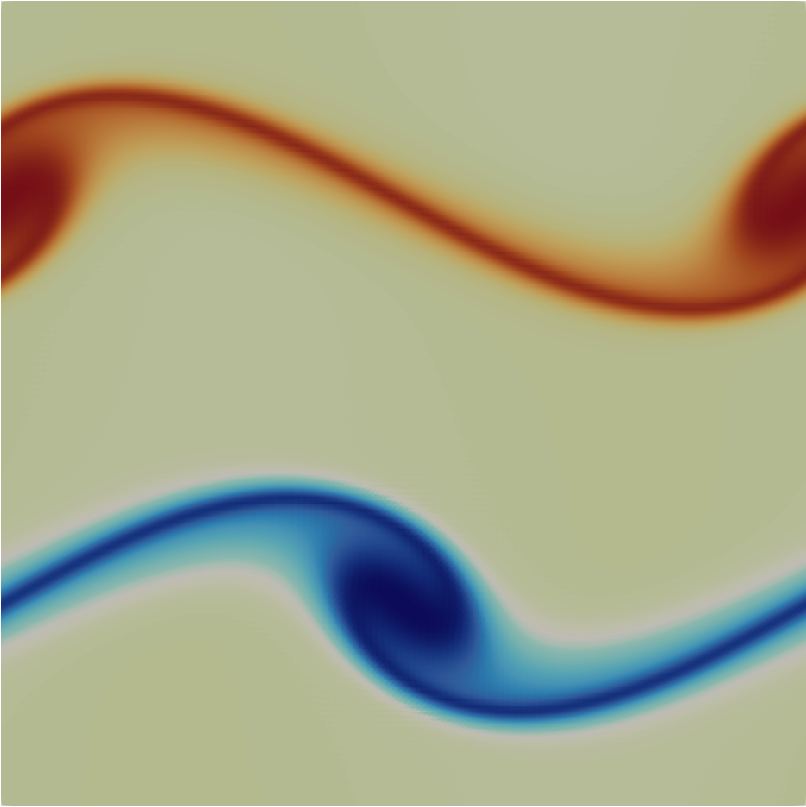}
        \caption{$t = 0.8$}
        \label{fig:thick_sub4}
    \end{subfigure}
    \hfill
    \begin{subfigure}[b]{0.3\textwidth}
        \includegraphics[width=\linewidth, height=\linewidth]{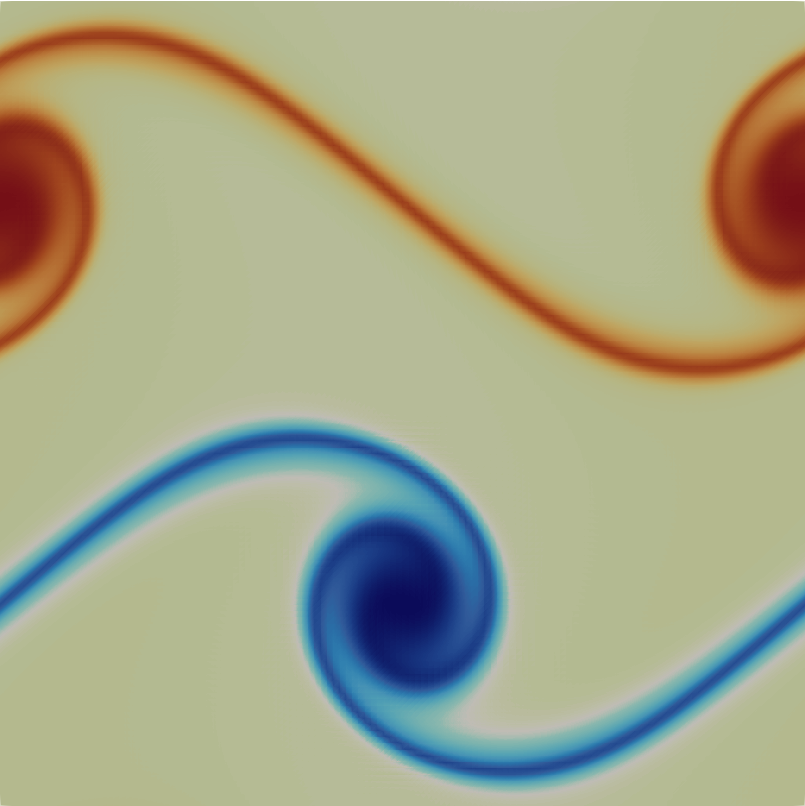}
        \caption{$t = 1.0$}
        \label{fig:thick_sub5}
    \end{subfigure}
    \hfill
    \begin{subfigure}[b]{0.3\textwidth}
        \includegraphics[width=\linewidth, height=\linewidth]{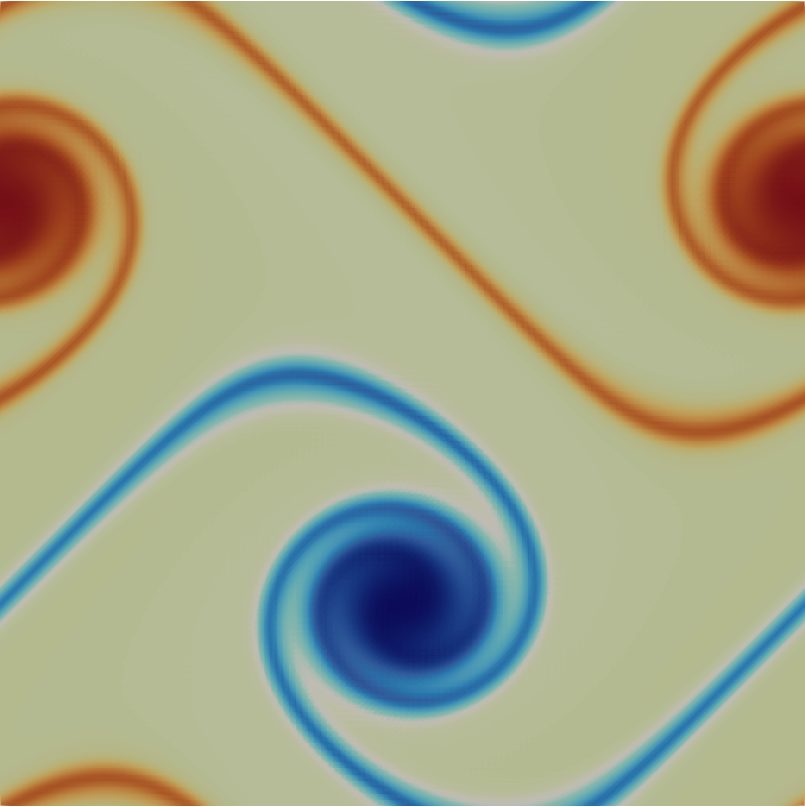}
        \caption{$t = 1.2$}
        \label{fig:thick_sub6}
    \end{subfigure}
    \vspace{-0.5em} 
\caption{Snapshots of vorticity for the thick shear layer problem computed by using the BDF$6$ scheme with $\nu$ = 0.0001 at different times.}
    \label{fig:sequence1}
    \end{figure}
% \vspace*{\fill} % Push the following text to the bottom of the page
% aaaaaaaaaaaaaaaaaaaaaaaaaa
% \vspace*{\fill} % Fill the remaining space

To further assess the performance differences between high-order and low-order numerical schemes in high Reynolds number   flows with $\nu=0.0001$,   simulations are carried out using the  first-order  temporal discretization ($k= 1$) with the time step $ \tau  = 5.0 \times 10^{-4}$. As shown in Figure \ref{fig:res_norm}, the linear system solver for  $ k = 1 $ scheme exhibits numerical instability and  blows up at approximately $t = 0.6$. To investigate the limitations of the first-order scheme, the time step is reduced to $ \tau  = 2.5 \times 10^{-4} $, and computations are performed at  $ t = 1.2 $ and $ t = 1.4 $ for $k=1$. For comparison, the same  scenarios are simulated using the fourth-order scheme ($k=4$)  with $ \tau  = 5.0 \times 10^{-4} $. As shown in Figure \ref{fig:Comparison}, the low-order scheme fails to   yield physically meaningful  solutions even with halved temporal resolution, whereas the high-order scheme maintains robust convergence and accuracy. This demonstrates that for high-Reynolds-number flows with complex structures, high-order schemes outperform low-order ones, since significantly reduced time steps are required to obtain correct solutions with low-order schemes. Furthermore, high-order schemes exhibit superior stability to low-order schemes at high Reynolds numbers.
%\vspace{-0.8em}
\begin{figure}[H]
\centering
%\hspace{-0.1\linewidth} % 向左偏移10%的页面宽度
\includegraphics[width=0.9\linewidth]{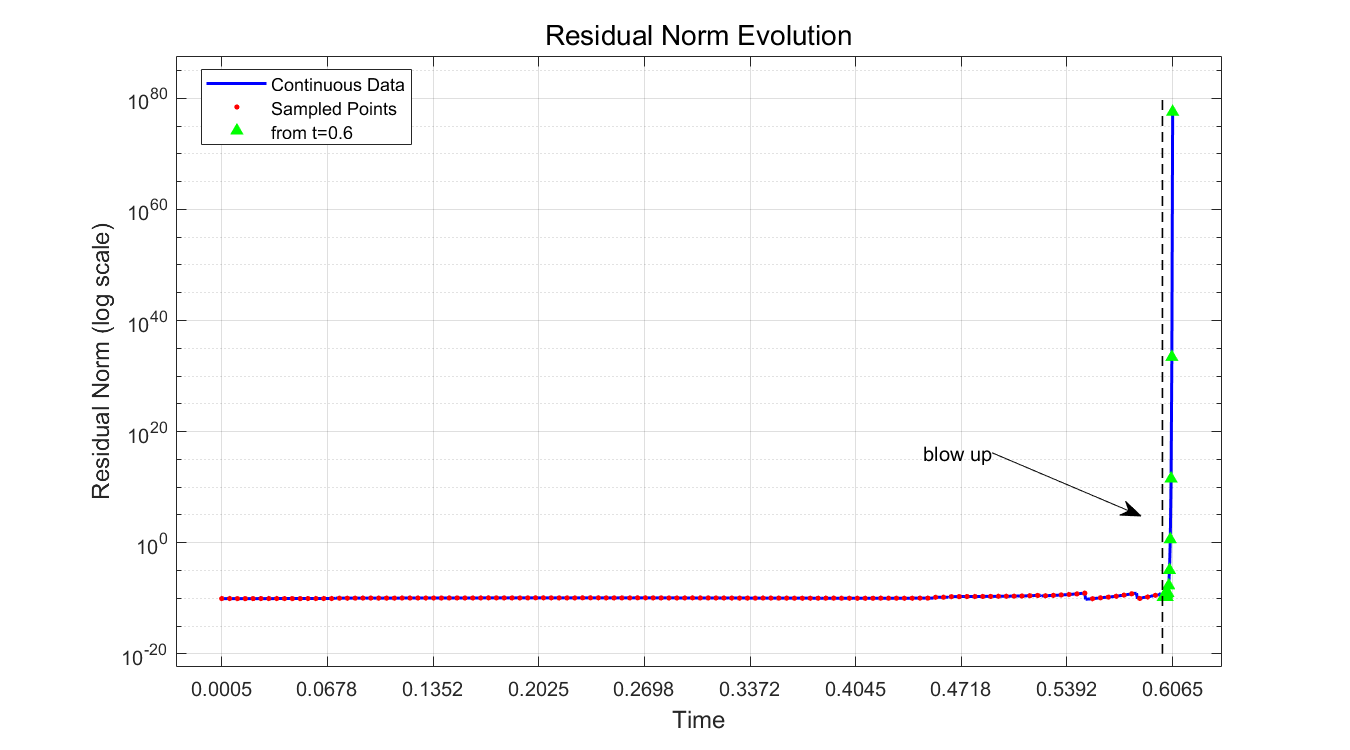}
%\vspace{-1em}
\caption{Residual variation over time for FGMRES solver, $k=1$, $\tau =5 \times 10^{-4}$}
\label{fig:res_norm}
\end{figure}

%\vspace{-1.5em}
\begin{figure}[H]
    \centering
    \begin{subfigure}[b]{0.49\textwidth}
 %       \centering
        \includegraphics[width=\linewidth, height=\linewidth]{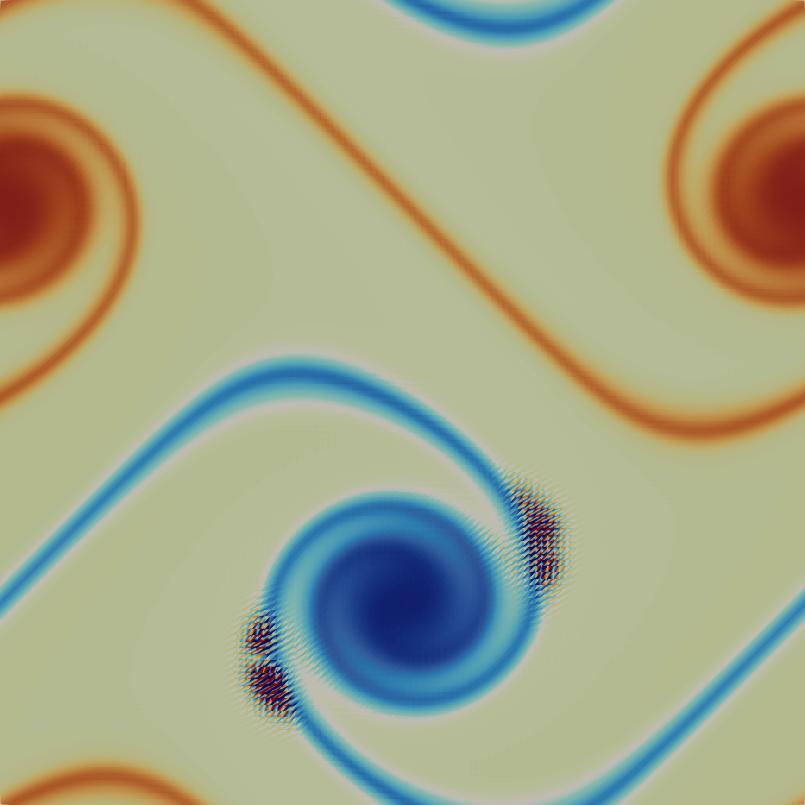} % 第一张图片缩放90%
        \caption{$k=1,t=1.2$, $\tau =2.5 \times 10^{-4}$}
        \label{fig:k200}
    \end{subfigure}
   \hfill
    \begin{subfigure}[b]{0.49\textwidth}
%        \centering
         \includegraphics[width=\linewidth, height=\linewidth]{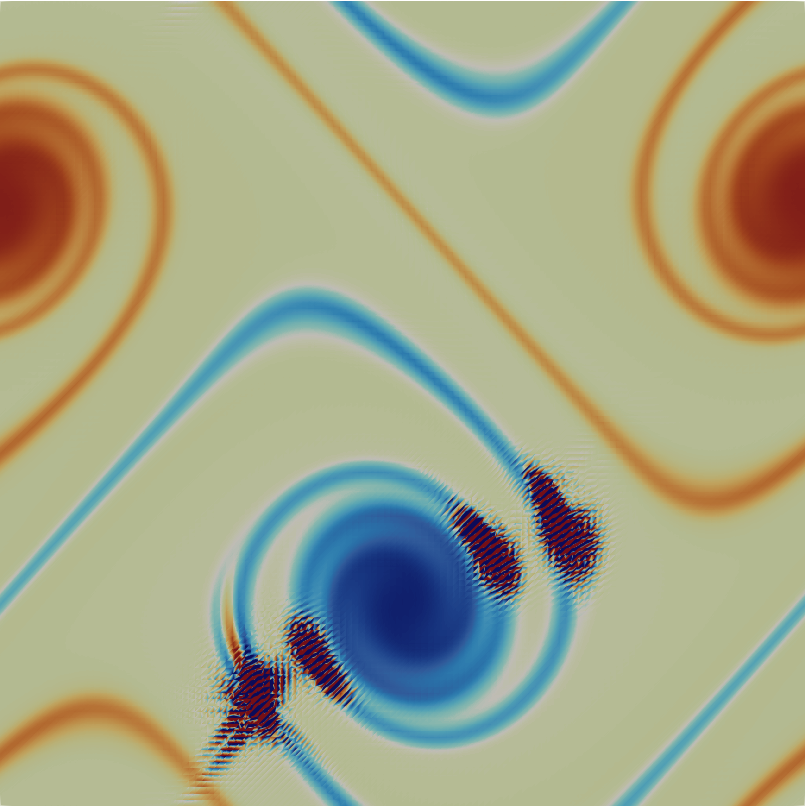}  % 第二张图片缩放90%
        \caption{$k=1,t=1.4$,$\tau =2.5 \times 10^{-4}$}
        \label{fig:k300}
    \end{subfigure} 
  %   \vspace{-0.5em}
    \vspace{0.5em}  % 添加行间距
    \hfill
    \begin{subfigure}[b]{0.49\textwidth}
%        \centering
        \includegraphics[width=\linewidth, height=\linewidth]{A12.png} % 第三张图片缩放90%
        \caption{$k=4,t=1.2$, $\tau =5 \times 10^{-4}$}
        \label{fig:k400}
    \end{subfigure}
      \hfill
  %  \vspace{-1em}  
    \begin{subfigure}[b]{0.49\textwidth}
%        \centering
        \includegraphics[width=\linewidth, height=\linewidth]{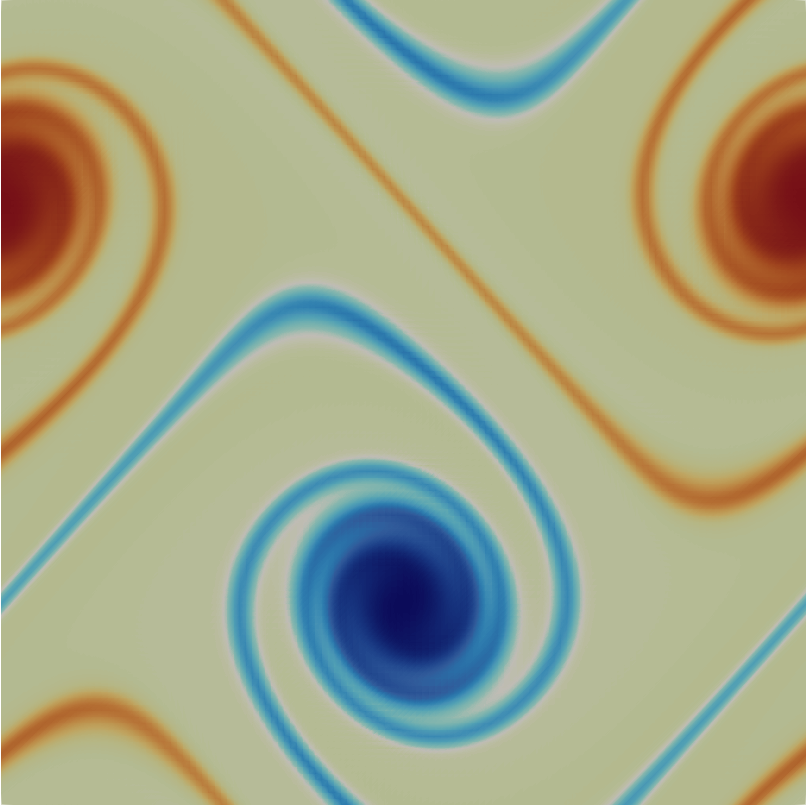}  % 第四张图片缩放90%
        \caption{$k=4,t=1.4$, $\tau =5 \times 10^{-4}$}
        \label{fig:k500}
    \end{subfigure}
 %   \vspace{-0.3em}
    \caption{Comparison of vorticity contours between high-order and low-order schemes. The solution obtained using the low-order scheme is inaccurate.}
    \label{fig:Comparison}
\end{figure}
%\vspace{-1em}

Finally, to evaluate  the capability of the high-order scheme in resolving fine-scale structures at high Reynolds numbers, we conduct simulations of the thin shear layer problem using BDF$4$  and BDF$6$  temporal discretizations with physical parameters $ \rho = 100.0 $ and $ \nu = 0.00005$. The spatial discretization is performed on a $256 \times 256$ uniform grid, and the time step is set to $\tau =1.25 \times 10^{-4}$ for $k=4$ and $\tau =5 \times 10^{-5}$ for $k=6$. As illustrated in Figure \ref{fig:sequence10} and Figure \ref{fig:sequence20}, the vorticity contours produced by the proposed method exhibit good agreement with those reported in the literature \cite{HuangFukeng2021SIAM}. These findings confirm the robustness and accuracy of the high-order algorithm in capturing complex flow dynamics under extreme Reynolds number conditions.
\vspace{-0.8\baselineskip}
\begin{figure}[H]
    \centering
    \begin{subfigure}{0.3\textwidth}
        \includegraphics[width=\linewidth]{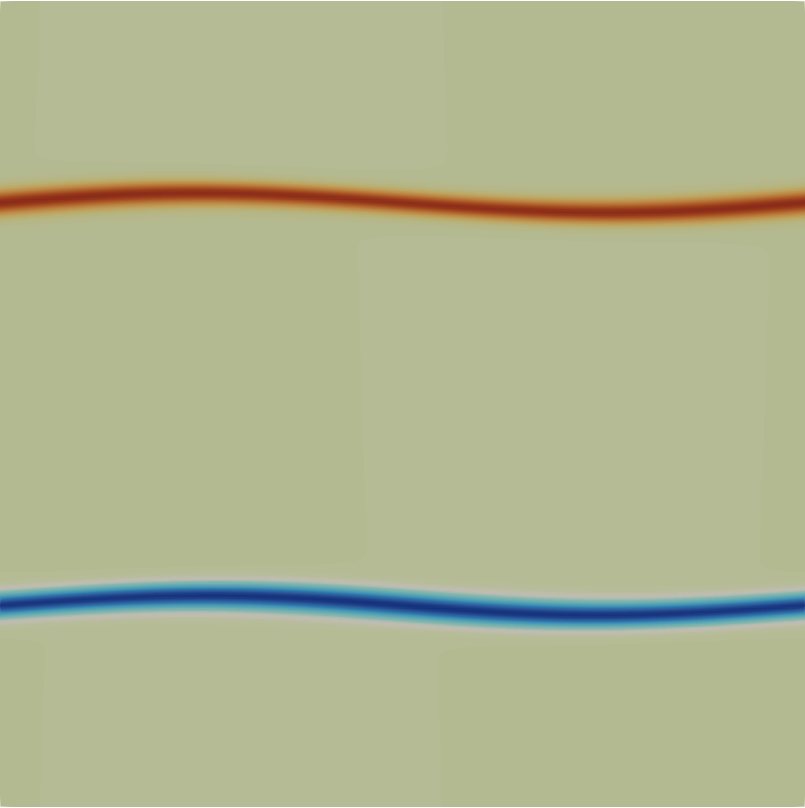}
        \caption{$t = 0.2$}
        \label{fig:sub100}
    \end{subfigure}
    \hfill
    \begin{subfigure}{0.3\textwidth}
        \includegraphics[width=\linewidth]{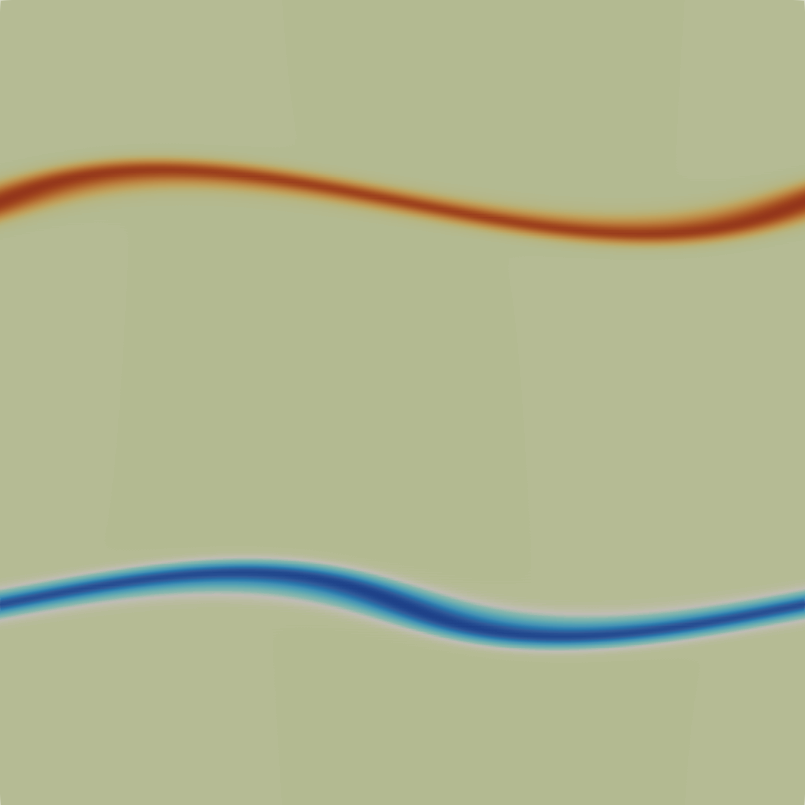}
        \caption{$t = 0.4$}
        \label{fig:sub200}
    \end{subfigure}
    \hfill
    \begin{subfigure}{0.3\textwidth}
        \includegraphics[width=\linewidth]{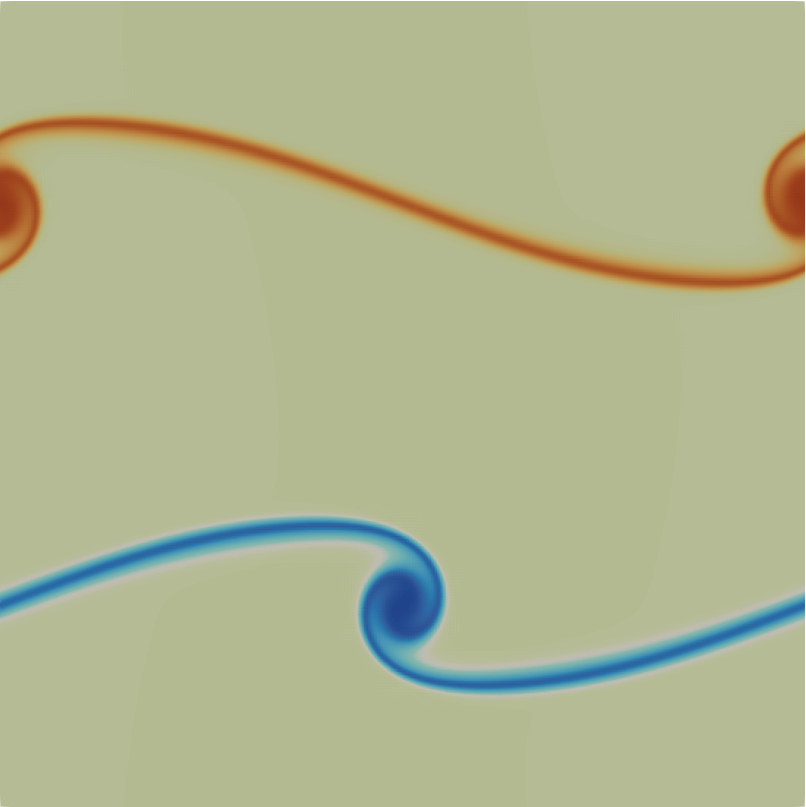}
        \caption{$t = 0.6$}
        \label{fig:sub300}
    \end{subfigure}
    
   % \vspace{0.5cm} % Add some vertical space between rows
    \vspace{-0.2em} 
    
    \begin{subfigure}{0.3\textwidth}
        \includegraphics[width=\linewidth]{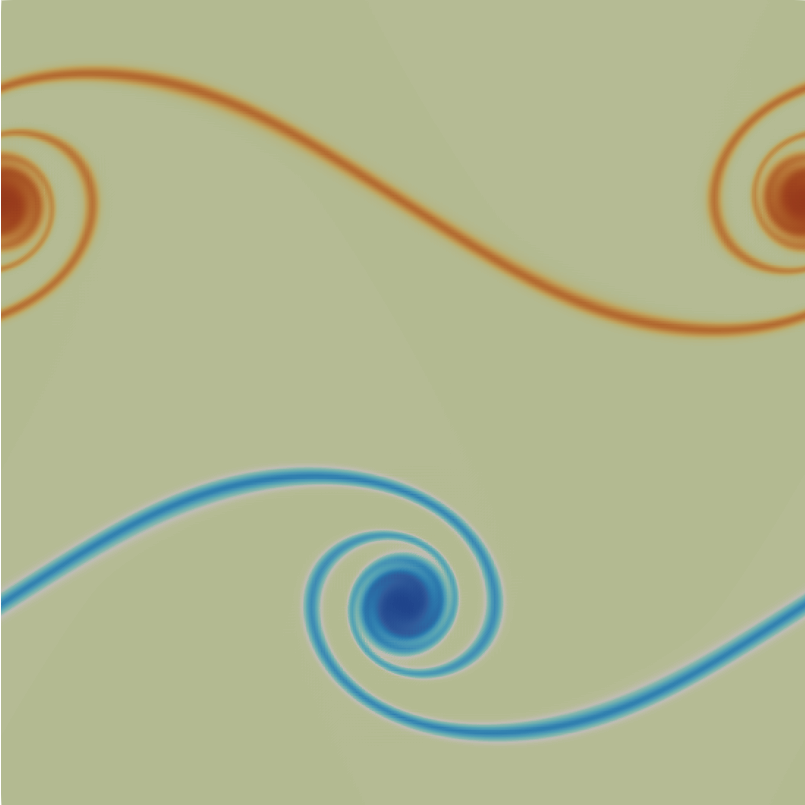}
        \caption{$t = 0.8$}
        \label{fig:sub400}
    \end{subfigure}
    \hfill
    \begin{subfigure}{0.3\textwidth}
        \includegraphics[width=\linewidth]{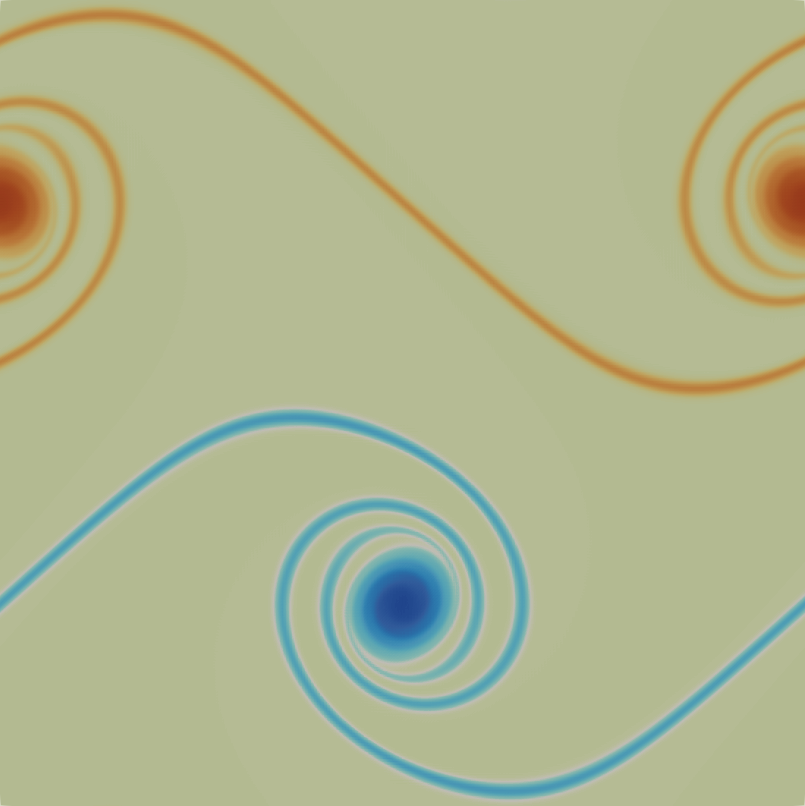}
        \caption{$t = 1.0$}
        \label{fig:sub500}
    \end{subfigure}
    \hfill
    \begin{subfigure}{0.3\textwidth}
        \includegraphics[width=\linewidth]{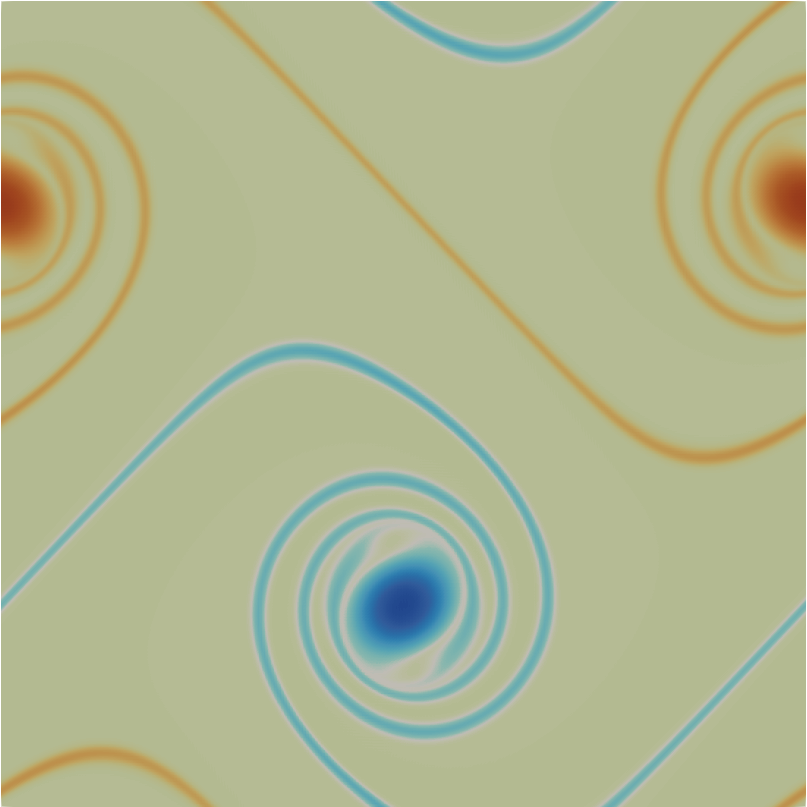}
        \caption{$t = 1.2$}
        \label{fig:sub600}
    \end{subfigure}
    \vspace{-0.2em} 
    \caption{Snapshots of vorticity for the thin shear layer problem computed by using the BDF$4$ scheme with $\nu$ = 0.00005 at different times.}
    \label{fig:sequence10}
\end{figure}
\begin{figure}[H]
    \centering
\begin{subfigure}[b]{0.3\textwidth}
        \includegraphics[width=\linewidth, height=\linewidth]{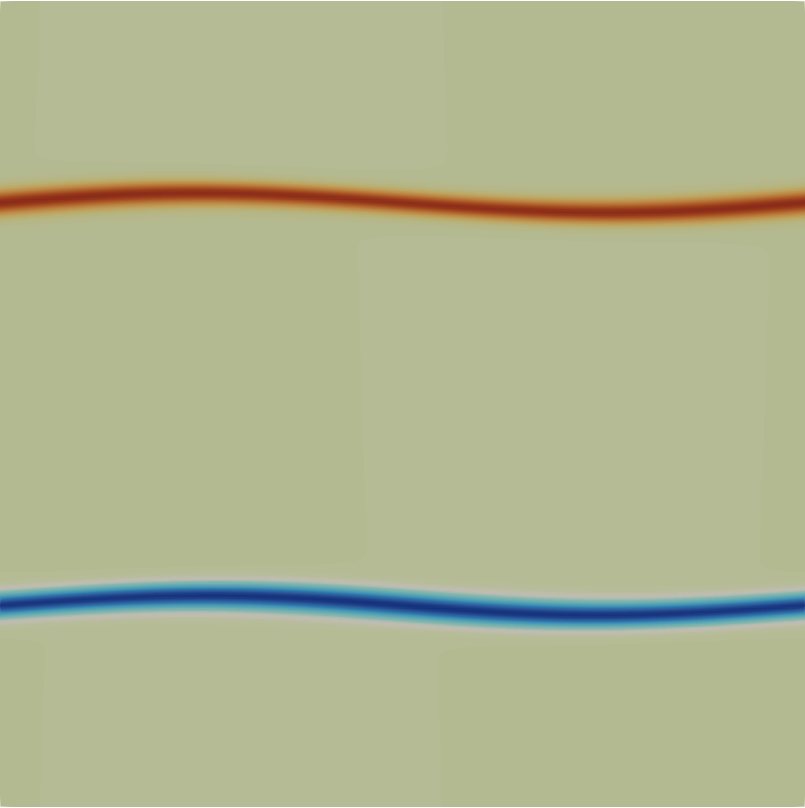}
        \caption{$t = 0.2$}
        \label{fig:thin_sub100}
    \end{subfigure}
    \hfill
    \begin{subfigure}[b]{0.3\textwidth}
        \includegraphics[width=\linewidth, height=\linewidth]{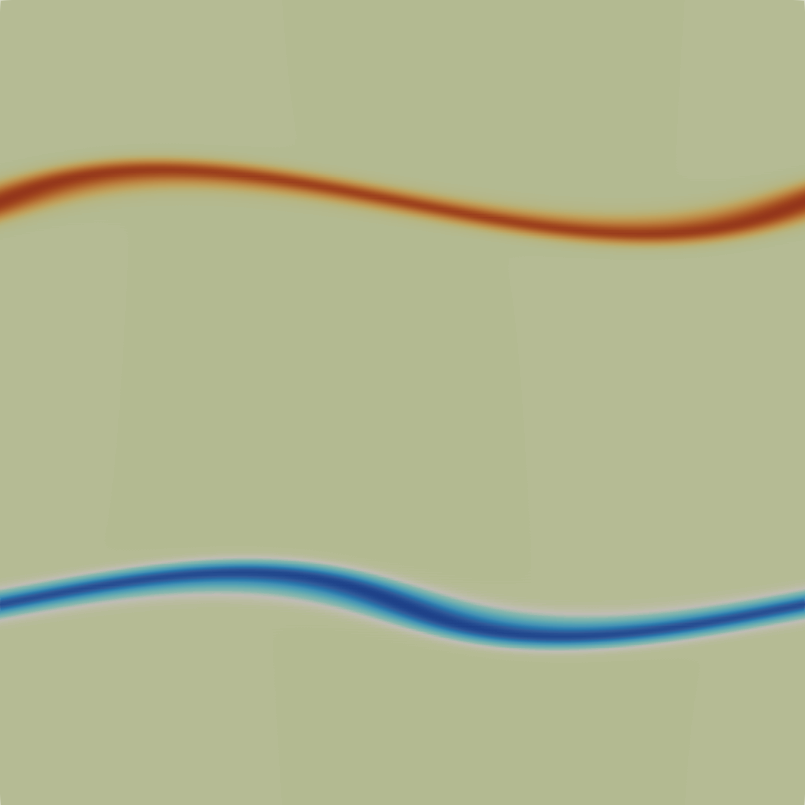}
        \caption{$t = 0.4$}
        \label{fig:thin_sub200}
    \end{subfigure}
    \hfill
    \begin{subfigure}[b]{0.3\textwidth}
        \includegraphics[width=\linewidth, height=\linewidth]{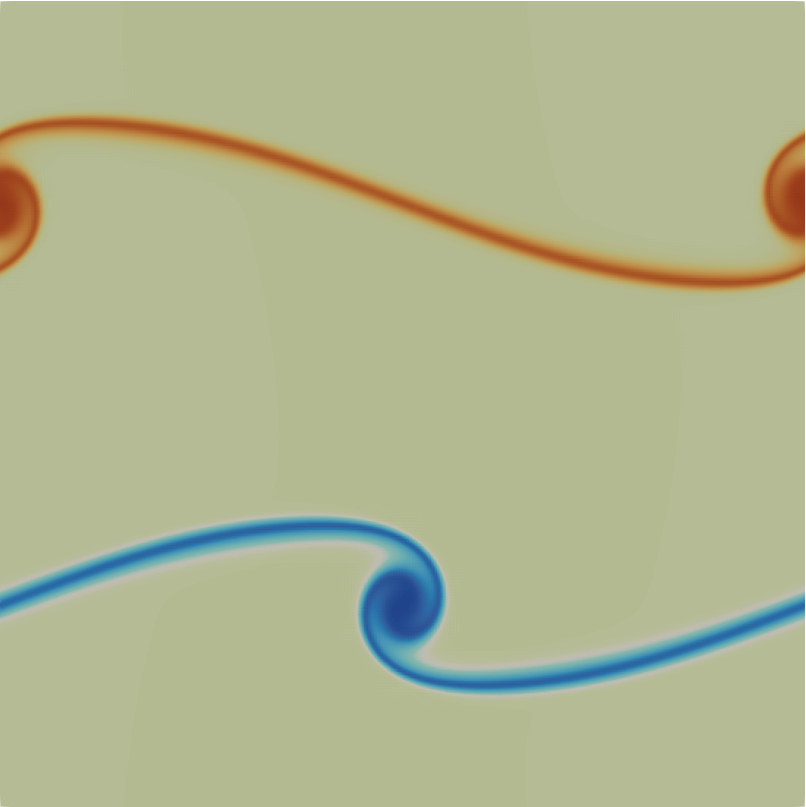}
        \caption{$t = 0.6$}
        \label{fig:thin_sub300}
    \end{subfigure}
    
   % \vspace{0.01cm}
    \vspace{-0.2em} 
    \begin{subfigure}[b]{0.3\textwidth}
        \includegraphics[width=\linewidth, height=\linewidth]{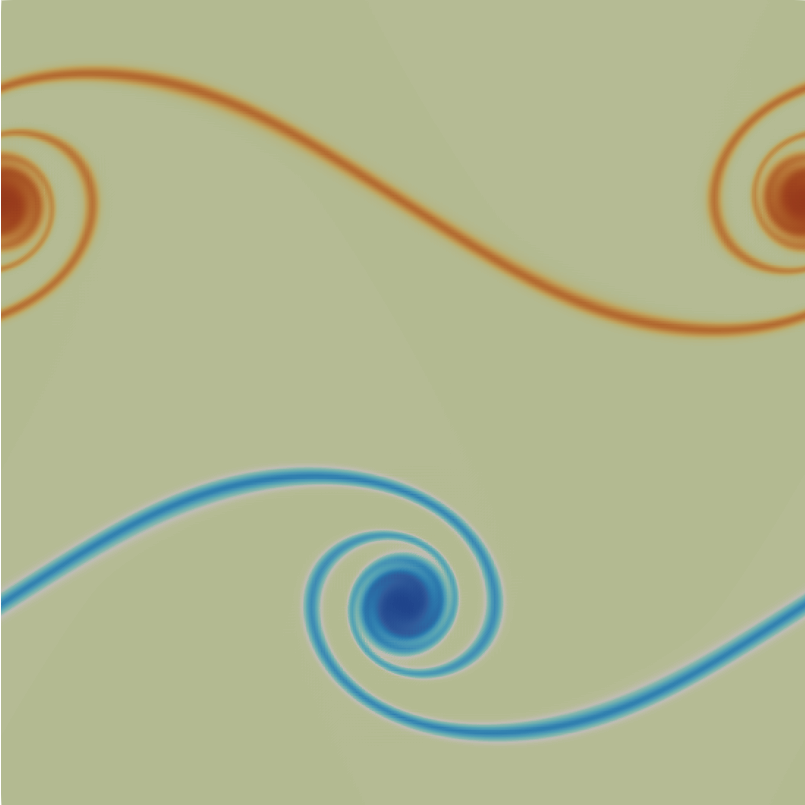}
        \caption{$t = 0.8$}
        \label{fig:thin_sub400}
    \end{subfigure}
    \hfill
    \begin{subfigure}[b]{0.3\textwidth}
        \includegraphics[width=\linewidth, height=\linewidth]{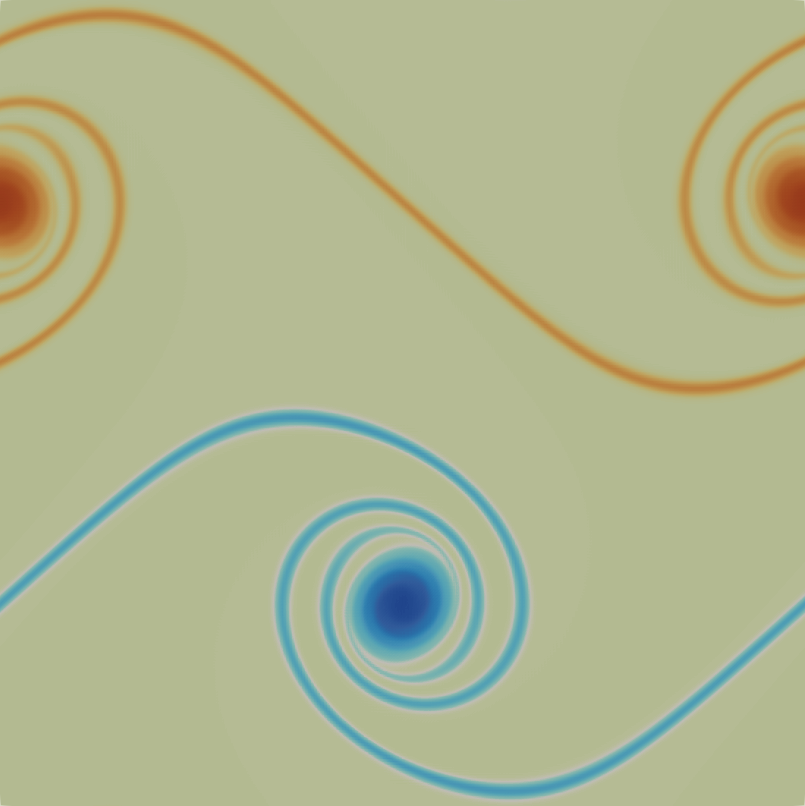}
        \caption{$t = 1.0$}
        \label{fig:thin_sub500}
    \end{subfigure}
    \hfill
    \begin{subfigure}[b]{0.3\textwidth}
        \includegraphics[width=\linewidth, height=\linewidth]{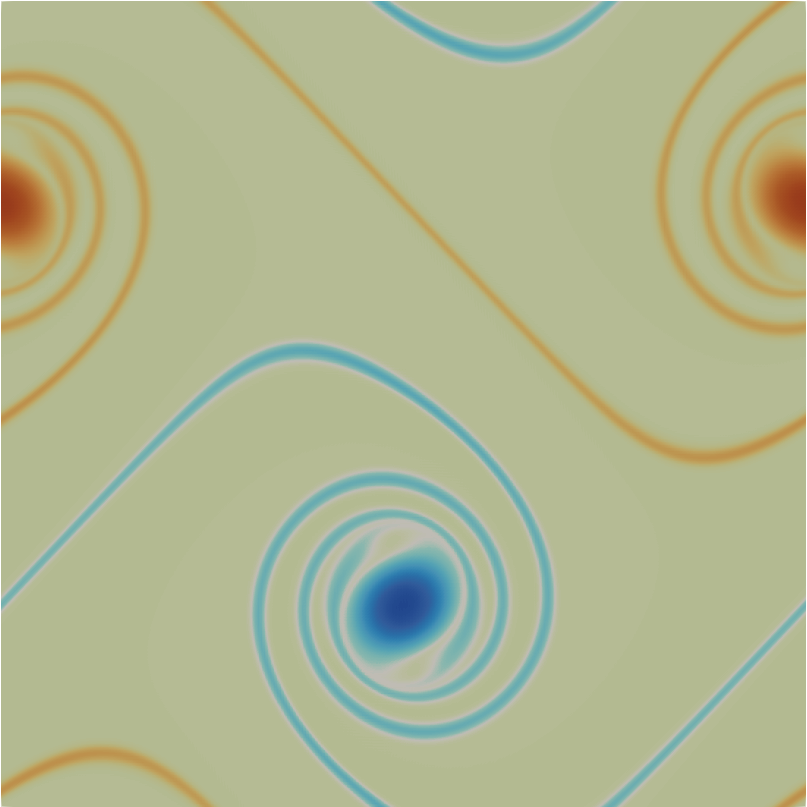}
        \caption{$t = 1.2$}
        \label{fig:thin_sub600}
    \end{subfigure}
    \vspace{-0.2em} 
    \caption{Snapshots of vorticity for the thin shear layer problem computed by using the BDF$6$ scheme with $\nu$ = 0.00005 at different times.}
    \label{fig:sequence20}
    \end{figure}

%\section{Summary}
%\label{section:summary}
%We propose and analyze a class of fully discrete finite element schemes based on IMEX-BDF$k$ time stepping for the incompressible Navier-Stokes equations with no-slip boundary conditions. 
%This work rigorously establishes the stability and uniform boundedness of the numerical solution without any restrictive condition between the time step and the mesh size. Furthermore, we derive spatially optimal $L^2$- and $H^1$-error estimates  for the velocity, as well as optimal $L^2$-error estimates for the pressure in  three-dimensional setting, while  achieving full  $k$th-order  temporal convergence for both variables  for all $k=1,\cdots,6$. That is to say,  for all $m+1\leq N$,   the following estimates hold
%\begin{equation*}
%\begin{split}
%&\Vert\bm{u}^{m+1}-\bm{u}_h^{m+1}\Vert_0%+\tau \frac{1-\mu_k^2}{2}\sum_{n=1}^m\nu\Vert\nabla(\bm{u}^n-\bm{u}_h^n)\Vert_0^2
%\approx \mathcal{O}\left(h^{l+1}+\tau^{k}\right),\quad
%\Vert\nabla\bm{u}^{m+1}-\nabla\bm{u}_h^{m+1}\Vert_0\approx \mathcal{O}\left(h^{l}+\tau^{k}\right),\\[2mm]
%&\tau  \sum_{n=k-1}^{m}\Vert p^{n+1}-p_h^{n+1}\Vert_0^2\approx \mathcal{O}\left(h^{2l}+\tau^{2k}\right).
%\end{split}
%\end{equation*}
%To validate the theoretical analysis, a series of numerical experiments are conducted. These simulations confirm not only the validity of the theoretical convergence rates but also the stability and efficiency of the proposed schemes. 

\section{Conclusion}
\label{section:summary}

In this work, we have developed and  analyzed a class of fully discrete IMEX-BDF$k$ finite element schemes for the incompressible Navier-Stokes equations,  for  temporal orders $k=1,\cdots,6$.  The analysis provides a unified finite element  framework for BDF$k$ time discretizations and establishes stability as well as optimal error estimates for the  fully discrete approximations.

By treating the nonlinear convection term explicitly while maintaining an implicit discretization for the viscous and pressure terms, the proposed formulation yields a sequence of  linear problems at each time step, thereby reducing computational cost while preserving  energy stability. From a theoretical perspective, the main contribution lies in the development of a unified stability and error analysis applicable to all BDF orders up to six. In particular, we prove stability and uniform boundedness of the fully discrete solutions without any restrictive condition between the time step and the mesh size. Moreover, optimal error estimates are derived in the three-dimensional setting, including optimal $L^{2}$- and $H^{1}$-convergence rates for the velocity and optimal $L^{2}$-convergence rates for the pressure. The schemes achieve full $k$th-order  temporal accuracy for both velocity and pressure.

Although BDF schemes have been extensively studied in the literature, the available rigorous analyses for higher-order fully discrete finite element approximations are considerably more restrictive. Existing results for third-, fourth-, and fifth-order BDF finite element schemes typically rely on CFL-type restrictions of the form  $\tau\leq Ch^{\alpha}$. By contrast, the present work removes such restrictions  and establishes CFL-free  stability together with optimal convergence for fully discrete IMEX-BDF finite element schemes of orders $k=1,\cdots,6$. To the best of our knowledge, this is the first rigorous stability and convergence analysis of a sixth-order IMEX-BDF finite element discretization of the incompressible  Navier-Stokes equations.  The analysis of the sixth-order scheme requires several new ingredients, including a telescoping $G$-energy construction, a finite-interval Toeplitz coercivity estimate, and projection-based consistency and multistep identities that avoid unnecessarily strong temporal regularity assumptions on the pressure and retain full sixth-order accuracy under the explicit treatment of convection.

%The numerical experiments  support  the theoretical convergence rates and demonstrate the practical performance of the  high-order schemes. In particular, the fourth- and sixth-order methods accurately resolve the evolution of both thick and thin double shear layers at small viscosities. The sixth-order scheme remains robust in the high-Reynolds-number tests and captures fine-scale vortical structures without visible loss of coherence. These results indicate that high-order IMEX--BDF discretizations are promising for large-scale simulations of convection-dominated incompressible flows with complex multiscale dynamics.

Beyond the finite element pairs considered herein, the  structure of the analysis suggests that the methodology can be extended to a broader class of inf-sup stable discretizations. In particular, it would be of considerable interest to investigate its extension to divergence-conforming $H(\mathrm{div})$ finite element methods, such as Raviart-Thomas and Brezzi-Douglas-Marini discretizations, which provide a natural framework for exactly divergence-free velocity approximations and pressure-robust formulations.   Overall, the combination of high-order temporal accuracy, stability, and computational efficiency makes the IMEX-BDF$k$ framework developed in this work particularly attractive for the simulation of incompressible flows with complex multiscale dynamics, especially in high-Reynolds-number regimes. Numerical experiments are presented to illustrate the theoretical results and to assess the practical performance of the proposed schemes.  
 In particular, the sixth-order approximation remains  reliable for large Reynolds numbers in the numerical tests, indicating its potential for large-scale computations of convection-dominated incompressible flows and motivating further investigations in high-Reynolds-number and turbulence-resolving simulations.

\begin{appendix}
\section{Proof of the discrete Sobolev estimates}\label{appen}

\begin{proof}[Proof of Lemma \ref{lemmaSobolev}]
The first estimate follows immediately from the three-dimensional Sobolev embedding
\begin{align*}
 \bm{H}_0^1(\Omega)\hookrightarrow \bm{L}^6(\Omega),
\end{align*}
because $\bm{V}_h\subset \bm{H}_0^1(\Omega)$. Hence,
\begin{align*}
                 \|\bm{v}_h\|_{0,6}\le C\|\nabla \bm{v}_h\|_0.
\end{align*}
We next prove the second estimate in Lemma \ref{lemmaSobolev}. Set
\begin{align*}
  \bm{g}_h:=\mathcal{A}_h\bm{v}_h\in \bm{V}_h.
\end{align*}
Let $(\bm{z},r)\in (\bm{H}_0^1(\Omega)\cap \bm{H}^2(\Omega))
\times (H^1(\Omega)\cap Q)$ be the solution of the continuous Stokes problem
\[
\begin{aligned}
 -\Delta\bm{z}+\nabla r &=\bm{g}_h &&\text{in }\Omega,\\
 \nabla\cdot \bm{z}&=0 &&\text{in }\Omega,\\
 \bm{z}&=\bm{0} &&\text{on }\partial\Omega.
\end{aligned}
\]
By the Stokes regularity assumption,
\begin{equation}
 \|\bm{z}\|_{2,2}+\|r\|_{1,2} \le C\|\bm{g}_h\|_0.
 \label{eq:stokes-regularity-discrete-Sobolev}
\end{equation}
Let $(\bm{z}_h,r_h)\in \bm{X}_h^l\times Q_h^{l-1}$ be the corresponding mixed finite element approximation:
\[
\begin{aligned}
 (\nabla \bm{z}_h,\nabla \bm{w}_h)
 -(r_h,\nabla\cdot \bm{w}_h)
   &=(\bm{g}_h,\bm{w}_h)
   &&\forall \bm{w}_h\in \bm{X}_h^l,\\
 (\nabla\cdot \bm{z}_h,q_h)&=0
   &&\forall q_h\in Q_h^{l-1}.
\end{aligned}
\]
In particular, $\bm{z}_h\in \bm{V}_h$. Taking $\bm{w}_h\in \bm{V}_h$ eliminates the
discrete pressure and gives
\begin{align*}
  (\nabla \bm{z}_h,\nabla \bm{w}_h)=(\bm{g}_h,\bm{w}_h).
\end{align*}
On the other hand, the definition of $\mathcal{A}_h$ gives
\begin{align*}
 (\bm{g}_h,\bm{w}_h)
 =(\mathcal A_h\bm{v}_h,\bm{w}_h)
 =(\nabla \bm{v}_h,\nabla \bm{w}_h)
 \qquad\forall \bm{w}_h\in \bm{V}_h.
\end{align*}
Consequently,
\begin{align*}
 (\nabla(\bm{z}_h-\bm{v}_h),\nabla \bm{w}_h)=0
 \qquad\forall \bm{w}_h\in \bm{V}_h.
\end{align*}
Choosing $\bm{w}_h=\bm{z}_h-\bm{v}_h$ shows that
\begin{align*}
  \bm{z}_h=\bm{v}_h.
\end{align*}
Since $(\bm z_h,r_h)$ is the mixed Galerkin approximation of $(\bm z,r)$, the standard mixed finite element error estimate yields
\begin{align*}
\|\nabla(\bm z-\bm z_h)\|_0\le Ch\bigl(\|\bm z\|_{2,2}+\|r\|_{1,2}\bigr).
\end{align*}
Since $\bm z_h=\bm v_h$, the Stokes regularity estimate implies
\begin{equation}
 \|\nabla(\bm{z}-\bm{v}_h)\|_0
 \le Ch\bigl(\|\bm{z}\|_{2,2}+\|r\|_{1,2}\bigr)
 \le Ch\|\bm{g}_h\|_0.
 \label{eq:stokes-H1-error}
\end{equation}
Let $I_h\bm{z}\in \bm{X}_h^l$ be a Scott-Zhang-type quasi-interpolant that preserves the homogeneous boundary condition. On a shape-regular mesh, it satisfies
\begin{equation}
 \|\nabla(\bm{z}-I_h\bm{z})\|_0\le Ch\|\bm{z}\|_{2,2},
 \qquad
 \|\nabla I_h\bm{z}\|_{0,3}\le C\|\nabla \bm{z}\|_{0,3}.
 \label{eq:SZ-properties}
\end{equation}
Combining (\ref{eq:stokes-H1-error}), (\ref{eq:SZ-properties}), and the Stokes regularity estimate gives
\[
\begin{aligned}
 \|\nabla(\bm{v}_h-I_h\bm{z})\|_0
 &\le \|\nabla(\bm{v}_h-\bm{z})\|_0
      +\|\nabla(\bm{z}-I_h\bm{z})\|_0 \le Ch\|\bm{g}_h\|_0.
\end{aligned}
\]
Since $\bm{v}_h-I_h\bm{z}\in \bm{X}_h^l$, the three-dimensional finite element inverse inequality gives
\begin{align}
 \|\nabla(\bm{v}_h-I_h\bm{z})\|_{0,3}
 \le Ch^{-1/2}\|\nabla(\bm{v}_h-I_h\bm{z})\|_0
 \le Ch^{1/2}\|\bm{g}_h\|_0.
 \label{eq:inverse-L3}
\end{align}
We also need an inverse estimate for the discrete Stokes operator. Since $\bm{g}_h\in \bm{V}_h$,
\[
\begin{aligned}
 \|\bm{g}_h\|_0
 &=\sup_{0\ne \bm{w}_h\in \bm{V}_h}
   \frac{(\bm{g}_h,\bm{w}_h)}{\|\bm{w}_h\|_0} \\
 &=\sup_{0\ne \bm{w}_h\in \bm{V}_h}
   \frac{(\nabla \bm{v}_h,\nabla \bm{w}_h)}{\|\bm{w}_h\|_0}\\
 &\le \|\nabla \bm{v}_h\|_0
   \sup_{0\ne \bm{w}_h\in \bm{V}_h}
   \frac{\|\nabla \bm{w}_h\|_0}{\|\bm{w}_h\|_0}\\
 &\le Ch^{-1}\|\nabla \bm{v}_h\|_0.
\end{aligned}
\]
Thus,
\begin{equation}
   h\|\bm{g}_h\|_0\le C\|\nabla \bm{v}_h\|_0.
 \label{eq:Ah-inverse}
\end{equation}
In three dimensions, interpolation between $L^2$ and $L^6$, followed by the Sobolev embedding $H^2(\Omega)\hookrightarrow W^{1,6}(\Omega)$, gives
\begin{align*}
 \|\nabla \bm{z}\|_{0,3}
 &\le
 C\|\nabla \bm{z}\|_0^{1/2}
  \|\nabla \bm{z}\|_{0,6}^{1/2}\le
 C\|\nabla \bm{z}\|_0^{1/2}
  \|\bm{z}\|_{2,2}^{1/2}.
\end{align*}
Moreover, by (\ref{eq:stokes-H1-error}) and (\ref{eq:Ah-inverse}),
\[
\begin{aligned}
 \|\nabla \bm{z}\|_0
 &\le \|\nabla \bm{v}_h\|_0+\|\nabla(\bm{z}-\bm{v}_h)\|_0\\
 &\le \|\nabla \bm{v}_h\|_0+Ch\|\bm{g}_h\|_0\le C\|\nabla \bm{v}_h\|_0.
\end{aligned}
\]
Together with the Stokes regularity estimate, this implies
\begin{equation}
 \|\nabla \bm{z}\|_{0,3}
 \le
 C\|\nabla \bm{v}_h\|_0^{1/2}\|\bm{g}_h\|_0^{1/2}.
 \label{eq:continuous-reconstruction-L3}
\end{equation}
Finally, using (\ref{eq:SZ-properties}), (\ref{eq:inverse-L3}), and (\ref{eq:continuous-reconstruction-L3}), we obtain
\[
\begin{aligned}
 \|\nabla \bm{v}_h\|_{0,3}
 &\le
 \|\nabla(\bm{v}_h-I_h\bm{z})\|_{0,3}
 +\|\nabla I_h\bm{z}\|_{0,3}\\
 &\le
 Ch^{1/2}\|\bm{g}_h\|_0
 +C\|\nabla \bm{z}\|_{0,3}\\
 &\le
 C(h\|\bm{g}_h\|_0)^{1/2}\|\bm{g}_h\|_0^{1/2}
 +C\|\nabla \bm{v}_h\|_0^{1/2}\|\bm{g}_h\|_0^{1/2}\\
 &\le
 C\|\nabla \bm{v}_h\|_0^{1/2}\|\bm{g}_h\|_0^{1/2}.
\end{aligned}
\]
Recalling that $\bm{g}_h=\mathcal{A}_h\bm{v}_h$ proves
\begin{align*}
 \|\nabla \bm{v}_h\|_{0,3}
 \le C\|\nabla \bm{v}_h\|_0^{1/2}
       \|\mathcal{A}_h\bm{v}_h\|_0^{1/2}.
\end{align*}
The constant $C$ is independent of $h$ and depends only on the domain, the Stokes regularity constant, the polynomial degree, the discrete inf-sup constant, and the shape-regularity and quasi-uniformity constants of the mesh family.
\end{proof}
\end{appendix}

%\begin{acknowledgements}
%\end{acknowledgements}



%\bibliographystyle{siam}
%\bibliography{mybib.bib}

%\bibliography{IMANUM-refs}


\bibliographystyle{amsplain}
\bibliography{mybib.bib}

@article{Qiuhe2026ima,
    author = {Qiu, Changxin and Wang, Kai and He, Xiaoming and Lin, Yanping},
    title = {Analysis of a joint Stokes–Darcy Ritz-projection and multi-step BDF schemes for decoupling the unsteady Navier–Stokes–Darcy model},
    journal = {IMA Journal of Numerical Analysis},
    pages = {drag037},
    year = {2026},
    month = {06},
    issn = {0272-4979},
    doi = {10.1093/imanum/drag037},
    url = {https://doi.org/10.1093/imanum/drag037},
    eprint = {https://academic.oup.com/imajna/advance-article-pdf/doi/10.1093/imanum/drag037/68536182/drag037.pdf},
}

@misc{Bosco2025arxiv,
      title={Error analysis of BDF schemes for the evolutionary incompressible Navier--Stokes equations}, 
      author={Bosco García-Archilla and V. John and Julia Novo},
      year={2025, arXiv},
      eprint={2506.16917},
      archivePrefix={arXiv},
      primaryClass={math.NA},
      url={https://arxiv.org/abs/2506.16917}, 
}

@article {HanYongbin2023mc,
    AUTHOR = {Han, Yongbin and Hou, Yanren and Zhang, Min},
     TITLE = {Analysis of divergence-free {$H^1$} conforming {FEM} with
              {IMEX}-{SAV} scheme for the {N}avier-{S}tokes equations at
              high {R}eynolds number},
   JOURNAL = {Math. Comp.},
  FJOURNAL = {Mathematics of Computation},
    VOLUME = {92},
      YEAR = {2023},
    NUMBER = {340},
     PAGES = {557--582},
      ISSN = {0025-5718,1088-6842},
   MRCLASS = {65M60 (35Q30 65M12)},
  MRNUMBER = {4524102},
MRREVIEWER = {Zhu\ Wang},
}

@article {Allendes2021sisc,
    AUTHOR = {Allendes, Alejandro and Barrenechea, Gabriel R. and Novo,
              Julia},
     TITLE = {A divergence-free stabilized finite element method for the
              evolutionary {N}avier-{S}tokes equations},
   JOURNAL = {SIAM J. Sci. Comput.},
  FJOURNAL = {SIAM Journal on Scientific Computing},
    VOLUME = {43},
      YEAR = {2021},
    NUMBER = {6},
     PAGES = {A3809--A3836},
      ISSN = {1064-8275,1095-7197},
   MRCLASS = {65M60 (35Q30 65M12 65M15 76D05)},
  MRNUMBER = {4336329},
MRREVIEWER = {Mustafa\ Aggul},
       DOI = {10.1137/21M1394709},
       URL = {https://doi.org/10.1137/21M1394709},
}

@article {HanYongbin2022ima,
    AUTHOR = {Han, Yongbin and Hou, Yanren},
     TITLE = {Semirobust analysis of an {$\rm H(div)$}-conforming {DG}
              method with semi-implicit time-marching for the evolutionary
              incompressible {N}avier-{S}tokes equations},
   JOURNAL = {IMA J. Numer. Anal.},
  FJOURNAL = {IMA Journal of Numerical Analysis},
    VOLUME = {42},
      YEAR = {2022},
    NUMBER = {2},
     PAGES = {1568--1597},
      ISSN = {0272-4979,1464-3642},
   MRCLASS = {65M60 (65M12 65M15 76D05)},
  MRNUMBER = {4410752},
}

@article {Bermejo2012sinum,
    AUTHOR = {Bermejo, R. and Gal\'{a}n del Sastre, P. and Saavedra, L.},
     TITLE = {A second order in time modified {L}agrange-{G}alerkin finite
              element method for the incompressible {N}avier-{S}tokes
              equations},
   JOURNAL = {SIAM J. Numer. Anal.},
  FJOURNAL = {SIAM Journal on Numerical Analysis},
    VOLUME = {50},
      YEAR = {2012},
    NUMBER = {6},
     PAGES = {3084--3109},
      ISSN = {0036-1429,1095-7170},
   MRCLASS = {65M60 (65M25 76D05 76M10)},
  MRNUMBER = {3022255},
MRREVIEWER = {Shawn\ W.\ Walker},
       DOI = {10.1137/11085548X},
       URL = {https://doi.org/10.1137/11085548X},
}

@article {Archilla2023ima,
    AUTHOR = {Garc\'{\i}a-Archilla, Bosco and Novo, Julia},
     TITLE = {Robust error bounds for the {N}avier-{S}tokes equations using
              implicit-explicit second-order {BDF} method with variable
              steps},
   JOURNAL = {IMA J. Numer. Anal.},
  FJOURNAL = {IMA Journal of Numerical Analysis},
    VOLUME = {43},
      YEAR = {2023},
    NUMBER = {5},
     PAGES = {2892--2933},
      ISSN = {0272-4979,1464-3642},
   MRCLASS = {65M25 (76D05)},
  MRNUMBER = {4648524},
MRREVIEWER = {Jian\ Li},
       DOI = {10.1093/imanum/drac058},
       URL = {https://doi.org/10.1093/imanum/drac058},
}

@article {Diegel2017,
    AUTHOR = {Diegel, Amanda E. and Wang, Cheng and Wang, Xiaoming and Wise,
              Steven M.},
     TITLE = {Convergence analysis and error estimates for a second order
              accurate finite element method for the
              {C}ahn-{H}illiard-{N}avier-{S}tokes system},
   JOURNAL = {Numer. Math.},
  FJOURNAL = {Numerische Mathematik},
    VOLUME = {137},
      YEAR = {2017},
    NUMBER = {3},
     PAGES = {495--534},
      ISSN = {0029-599X,0945-3245},
   MRCLASS = {65M60 (35K55 35Q35 65M12 65M15)},
  MRNUMBER = {3712284},
MRREVIEWER = {Daniele\ Antonio\ Di Pietro},
       DOI = {10.1007/s00211-017-0887-5},
       URL = {https://doi.org/10.1007/s00211-017-0887-5},
}

@article {Obbadi2025cmame,
    AUTHOR = {Obbadi, Anouar and El-Amrani, Mofdi and Seaid, Mohammed and
              Yakoubi, Driss},
     TITLE = {A stable second-order splitting method for incompressible
              {N}avier-{S}tokes equations using the scalar auxiliary
              variable approach},
   JOURNAL = {Comput. Methods Appl. Mech. Engrg.},
  FJOURNAL = {Computer Methods in Applied Mechanics and Engineering},
    VOLUME = {437},
      YEAR = {2025},
     PAGES = {Paper No. 117801, 24},
      ISSN = {0045-7825,1879-2138},
   MRCLASS = {76M20 (65M22 76D05)},
  MRNUMBER = {4859092},
       DOI = {10.1016/j.cma.2025.117801},
       URL = {https://doi.org/10.1016/j.cma.2025.117801},
}

@article{JiBingquan2025ima,
    author = {Ji, Bingquan and Gong, Yuezheng and Wang, Yushun and Zhao, Xuan},
    title = {Time-grid independent error analysis of adaptive predictor-corrector BDF2 scheme for the unsteady Navier–Stokes equations with high Reynolds number},
    journal = {IMA Journal of Numerical Analysis},
    pages = {draf094},
    year = {2025},
    month = {10},
}

@article {JiBingquan2024jsc,
    AUTHOR = {Ji, Bingquan and Liao, Hong-lin},
     TITLE = {A unified {$L^2$} norm error analysis of {SAV}-{BDF} schemes
              for the incompressible {N}avier-{S}tokes equations},
   JOURNAL = {J. Sci. Comput.},
  FJOURNAL = {Journal of Scientific Computing},
    VOLUME = {100},
      YEAR = {2024},
    NUMBER = {1},
     PAGES = {Paper No. 5, 25},
      ISSN = {0885-7474},
   MRCLASS = {65M15 (65M70 76D05)},
  MRNUMBER = {4748939},
       DOI = {10.1007/s10915-024-02555-9},
       URL = {https://doi.org/10.1007/s10915-024-02555-9},
}

@article {Akrivis2021sinum,
    AUTHOR = {Akrivis, Georgios and Chen, Minghua and Yu, Fan and Zhou, Zhi},
     TITLE = {The energy technique for the six-step {BDF} method},
   JOURNAL = {SIAM J. Numer. Anal.},
  FJOURNAL = {SIAM Journal on Numerical Analysis},
    VOLUME = {59},
      YEAR = {2021},
    NUMBER = {5},
     PAGES = {2449--2472},
      ISSN = {0036-1429},
   MRCLASS = {65M60 (65L06 65M12 65M20)},
  MRNUMBER = {4316580},
MRREVIEWER = {Hamdullah Y\"{u}cel},
       DOI = {10.1137/21M1392656},
       URL = {https://doi.org/10.1137/21M1392656},
}

@article {HeYinnian2008mc,
    AUTHOR = {He, Yinnian},
     TITLE = {The {E}uler implicit/explicit scheme for the 2{D}
              time-dependent {N}avier-{S}tokes equations with smooth or
              non-smooth initial data},
   JOURNAL = {Math. Comp.},
  FJOURNAL = {Mathematics of Computation},
    VOLUME = {77},
      YEAR = {2008},
    NUMBER = {264},
     PAGES = {2097--2124},
      ISSN = {0025-5718},
   MRCLASS = {65M60 (35Q30 76D05 76M25)},
  MRNUMBER = {2429876},
MRREVIEWER = {Lorenzo H\'{e}ctor Ju\'{a}rez},
       DOI = {10.1090/S0025-5718-08-02127-3},
       URL = {https://doi.org/10.1090/S0025-5718-08-02127-3},
}

@article {LiuJie2013sinum,
    AUTHOR = {Liu, Jie},
     TITLE = {Simple and efficient {ALE} methods with provable temporal
              accuracy up to fifth order for the {S}tokes equations on time
              varying domains},
   JOURNAL = {SIAM J. Numer. Anal.},
  FJOURNAL = {SIAM Journal on Numerical Analysis},
    VOLUME = {51},
      YEAR = {2013},
    NUMBER = {2},
     PAGES = {743--772},
      ISSN = {0036-1429},
   MRCLASS = {65M60 (65M12 76D05 76M10)},
  MRNUMBER = {3033031},
MRREVIEWER = {Alexander Ostermann},
       DOI = {10.1137/110825996},
       URL = {https://doi.org/10.1137/110825996},
}

@article {Ahmed2017cmame,
    AUTHOR = {Ahmed, Naveed and Becher, Simon and Matthies, Gunar},
     TITLE = {Higher-order discontinuous {G}alerkin time stepping and local
              projection stabilization techniques for the transient {S}tokes
              problem},
   JOURNAL = {Comput. Methods Appl. Mech. Engrg.},
  FJOURNAL = {Computer Methods in Applied Mechanics and Engineering},
    VOLUME = {313},
      YEAR = {2017},
     PAGES = {28--52},
      ISSN = {0045-7825},
   MRCLASS = {65M60 (65M12 65M15 76D07 76M10)},
  MRNUMBER = {3577930},
       DOI = {10.1016/j.cma.2016.09.026},
       URL = {https://doi.org/10.1016/j.cma.2016.09.026},
}

@book {Hairer1996ode2,
    AUTHOR = {Hairer, E. and Wanner, G.},
     TITLE = {Solving ordinary differential equations. {II}},
    SERIES = {Springer Series in Computational Mathematics},
    VOLUME = {14},
   EDITION = {Second},
      NOTE = {Stiff and differential-algebraic problems},
 PUBLISHER = {Springer-Verlag, Berlin},
      YEAR = {1996},
     PAGES = {xvi+614},
      ISBN = {3-540-60452-9},
   MRCLASS = {65-02 (34A09 34A45 65-01 65Lxx)},
  MRNUMBER = {1439506},
       DOI = {10.1007/978-3-642-05221-7},
       URL = {https://doi.org/10.1007/978-3-642-05221-7},
}

@book {Hairer1993ode1,
    AUTHOR = {Hairer, E. and N\o rsett, S. P. and Wanner, G.},
     TITLE = {Solving ordinary differential equations. {I}},
    SERIES = {Springer Series in Computational Mathematics},
    VOLUME = {8},
   EDITION = {Second},
      NOTE = {Nonstiff problems},
 PUBLISHER = {Springer-Verlag, Berlin},
      YEAR = {1993},
     PAGES = {xvi+528},
      ISBN = {3-540-56670-8},
   MRCLASS = {65-02 (34A50 65L99)},
  MRNUMBER = {1227985},
}

@article {Alessandro2025sinum,
    AUTHOR = {Contri, Alessandro and Kov\'{a}cs, Bal\'{a}zs and Massing, Andr\'{e}},
     TITLE = {Error analysis of {BDF} 1--6 time-stepping methods for the
              transient {S}tokes problem: velocity and pressure estimates},
   JOURNAL = {SIAM J. Numer. Anal.},
  FJOURNAL = {SIAM Journal on Numerical Analysis},
    VOLUME = {63},
      YEAR = {2025},
    NUMBER = {4},
     PAGES = {1586--1616},
      ISSN = {0036-1429},
   MRCLASS = {65M60 (65M12 76M10)},
  MRNUMBER = {4937622},
       DOI = {10.1137/23M1606800},
       URL = {https://doi.org/10.1137/23M1606800},
}

@article {LiXiaoli2023anm,
    AUTHOR = {Li, Xiaoli and Shen, Jie},
     TITLE = {Error estimate of a consistent splitting {GSAV} scheme for the
              {N}avier-{S}tokes equations},
   JOURNAL = {Appl. Numer. Math.},
  FJOURNAL = {Applied Numerical Mathematics. An IMACS Journal},
    VOLUME = {188},
      YEAR = {2023},
     PAGES = {62--74},
      ISSN = {0168-9274},
   MRCLASS = {65M12 (76D05)},
  MRNUMBER = {4563542},
       DOI = {10.1016/j.apnum.2023.03.004},
       URL = {https://doi.org/10.1016/j.apnum.2023.03.004},
}

@article {LiBuyang2022sinum,
    AUTHOR = {Li, Buyang and Ma, Shu and Schratz, Katharina},
     TITLE = {A semi-implicit exponential low-regularity integrator for the
              {N}avier-{S}tokes equations},
   JOURNAL = {SIAM J. Numer. Anal.},
  FJOURNAL = {SIAM Journal on Numerical Analysis},
    VOLUME = {60},
      YEAR = {2022},
    NUMBER = {4},
     PAGES = {2273--2292},
      ISSN = {0036-1429},
   MRCLASS = {65M12 (65M15 76D05)},
  MRNUMBER = {4471050},
       DOI = {10.1137/21M1437007},
       URL = {https://doi.org/10.1137/21M1437007},
}

@article {AnRong2022cnsns,
    AUTHOR = {Li, Yuan and An, Rong},
     TITLE = {Temporal error analysis of a new {E}uler semi-implicit scheme
              for the incompressible {N}avier-{S}tokes equations with
              variable density},
   JOURNAL = {Commun. Nonlinear Sci. Numer. Simul.},
  FJOURNAL = {Communications in Nonlinear Science and Numerical Simulation},
    VOLUME = {109},
      YEAR = {2022},
     PAGES = {Paper No. 106330, 17},
      ISSN = {1007-5704},
   MRCLASS = {65M60 (65M15 76D05)},
  MRNUMBER = {4383081},
       DOI = {10.1016/j.cnsns.2022.106330},
       URL = {https://doi.org/10.1016/j.cnsns.2022.106330},
}

@article {Anderson2021,
    AUTHOR = {Anderson, Robert and Andrej, Julian and Barker, Andrew and et
              al.},
     TITLE = {M{FEM}: {A} modular finite element methods library},
   JOURNAL = {Comput. Math. Appl.},
  FJOURNAL = {Computers \& Mathematics with Applications. An International
              Journal},
    VOLUME = {81},
      YEAR = {2021},
     PAGES = {42--74},
      ISSN = {0898-1221},
   MRCLASS = {65N30 (65Y15)},
  MRNUMBER = {4189802},
}

@article {Bakermathcomp1982,
    AUTHOR = {Baker, Garth A. and Dougalis, Vassilios A. and Karakashian,
              Ohannes A.},
     TITLE = {On a higher order accurate fully discrete {G}alerkin
              approximation to the {N}avier-{S}tokes equations},
   JOURNAL = {Math. Comp.},
  FJOURNAL = {Mathematics of Computation},
    VOLUME = {39},
      YEAR = {1982},
    NUMBER = {160},
     PAGES = {339--375},
      ISSN = {0025-5718},
   MRCLASS = {65M60 (65N30 76D05)},
  MRNUMBER = {669634},
       DOI = {10.2307/2007319},
       URL = {https://doi.org/10.2307/2007319},
}

@article {WangXiaomingnm2012,
    AUTHOR = {Wang, Xiaoming},
     TITLE = {An efficient second order in time scheme for approximating
              long time statistical properties of the two dimensional
              {N}avier-{S}tokes equations},
   JOURNAL = {Numer. Math.},
  FJOURNAL = {Numerische Mathematik},
    VOLUME = {121},
      YEAR = {2012},
    NUMBER = {4},
     PAGES = {753--779},
      ISSN = {0029-599X},
   MRCLASS = {76M20 (65M06 76D06 76F20)},
  MRNUMBER = {2945615},
MRREVIEWER = {Yinnian He},
       DOI = {10.1007/s00211-012-0450-3},
       URL = {https://doi.org/10.1007/s00211-012-0450-3},
}

@article {Ingramint2013,
    AUTHOR = {Ingram, Ross},
     TITLE = {Unconditional convergence of high-order extrapolations of the
              {C}rank-{N}icolson, finite element method for the
              {N}avier-{S}tokes equations},
   JOURNAL = {Int. J. Numer. Anal. Model.},
  FJOURNAL = {International Journal of Numerical Analysis and Modeling},
    VOLUME = {10},
      YEAR = {2013},
    NUMBER = {2},
     PAGES = {257--297},
      ISSN = {1705-5105},
   MRCLASS = {65M60 (65M12 76D05)},
  MRNUMBER = {3055600},
MRREVIEWER = {Daniele Boffi},
}

@article {Hesinum2003,
    AUTHOR = {He, Yinnian},
     TITLE = {Two-level method based on finite element and
              {C}rank-{N}icolson extrapolation for the time-dependent
              {N}avier-{S}tokes equations},
   JOURNAL = {SIAM J. Numer. Anal.},
  FJOURNAL = {SIAM Journal on Numerical Analysis},
    VOLUME = {41},
      YEAR = {2003},
    NUMBER = {4},
     PAGES = {1263--1285},
      ISSN = {0036-1429},
   MRCLASS = {65M60 (35Q30 76D05 76M10)},
  MRNUMBER = {2034880},
MRREVIEWER = {Long An Ying},
       DOI = {10.1137/S0036142901385659},
       URL = {https://doi.org/10.1137/S0036142901385659},
}

@book {Temam1983,
    AUTHOR = {Temam, Roger},
     TITLE = {Navier-{S}tokes equations and nonlinear functional analysis},
    SERIES = {CBMS-NSF Regional Conference Series in Applied Mathematics},
    VOLUME = {41},
 PUBLISHER = {Society for Industrial and Applied Mathematics (SIAM),
              Philadelphia, PA},
      YEAR = {1983},
     PAGES = {xii+122},
      ISBN = {0-89871-183-5},
   MRCLASS = {35Q10 (46N05 58D25 65M60 65N30 76D05)},
  MRNUMBER = {764933},
MRREVIEWER = {Howard Swann},
}

@article {ChengKelong2016,
    AUTHOR = {Cheng, Kelong and Wang, Cheng},
     TITLE = {Long time stability of high order multistep numerical schemes
              for two-dimensional incompressible {N}avier-{S}tokes
              equations},
   JOURNAL = {SIAM J. Numer. Anal.},
  FJOURNAL = {SIAM Journal on Numerical Analysis},
    VOLUME = {54},
      YEAR = {2016},
    NUMBER = {5},
     PAGES = {3123--3144},
      ISSN = {0036-1429},
   MRCLASS = {65M70 (65M12 76D05 76M22)},
  MRNUMBER = {3564779},
MRREVIEWER = {Temur Jangveladze},
       DOI = {10.1137/16M1061588},
       URL = {https://doi.org/10.1137/16M1061588},
}

@article {HuangFukeng2021SIAM,
    AUTHOR = {Huang, Fukeng and Shen, Jie},
     TITLE = {Stability and error analysis of a class of high-order {IMEX}
              schemes for {N}avier-{S}tokes equations with periodic boundary
              conditions},
   JOURNAL = {SIAM J. Numer. Anal.},
  FJOURNAL = {SIAM Journal on Numerical Analysis},
    VOLUME = {59},
      YEAR = {2021},
    NUMBER = {6},
     PAGES = {2926--2954},
      ISSN = {0036-1429},
   MRCLASS = {65M70 (65M12 65M15 76D05)},
  MRNUMBER = {4342123},
MRREVIEWER = {Jean-Pierre Croisille},
       DOI = {10.1137/21M1404144},
       URL = {https://doi.org/10.1137/21M1404144},
}

@article {Nevanlinna1981,
    AUTHOR = {Nevanlinna, Olavi and Odeh, F.},
     TITLE = {Multiplier techniques for linear multistep methods},
   JOURNAL = {Numer. Funct. Anal. Optim.},
  FJOURNAL = {Numerical Functional Analysis and Optimization. An
              International Journal},
    VOLUME = {3},
      YEAR = {1981},
    NUMBER = {4},
     PAGES = {377--423},
      ISSN = {0163-0563},
   MRCLASS = {65L05},
  MRNUMBER = {636736},
MRREVIEWER = {Peter Alfeld},
       DOI = {10.1080/01630568108816097},
       URL = {https://doi.org/10.1080/01630568108816097},
}

@Article{Hiptmair2002,
  author     = {Hiptmair, R.},
  title      = {Finite elements in computational electromagnetism},
  journal    = {Acta Numer.},
  year       = {2002},
  volume     = {11},
  pages      = {237--339},
  issn       = {0962-4929},
  doi        = {10.1017/S0962492902000041},
  fjournal   = {Acta Numerica},
  mrclass    = {78M10 (65N30)},
  mrnumber   = {2009375},
  mrreviewer = {JiChun Li},
  url        = {https://doi.org/10.1017/S0962492902000041},
}

@Article{Heywood1990,
  author   = {Heywood, John G. and Rannacher, Rolf},
  title    = {Finite-element approximation of the nonstationary {N}avier-{S}tokes problem. {IV}. {E}rror analysis for second-order time discretization},
  journal  = {SIAM J. Numer. Anal.},
  year     = {1990},
  volume   = {27},
  number   = {2},
  pages    = {353--384},
  issn     = {0036-1429},
  fjournal = {SIAM Journal on Numerical Analysis},
  mrclass  = {65N30 (76D05 76M10)},
  mrnumber = {1043610},
}

@Article{Heywood1982,
  author   = {Heywood, J. G. and Rannacher, Rolf},
  title    = {Finite element approximation of the nonstationary {N}avier-{S}tokes problem. {I}. {R}egularity of solutions and second-order error estimates for spatial discretization},
  journal  = {SIAM J. Numer. Anal.},
  year     = {1982},
  volume   = {19},
  number   = {2},
  pages    = {275--311},
  issn     = {0036-1429},
  doi      = {10.1137/0719018},
  fjournal = {SIAM Journal on Numerical Analysis},
  mrclass  = {65M60 (35Q10 76D05)},
  mrnumber = {650052},
  url      = {https://doi.org/10.1137/0719018},
}

@Book{Girault1986,
  title      = {Finite element methods for {N}avier-{S}tokes equations},
  publisher  = {Springer-Verlag, Berlin},
  year       = {1986},
  author     = {Girault, Vivette and Raviart, Pierre-Arnaud},
  volume     = {5},
  series     = {Springer Series in Computational Mathematics},
  isbn       = {3-540-15796-4},
  note       = {Theory and algorithms},
  doi        = {10.1007/978-3-642-61623-5},
  mrclass    = {65N30 (65-02 76-08)},
  mrnumber   = {851383},
  mrreviewer = {Max D. Gunzburger},
  pages      = {x+374},
  url        = {https://doi.org/10.1007/978-3-642-61623-5},
}

@Book{Brezzi1991,
  title      = {Mixed and hybrid finite element methods},
  publisher  = {Springer-Verlag, New York},
  year       = {1991},
  author     = {Brezzi, Franco and Fortin, Michel},
  volume     = {15},
  series     = {Springer Series in Computational Mathematics},
  isbn       = {0-387-97582-9},
  doi        = {10.1007/978-1-4612-3172-1},
  mrclass    = {65N30 (65-02 73V05 76M10)},
  mrnumber   = {1115205},
  mrreviewer = {Lubor Malina},
  pages      = {x+350},
  url        = {https://doi.org/10.1007/978-1-4612-3172-1},
}

@Article{AitOuAmmi1994,
  author     = {Ait Ou Ammi, Ali and Marion, Martine},
  title      = {Nonlinear {G}alerkin methods and mixed finite elements: two-grid algorithms for the {N}avier-{S}tokes equations},
  journal    = {Numer. Math.},
  year       = {1994},
  volume     = {68},
  number     = {2},
  pages      = {189--213},
  issn       = {0029-599X},
  fjournal   = {Numerische Mathematik},
  mrclass    = {65N30 (35Q30 65M60 76D05 76M10)},
  mrnumber   = {1283337},
  mrreviewer = {Wolfgang Moldenhauer},
}

\end{document}